\documentclass[11pt]{article}
\usepackage{amssymb, amsthm, amsmath, amscd}
\newtheorem{thm}{Theorem}[section]
\newtheorem{lem}[thm]{Lemma}
\newtheorem{prop}[thm]{Proposition}
\newtheorem{cor}[thm]{Corollary}
\newtheorem{NN}[thm]{}
\theoremstyle{definition}\newtheorem{df}[thm]{Definition}
\theoremstyle{definition}\newtheorem{rem}[thm]{Remark}
\theoremstyle{definition}

\renewcommand{\phi}{\varphi}

\newcommand{\N}{\mathbb{N}}
\newcommand{\Z}{\mathbb{Z}}

\newcommand{\R}{\mathbb{R}}
\newcommand{\C}{\mathbb{C}}
\newcommand{\T}{\mathbb{T}}

\newcommand{\Aff}{\operatorname{Aff}}

\newcommand{\id}{\operatorname{id}}

\newcommand{\morp}{contractive completely positive linear map}

\newcommand{\hm}{homomorphism}
\newcommand{\dt}{\delta}
\newcommand{\ep}{\epsilon}

\newcommand{\andeqn}{\,\,\,{\rm and}\,\,\,}
\newcommand{\rforal}{\,\,\,{\rm for\,\,\,all}\,\,\,}
\newcommand{\CA}{$C^*$-algebra}
\newcommand{\SCA}{$C^*$-subalgebra}

\newcommand{\af}{{\alpha}}
\newcommand{\bt}{{\beta}}

\newcommand{\beq}{\begin{eqnarray}}
\newcommand{\eneq}{\end{eqnarray}}
\newcommand{\tforal}{\,\,\,\text{for\,\,\,all}\,\,\,}
\newcommand{\tand}{\,\,\,\text{and}\,\,\,}

\title{Homomorphisms from AH-algebras}
\author{Huaxin Lin
 }
\date{}

\begin{document}

\maketitle

\begin{abstract}
Let $C$ be a general unital AH-algebra and let $A$ be a unital
simple $C^*$-algebra with tracial  rank at most one. Suppose that $\phi, \psi: C\to A$ are two unital monomorphisms. We show that $\phi$ and $\psi$ are approximately unitarily equivalent if and only if
\beq\nonumber
[\phi]&=&[\psi]\,\,\,{\rm in}\,\,\, KL(C,A),\\\nonumber
\phi_{\sharp}&=&\psi_{\sharp}\tand\\
\phi_{\tiny{\rho}}&=&\psi_{\rho},
\eneq
where $\phi_{\sharp}$ and $\psi_{\sharp}$ are continuous affine
maps from tracial state space $T(A)$ of $A$ to faithful tracial state space $T_{\rm f}(C)$ of $C$ induced by $\phi$ and $\psi,$ respectively, and $\phi_{\rho}$ and
$\psi_{\rho}$ are induced \hm s from $K_1(C)$ into $\Aff(T(A))/\overline{\rho_A(K_0(A))},$ where $\Aff(T(A))$ is the space of all real affine continuous functions on $T(A)$ and
$\overline{\rho_A(K_0(A))}$ is the closure of the image of $K_0(A)$ in
the affine space $\Aff(T(A)).$  In particular, the above holds
for $C=C(X),$ the algebra of continuous functions on a compact metric space.
  An approximate version
 of this is also obtained. We also show that, given a triple of
compatible elements $\kappa\in KL_e(C,A)^{++},$ an affine map
$\gamma: T(C)\to T_{\rm f}(C)$ and a \hm\, $\af: K_1(C)\to \Aff(T(A))/\overline{\rho_A(K_0(A))},$ there exists a unital monomorphism $\phi: C\to A$ such that
$[h]=\kappa,$ $h_{\sharp}=\gamma$ and $\phi_{\rho}=\af.$
\end{abstract}

\section{Introduction}

Let $X$ be a compact metric space and let $A$ be a unital simple \CA.
Let $\phi, \, \psi: C(X)\to A$ be two \hm s.
We study the problem
when these two maps from $C(X),$ the commutative \CA\, of continuous functions on $X,$ into $A$ are approximately unitarily equivalent, i.e., when there exists a sequence of unitaries $\{u_n\}\subset A$ such that
$$
\lim_{n\to\infty} u_n^*\psi(f)u_n=\phi(f)\tforal f\in C(X).
$$
In the case that $X$ is a compact subset of the plane and $A$ is the $n\times n$ matrix algebra,
two such maps are unitarily equivalent if and only if the corresponding normal matrices have the same set of eigenvalues
(counting multiplicity).   Brown-Douglass-Fillmore's study of essentially normal operators led
to the following theorem: Two unital monomorphisms from $C(X)$ (when $X$ is a compact subset of the plane) into the Calkin algebra are unitarily equivalent if and only if they induce the same \hm\, from $K_1(C(X))$ into $\Z. $  It should be noted that
both the $n\times n$ matrix algebra and the Calkin algebra are unital simple \CA s of real rank zero.

Unital separable commutative \CA s are of the form $C(X)$ for some compact metric space by the Gelfand transformation. Therefore
the study of \CA s may be viewed as the study of non-commutative topology. As in the topology, one studies
continuous maps between spaces, in \CA\, theory, one studies the \hm s
from one \CA\, to another. In this point of view, the study of \hm s from one \CA \, to another is one of the  fundamental problems in the \CA\, theory. At the present paper, we assume that the target algebra is a unital simple \CA, which conforms to the previous two mentioned cases. Simple \CA s may also be viewed as the opposite end of commutative \CA s.  For the source algebra, we begin with the case that it is the commutative \CA\, following the two above mentioned cases. However, we will study the case that source algebras are general unital AH-algebras (They are not necessarily simple, nor of slow dimension growth).

Let $\phi,\, \psi: C(X)\to A$ be two unital \hm s and  let $I={\rm ker}\phi.$ Then $I={\rm ker}\psi,$ if $\phi$ and $\psi$ are
approximately unitarily equivalent. Therefore, one may study the induced \hm s from $C(X)/I$ instead.
Note that $C(X)/I$ is isomorphic to $C(Y)$ for some compact subset of $X.$ To simplify the matter,
we will only study monomorphisms.
The problem has been studied (for some earlier results, for example, see \cite{Lnpac96} and \cite{Lncan97} ). Dadarlat
(\cite{D1}) showed that, if $C=C(X)$ and
$A$ is a unital purely infinite simple \CA\, (such as the Calkin algebra),
then two unital monomorphisms from $C$ into $A$ are approximately unitarily equivalent if and only if they induce the same element in $KL(C,A).$     When the target \CA s are finite,  other invariants such as traces have to be considered.
When $A$ is a unital simple \CA\, with stable rank one, real rank zero, weakly unperforated $K_0(A)$ and a unique tracial state, it is shown in \cite{GL2} that $\phi$ and $\psi$ are approximately unitarily equivalent if and only if
$[\phi]=[\psi]$ in $KL(C(X), A)$ and $\tau\circ \phi=\tau\circ \psi.$  When the real rank of $A$ is not zero  one needs
additional data to determine when $\phi$ and $\psi$ are approximately unitarily equivalent.
In fact, it is shown (\cite{LnApp}) that when $C$ is a some special unital AH-algebra and $A$ is a unital simple
\CA\, with tracial rank at most one, two unital monomorphisms $\phi, \psi: C\to A$ are approximately unitarily equivalent if and only $[\phi]=[\psi],$ $\phi_{\sharp}=\psi_{\sharp}$ and $\phi^{\ddag}=\psi^{\ddag},$    where
$\phi_{\sharp}$ and $\phi^{\ddag}$ will be defined below ((\ref{d1}) and (\ref{Harp-Scand})).
The  technical condition imposed on $C(X)$ is basically said that, $K$-theoretically speaking, $C(X)$ has a lower rank.
In this paper this  restriction
on $AH$-algebras has been removed.  A complete criterion is given for two unital monomorphisms from  a general $AH$-algebra into a unital simple \CA\, with tracial rank at most one being approximately unitarily equivalent.


 One of the long standing problems in the classification theory is to classify locally AH-algebra with no dimension growth.
 The problem could be solved if one could establish  a version of Gong's  decomposition theorem  which
 allows maps that are not exactly \hm s.  Over more than a decade, since
 the proof of Gong's decomposition theorem first appeared, the technical difficulty to generalize it to include almost multiplicative maps had remained elusive.  This author's many attempts failed during these years. It is the desire to prove that unital simple locally AH-algebras
with no dimension growth can also be classified by their Elliott invariant drew author's attention again to Gong's decomposition
 theorem. One application of the results in this paper will be the proof
 that unital simple locally AH-algebras with slow dimension growth are classifiable by the Elliott invariant (\cite{LnLAH}).

 We also believe that the main results presented here have their own independent interest as discussed at the beginning of this introduction.

  The paper is organized as follows:
 Section 2 serves largely as preliminaries for the whole paper.   In Section 3, we prove Theorem \ref{NMT1}
 which is the main technical advance of this paper.  In Section 4, we collect  a number of  miscellaneous lemmas which will be used in the proof
 of the main results.  In Section 5, we prove the  main results.
 To complete our results and make application possible, in Section 6, we provide the description
 of the range of approximate unitary equivalence classes of unital monomorphisms from a unital AH-algebra to a unital simple \CA\, of tracial rank at most one.
 Applications to the study of tracial rank and classification
 of unital simple locally AH-algebras will appear elsewhere (\cite{LnLAH}).

{\bf Acknowledgment}
Most of this work was done when the author was  in East China Normal University in the summer 2010. This work is partially supported by
University of Oregon, East China Normal University and
a NSF grant. The author would like to thank Claude Schochet who pointed to us there was an awkward computation
in an earlier version of this paper which was now avoided.

\section{Preliminaries}

\begin{NN}\label{d1}
{\rm Let $A$ be a unital \CA. Denote by $T(A)$ the convex set of tracial states of $C.$
Denote by $T_{\rm f}(A)$ the convex set of all faithful tracial states.
Let $\Aff(T(A))$ be the space of all real affine continuous functions on $T(A).$
Denote by $M_n(A)$ the matrixes over $A.$  By regarding $M_n(A)$ as a subset of $M_{n+1}(A),$ define $M_{\infty}(A)=\cup_{n=1}^{\infty}M_n(A).$  If $\tau\in T(A),$ then $\tau\otimes {\rm Tr},$ where ${\rm Tr}$ is standard
trace on $M_n,$ is a trace on $M_n(A).$  Throughout this paper, we may use $\tau$ for  $\tau\otimes {\rm Tr}$ without
warning.

If $B$ is another \CA\, and $\phi: A\to B$ is a  \morp, then
$\phi\otimes {\id}_{M_n}$ gives a \morp\, from $M_n(A)$ to $M_n(B).$ Throughout this paper, we may use
$\phi$ for $\phi\otimes {\rm id}_{M_n}$ for convenience.

Let  $C$ and $A$ be  two unital \CA s with $T(C)\not=\emptyset$ and $T(A)\not=\emptyset.$
Suppose that $h: C\to A$ is a unital \hm. Define an affine continuous map $h_{\sharp}: T(A)\to T(C)$ by
$h_{\sharp}(\tau)(c)=\tau\circ h(c)$ for all $\tau\in T(A)$ and $c\in C.$ If $A$ is simple and $h$ is a monomorphism,
then $h_{\sharp}$ maps $T(A)$ into $T_{\rm f}(C).$
}

\end{NN}

\begin{df}\label{rho}
Let $C$ be a unital \CA\, with $T(C)\not=\emptyset.$
For each $p\in M_n(C)$ define $\check{p}(\tau)=\tau\otimes {\rm Tr}(p)$ for all $\tau\in T(A),$ where ${\rm Tr}$ is the standard trace on $M_n.$ This gives  positive \hm\, $\rho_C: K_0(C)\to \Aff(T(C)).$
\end{df}

\begin{NN}
{\rm If $A$ is a unital  \CA\, with tracial rank at most one, then we will write $TR(A)\le 1$ (see \cite{Lntr}).}
\end{NN}

\begin{NN}\label{Harp-Scand}
{\rm Let $C$ be a unital \CA. Denote by $U(C)$ the unitary group of $C$ and denote
by $U_0(C)$ the subgroup of $U(C)$ consisting of unitaries which connected to $1_C$ by a continuous path of unitaries.  Denote by
$CU(C)$ be the closure of the normal subgroup generated by commutators of $U_0(C).$ Let $u\in U(C).$
Then ${\bar u}$ is the image of $u$ in $U(C)/CU(C).$


Now suppose that $T(C)\not=\emptyset.$
Let $n\ge 1$ be an integer. Let $u\in U_0(M_n(C)).$ Let $\gamma\in C([0,1], U_0(M_n(C)))$
which is piecewise smooth such that
$\gamma(1)=u$ and $\gamma(0)=1.$ Define
\beq\label{HS-1}
\Delta(\gamma)(\tau)={1\over{2\pi i}}\int_0^1 \tau({d\gamma(t)\over{dt}}\gamma^{-1}(t))dt,
\eneq
where $\tau$ is identified with $\tau\otimes {\rm Tr}$ (in particular,   for $n>1,$ $\tau$ in the above
formula is not the normalized trace).   As in (\cite{Tm}), the de la Harpe-Scandalis determinant
provides a continuous \hm\,
\beq\label{HS-2}
{\bar \Delta}: \cup_{k=1}^{\infty}(U_0(M_k(C))/U_0(M_k(C))\cap  CU_0(M_k(C))\to \Aff(T(C))/\overline{\rho_C(K_0(C))}.
\eneq

We will use $CU(M_{\infty}(C))$ for $\cup_{k=1}^{\infty} CU(M_k(C)).$

Define a metric as follows. If $u, v\in U(M_n(C))$ such that $uv^*\in U_0(M_n(C)),$ define
\beq\label{HS-3}
{\rm dist}({\bar u}, {\bar v})=\|{\bar \Delta}(uv^*)\|.
\eneq
Note that if $u, v\in U_0(M_n(C)),$ then
\beq\nonumber
{\rm dist}({\bar u}, {\bar v})=\|{\bar \Delta}(u)-{\bar \Delta}(v)\|
\eneq
(where the norm is the quotient norm in $\Aff(T(C))/\overline{\rho_C(K_0(C))}$).

Note that if $u\in CU(C),$  then $[u]=0$ in $K_1(C).$ Using de la Harpe-Scandalis determinant, by K. Thomsen (see \cite{Tm}),
one has the following short splitting exact sequence:
\beq\label{HS-4}
0\to \Aff(T(C))/\overline{\rho_C(K_0(C))}\to  U(M_{\infty}(C))/CU(M_{\infty}(C))
\to_{\pi_C} K_1(C)\to 0.
\eneq
We will fix one splitting map $J_C: K_1(C)\to  U(M_{\infty}(C))/CU(M_{\infty}(C))$  such that
$\pi_C\circ J_C={\rm id}_{K_1(C)}.$ For each ${\bar u}\in J_C(K_1(C)),$ select and fix one element $u_c\in \cup_{n=1}^{\infty}M_n(C)$ such that
${\overline{u_c}}={\bar u}.$ Denote this set by $U_c(K_1(C)).$

If $A$ is a unital \CA\, and $\phi: C\to A$ is a unital \hm, then $\phi$ induces a continuous map
$$\phi^{\ddag}:
U(M_{\infty}(C))/CU(M_{\infty}(C))\to  U(M_{\infty}(A))/CU(M_{\infty}(A)).
$$
Denote by $\phi^{\dag}: K_1(C)\to \Aff(T(A))/{\overline{\rho_A(K_0(A))}}$ the map $({\rm id}-J_A\circ \pi_A)\circ \phi^{\ddag}\circ J_C,$
where $J_A: K_1(A)\to U(M_{\infty}(A))/CU(M_{\infty}(A)$ is a fixed splitting map. Note that
$\phi^{\dag}$ depends on the splitting maps, but for each \CA,  splitting maps are fixed.

If  $A$ is a unital simple \CA\, with $TR(A)\le 1,$  by Cor. 3.5 of \cite{Lnhomtr1},
$$
U_0(C)/CU(C)=U_0(M_n(C))/CU(M_n(C))
$$
for all $n\ge 1.$
Since in this case, $K_1(C)=U(C)/U_0(C),$ we have isomorphism between
$U(C)/CU(C)$ and $U(M_{\infty}(C))/CU(M_{\infty}(C)).$
}

\end{NN}

\begin{NN}
{\rm Let $A$ be a unital \CA\, and let $u\in U_0(A).$
Let $\gamma\in C([0,1], U(A))$ such that $\gamma(0)=1$ and $\gamma(1)=u.$ Denote by
${\rm Length}(\{\gamma\})$ the length of the path $\gamma.$
Put
$$
{\rm cel}(u)=\inf\{{\rm Length}\gamma: \gamma\in C([0,1], U(A)), \gamma(0)=1\andeqn \gamma(1)=u\}.
$$
}
\end{NN}

\begin{df}\label{KKdef}
Let $C$ be a \CA\, and let ${\cal P}\subset \underline{K}(C).$
There exists $\dt>0$ and a finite subset ${\cal G}\subset C$ such that, for any $\dt$-${\cal G}$-multiplicative
\morp\, $L: C\to A$ (for any \CA\, $A$),
$[L]|_{\cal P}$ is well defined (see 0.6 of \cite{LnAMST99} and 2.3 of \cite{Lnmem} ). Such a triple $(\dt, {\cal G},{\cal P})$ is called
local $\underline{K}$-triple (see \cite{DE}). If $K_i(C)$ is finitely generated ($i=0,1$) and ${\cal P}$ is large enough,
then $[L]|_{\cal P}$ defines an element in $KK(C, A)$ (see 2.4 of \cite{Lnmem}).  In such cases, we will write
$[L]$ instead of $[L]|_{\cal P},$ and we will call $(\dt, {\cal G}, {\cal P})$  a $KK$-triple and $(\dt, {\cal G})$ a $KK$-pair.
Note that, if $u$ is a unitary then, we write
$\langle L(u)\rangle =L(u)(L(u)^*L(u))^{-1/2}$ when $\|L(u^*)L(u)-1\|<1$ and $\|L(u)L(u^*)-1\|<1.$  In what follows we will always assume
that $\|L(u^*)L(u)-1\|<1$ and $\|L(u)L(u^*)-1\|<1,$ when we
write $\langle L(u) \rangle.$

Suppose that $C$ is a separable \CA\, and $C$ is the closure of
$\cup_{n=1}^{\infty} C_n,$ where each $C_n=\lim_{m\to\infty} (C_{n,m}, \phi_m^{(n)})$ and $K_i(C_{n,m})$ is finitely generated ($i=0,1$). Denote by $\imath_n: C_n\to C$ the embedding and $\phi^{(n)}_{m, \infty}: C_{n,m}\to C_n$ the \hm\, induced by the inductive system $(C_{n,m},\phi_m).$
We say that $(\dt, {\cal G}, {\cal P})$ is a $KL$-triple, if
$[\imath_n\circ \phi^{(n)}_{m,\infty}]({\cal P'})\supset {\cal P}$ for some $n,\,m$ and some finite subset ${\cal P}'\subset \underline{K}(C_{n,m})$ and
if any
$\dt$-${\cal G}$-multiplicative \morp\, $L: C\to A$ (for any $A$) gives a $\dt$-${\cal G'}$-multiplicative \morp\, $L\circ \imath_n\circ\phi^{(n)}_{m,\infty}$ so that
$(\dt, {\cal G}', {\cal P}')$ is a $KK$-triple.
\end{df}

\begin{prop}\label{amlift}
Let $A$ be a (unital) separable amenable \CA\, and $\{I_n\}$ be a sequence of increasing closed two sided ideals of $A.$
Let $I=\overline{\cup_{n=1}^{\infty} I_n}$ and $B=A/I.$ Suppose that ${\cal F}\subset B$ is a finite subset and $\ep>0.$
Then there exists an integer $n\ge 1$ and a (unital) \morp\, $L: B\to A/I_n$ which is $\ep$-${\cal F}$-multiplicative
such that $\pi_n\circ L(b)=b$ for all $b\in B,$ where  $\pi_n: A/I_n\to A/I$ is the quotient map.
\end{prop}

\begin{proof}
By the Efros-Choi Lifting Theorem, there is a (unital) \morp\, $L_0: B\to A$ such that
$\pi\circ L_0={\rm id}_B,$ where $\pi: A\to A/I$ is the quotient map.
For each pair $a, b\in B,$ let $c(a,b)=L_0(ab)-L_0(a)L_0(b).$ Then
$$
\pi(c(a,b))=\pi(L_0(ab))-\pi(L_0(a)L_0(b))=0.
$$
It follows that $c(a,b)\in I.$ Therefore there exists an integer $n\ge 1$ such that
$$
\|c(a, b)-I_n\|<\ep\tforal a, b\in {\cal F}.
$$
Let $\pi_n: A\to A/I_n$ be the quotient map and $L=\pi_n\circ L_0: B\to A/I_n.$
Then
$$
\|L(ab)-L(a)L(b)\|=\|\pi_n(c(a,b)\|<\ep\tforal a, b\in {\cal F}.
$$
\end{proof}

\begin{df}\label{or}
{\rm Let $X$ be a compact metric space, let $x\in X$ and let $r>0.$ Denote by $O(x, r)$ the open ball with center at $x$ and radius $r.$
If $x$ is not specified $O(r)$ is an open ball of radius $r.$}

\end{df}

\begin{df}\label{mu}
Let $X$ be a metric space and $s: C(X)\to \C$ be a state.
Denote by $\mu_s$ the probability Borel measure induced by $s.$
\end{df}

The following could be proved directly but also follows from 4.6 of \cite{Lncd}.

\begin{thm}\label{Uniqq}
Let $X$ be a compact metric space, let $\ep>0$ and let ${\cal F}\subset C(X)$ be a finite
subset. There exists $\eta>0$ satisfying the  following: for any $\sigma>0,$ there exists
$\gamma>0,$ $\dt>0,$ a  finite subset ${\cal G}\subset C(X)$  and a finite
subset ${\cal H}\subset C(X)_{s.a}$ and a finite
subset ${\cal P}\subset \underline{K}(C(X))$ satisfying the following:

For any unital $\dt$-${\cal G}$-multiplicative \morp s $\phi, \psi: C(X)\to M_n$
(for some integer $n\ge 1$) for which
\beq\label{Uniqq-1}
[\phi]|_{\cal P}=[\psi]|_{\cal P},\,\,\,
\mu_{\tau\circ\phi}(O_r)\ge \sigma
\eneq
for all open balls $O_r$ of radius $r\ge \eta$ and
\beq\label{Uniqq-2}
|\tau\circ \phi(a)-\tau\circ \psi(a)|<\gamma\tforal a\in {\cal H},
\eneq
where $\tau$ is the tracial state of $M_n,$
there is a unitary $u\in M_n$ such that
\beq\label{Uniqq-3}
\|\phi(f)-{\rm Ad}\, u\circ \psi(f)\|<\ep
\tforal f\in {\cal F}.
\eneq
\end{thm}

The following is also known and could be derived from the above.

\begin{cor}\label{Add1109-1}
Let $X$ be a compact metric space, let $\ep>0$ and let ${\cal F}\subset C(X)$ be a finite subset, there exists $\eta>0$ satisfying the following:
For any $\sigma>0,$ there exists $\gamma>0,$ a finite subset ${\cal P}\subset K_0(C(X))$ and a finite subset ${\cal H}\subset C(X)_{s.a.}$ satisfying the following:
For any two unital \hm s $\phi, \psi: C(X)\to M_n$ (for some integer $n\ge 1$)  such that
\beq\label{Add1-1}
[\phi]|_{\cal P}=[\psi]|_{\cal P},\,\,\, \mu_{\tau\circ \phi}(O_r)\ge \sigma
\eneq
for all open ball $O_r$ of $X$ with radius $r\ge \eta$ and
\beq\label{Add1-2}
|\tau\circ \phi(a)-\tau\circ \psi(a)|<\gamma\tforal a\in {\cal H},
\eneq
there exists a unitary $u\in M_n$ such that
\beq\label{Add1-3}
\|{\rm Ad}\, u\circ \phi(f)-\psi(f)\|<\ep\tforal f\in {\cal F}.
\eneq
\end{cor}

The following is a variation of Lemma 4.3 of \cite{LnApp}.
\begin{cor}\label{Lijram}
Let $X$ be a compact metric space, $\ep>0$ and ${\cal F}\subset C(X)$ be a finite
subset. There exists $\eta_1>0$ satisfying the  following:
for any $\sigma_1>0$ and any $0<\lambda<1,$ there exists $\eta_2>0$ satisfying the following:
for any $\sigma_2>0,$ there exists
 $\dt>0,$  a finite subset ${\cal G}\subset C(X)$ and a finite
subset ${\cal P}\subset \underline{K}(C(X))$ satisfying the following:

For any unital $\dt$-${\cal G}$-multiplicative \morp\,  $\phi: C(X)\to M_n$
(for some integer $n\ge 1$)  such that
\beq\label{Li-2}
[\phi]|_{\cal P}=[H]|_{\cal P}
\eneq
for some unital \hm\, $H: C(X)\to M_n$ and
such that
\beq\label{Li-1}
\mu_{\tau\circ\phi}(O_r)\ge \sigma_1\tand \mu_{\tau\circ \phi}(O_r)\ge \sigma_2
\eneq
for all open balls $O_r$ of radius $r\ge \eta_1$ and
$r\ge \eta_2,$ respectively,
there is a unital \hm\, $h: C(X)\to  M_n$ such that
\beq\label{Li-3}
\|\phi(f)-h(f)\|<\ep
\tforal f\in {\cal F}.
\eneq
 Moreover,
\beq\label{Li-4}
\mu_{tr\circ h}(O_r)\ge \lambda\sigma_1
\eneq
for all $r\ge 2\eta_1.$
\end{cor}




\begin{NN}\label{prem0}
{\rm
Let $C$ be a unital \CA\, and let ${\cal P}\subset \underline{K}(C)$
be a finite subset.
There is a finite subset ${\cal F}_{C, {\cal P}, {\bf b}}\subset C$ and a positive number $\dt_{C, {\cal P}, \bf b}>0$ such that
${\rm  Bott}(u,\, h)|_{\cal P}$ (see the definition 2.10 of \cite{Lnmem}) is well defined for any unital \CA\, $A,$ any unital \hm\, $h: C\to A$ and any unitary $u\in A$ for which
\beq\label{prem0-1}
\|[h(f),\, u]\|<\dt_{C,{\cal P},  {\bf b}}\tforal f\in {\cal F}_{C, {\cal P}}.
\eneq
Moreover, by choosing even smaller $\dt_{C, {\cal P}, {\bf b}},$  if $h_1: C\to A$ is another unital \hm\, and
$$
\|h(f)-h_1(f)\|<\dt_{C, {\cal P}, \bf b},
$$
then
${\rm Bott}(u,\,h_1)|_{\cal P}$ is also well defined and
$$
{\rm Bott}(u,\,h)|_{\cal P}={\rm Bott}(u,\,h_1)|_{\cal P}.
$$
As in tradition,
$$
{\rm bott}_1(u,\,h)|_{\cal P}={\rm Bott}(u,\,h)|_{{\cal P}\cap K_1(C)}\andeqn {\rm bott}_0(u,\,h)|_{\cal P}={\rm Bott}(u\, h)|_{{\cal P}\cap K_0(C)}.
$$
If $K_i(C)$ ($i=0,1$) is finitely generated, then, by choosing  ${\cal P}$ large enough, we may assume
that, when (\ref{prem0-1}) holds, ${\rm Bott}(h, \, u)$ is well defined.  Furthermore, we will write $\dt_{C, \bf b}$ instead of $\dt_{C, {\cal P}, \bf b}$ and ${\cal F}_{C, \bf b}$ instead of ${\cal F}_{C, {\cal P}, \bf b}.$

If $C=C(\T),$ let $z\in U(C(\T))$ be the standard unitary generator, one writes that
$$
{\rm bott}_1(u,\,h)={\rm bott}(u,\,h(z)).
$$
Suppose that there is a continuous path of unitaries $u(t): [0,1]\to U_0(A)$ such that
\beq\label{prem0-2}
u(0)=u,\,\,\, u(1)=1_A\andeqn \|[h(f),\, u(t)]\|<\dt_{C, {\cal P}, {\bf b}}\rforal t\in [0,1],
\eneq
then
\beq\label{prem0-3}
{\rm Bott}(u,\,h)|_{\cal P}=0.
\eneq

Now suppose that $C$ is a unital separable amenable \CA\, which is the closure of $\cup_{n=1}^{\infty} C_n,$ where $C_n=\lim_{n\to\infty} (C_{n,m},\phi_m^{(n)})$  and $K_i(C_{n,m})$ is finitely generated
$(i=0,1$). Let $z$ be the standard unitary generator of $C(\T).$ We may view
${\cal P}$ as a subset of $\underline{K}(C\otimes C(\T)).$
Let ${\cal G}_0$ be a finite subset of $C.$ Define
${\cal G}_1=\{g\otimes f: g\in {\cal G}_0\andeqn f\in {\cal S}\},$ where $S=\{1, z, z^*\}.$
Let ${\cal P}_1={\cal P}\cup \boldsymbol{\bt}({\cal P})$ (see
2.10  of \cite{Lnmem} for the definition of
$\boldsymbol{\bt}$).
Let $\dt>0.$  Suppose that $(\dt, {\cal G}_1, {\cal P}_1)$ is a $KL$-triple for $C\otimes C(\T)$ (by selecting large ${\cal G}_0$ to begin with).

By choosing even smaller $\dt_{C, {\cal P}, {\bf b}},$ we may assume that, if there is a unitary $u\in A$ such that (\ref{prem0-1}) holds, and if there is a unital $\dt$-${\cal G}_1$-multiplicative \morp\, $L: C\otimes C(\T)\to A$ such that
\beq
\|L(f\otimes 1)-h(f)\|&<&\dt_{C,{\cal P}, {\bf b}}\tforal f\in {\cal F}_{C,{\cal P}, {\bf b}}\\
\andeqn
\|u-L(1\otimes z)\|&<&\dt_{C, {\cal P}, {\bf b}},
\eneq
then
$$
{\rm Bott}(u, h)|_{\cal P}=[L]|_{\boldsymbol{\bt}({\cal P})}.
$$


}
\end{NN}

The following is a restatement of Theorem 7.4 of \cite{Lnmem}.
\begin{thm}\label{Lmem}
Let $X$ be a compact metric space. For any $\ep>0$ and any finite subset ${\cal F}\subset C(X),$ there exists $\eta>0$
satisfying the following:
For any $\sigma>0,$ there exists $\dt>0,$ a finite subset ${\cal G}\subset C(X)$  and a finite subset
${\cal P}\subset \underline{K}(C(X))$ satisfying the following:
Suppose that $\phi: C(X)\to M_n$ is a unital \hm\, such that
$$
\mu_{tr\circ \phi}(O_r)\ge \sigma
$$
for all open ball $O_r$ with radius $r\ge \eta.$ If $u\in M_n$ is a unitary such that
$$
\|[u, \phi(g)]\|<\dt\tforal g\in {\cal G}\tand {\rm Bott}(h,\, u)|_{\cal P}=0,
$$
then there exists a continuous rectifiable path of unitaries
$\{u_t: t\in [0,1]\}$ of $M_n$ such that
$$
u_0=1, u_1=1_A\tand
\|[h(f),u_t]\|<\ep
$$
for all $f\in {\cal F}$ and $t\in [0,1].$ Moreover,
$$
{\rm Length}(\{u_t\})\le 2\pi+\ep\pi.
$$
\end{thm}

\section{ Almost multiplicative maps from $C(X)$ into interval algebras}

\begin{lem}\label{Bbot}
Let $X$ be a compact metric space.  For any $\ep>0,$ any finite subset ${\cal F}\subset C(X)$ and any finite subset
${\cal P}\subset \underline{K}(C(X)),$
there exists a finite subset $\{s_1, s_2,...,s_{k(X)}\}\subset K_1(C(X))$ and $\eta>0$
satisfying the following:
For  any
$1>\sigma>0,$ there exists $d>0,$ for  any
$\af\in KL(C(X)\otimes C(\T), \C)=Hom_{\Lambda}(\underline{K}(C(X)\otimes C(\T), \underline{K}(\C))$
and
for any unital \hm\, $\phi: C(X)\to M_n$ for some integer $n\ge 1$ for which
\beq\label{Bbot-1}
\mu_{tr\circ \phi}(O_r)\ge \sigma
\eneq
for any open balls $O_r$ with radius $r\ge \eta,$  where
$tr$ is the normalized trace on $M_n,$ and
\beq\label{Bot-1+}
\max\{|\af\circ \bt^{(1)}(s_i)|: 1\le i\le k(X)\}/n<d,
\eneq
there exists a unitary $u\in M_n$ such that
\beq\label{Bbot-3}
\|[\phi(f),\, u]\|<\ep\tforal f\in {\cal F}
\tand {\rm Bott}(\phi, \, u)|_{{\cal P}}=
\af|_{\boldsymbol{\bt}({\cal P})}.
\eneq
\end{lem}

(Note that, here $\bt^{(1)}: K_1(C(C(X))\to K_0(C(X)\otimes C(\T)))$ and ${\boldsymbol \bt}:
\underline{K}(C(X))\to \underline{K}(C(X)\otimes C(\T))$ are as defined in 2.10 of \cite{Lnmem}.)
\begin{proof}
Let $\ep>0,$  ${\cal F}\subset C(X)$ and ${\cal P}\subset \underline{K}(C(X))$  be  finite subsets.
Let $\ep_1=\min\{\ep/2, \dt_{C(X), {\cal P}, {\bf b}}\}$ and
let ${\cal F}_1={\cal F}\cup {\cal F}_{C(X), {\cal P}, {\bf b}}.$


Let
$\eta'>0$ (in place of $\eta$) be given by \ref{Uniqq} associated with
$\ep/16$ (in place of $\ep$) and ${\cal F}.$
Let $1>\sigma>0.$ Let $\gamma>0,$ $\dt>0,$ ${\cal G},$ ${\cal P}_1\subset \underline{K}(C(X))$ 
(in place of ${\cal P}$)  and ${\cal H}\subset C(X)$ be given by \ref{Uniqq} associated with the above $\ep/16$ (in place of $\ep$),
$\eta_1>0$ (in place of $\eta$) and $\sigma/2.$ We may assume that ${\cal P}\subset {\cal P}_1.$ 

For convenience, we may assume that ${\cal H}\cup {\cal F}\subset {\cal G}.$
We may assume that $\dt<\min\{\ep/2,1/4\},$  $\|g\|\le 1$ if $g\in {\cal G}$ and $1_{C(X)}\in {\cal G}.$

Let ${\cal G}_1=\{g\otimes f: g\in {\cal G}\andeqn f=1, z, z^*\}\subset
C(X)\otimes C(\T),$ where  $z$ is the identity
function on the unit circle.
We may also assume that $(\dt, {\cal G}_1, {\cal P}_1)$ is a $KL$-triple for $C(X)\otimes C(\T).$
Moreover, we may assume that $\dt<\dt_{C(X), {\cal P}, \bf{b}}$ and
${\cal G}\supset {\cal F}_{C(X),\bf{b}}.$

Suppose that $C(X)=\lim_{n\to\infty} C(Y_n),$ where each $Y_n$ is a finite CW complex.  Let $\imath_m: C(Y_m)\to C(X)$ be the unital \hm\, induced by the inductive limit system. We may assume that $\imath_m$ is injective (see \cite{Mar}).
Let $s: X\to Y_m$ be the surjective map such that
$\imath(f)(x)=f(s(x))$ for all $f\in C(Y_m)$ and $x\in X.$
By choosing a large $m,$ we may assume that there is a finite subset ${\cal G}'\subset C(Y_m)$ and there is a finite subset ${\cal P}'\subset \underline{K}(C(Y_m))$ such that $\imath_m({\cal G}')\supset {\cal G} $ and
$[\imath_m]({\cal P}')\supset {\cal P}.$
Let ${\cal G}'_1=\{g\otimes f: g\in {\cal G}'\andeqn f=1, z, z^*\}.$
We may further assume that $(2\dt, {\cal G}')$ is a $KK$-pair for
$C(Y_m)\otimes C(\T).$

Let $\{s_1', s_2',...,s_{k(X)}'\}\subset K_1(C(Y_m))$ be a generating subset.
Let $s_i=({\imath}_m)_{*1}(s_i),$ $i=1,2,...,k(X).$

Suppose that $Y_m$ is the disjoint union of finitely many connected CW complexes $Z_1, Z_2,...,Z_l.$ Without loss of generality, we may assume that there is, for each $i,$ a finite subset ${\cal G}^{(i)}\subset C(Z_i)$ such that $\oplus_{i=1}^l {\cal G}^{(i)}={\cal G}'$ and
there is a finite subset ${\cal P}_i'\subset \underline{K}(C(Z_i))$ such that $\oplus_{i=1}^l {\cal P}_i'={\cal P}'.$
Choose $\xi_i\in Z_i$ such that $\xi_i\in Y_m,$ $i=1,2,...,l.$ We assume also that
each $Z_i$ contains an open ball with radius $r\ge \eta_0>0.$ Put $\eta=\min\{\eta_0/2, \eta'\}.$ 

Let $N(\dt/4, {\cal G}_1^{(i)}, {\cal P}_i')$ be given by  Lemma 10.2 of \cite{LnApp}  for $C(\T\times Z_i).$
Define
$$
N(\dt/4, {\cal G}_1', {\cal P}')=\sum_{i=1}^l N(\dt/4, {\cal G}_1^{(i)}, {\cal P}_i').
$$

Let
$$
d=\min\{\sigma/2, \gamma\}\cdot {1\over{N(\dt/4, {\cal G}_1', {\cal P}')}}.
$$


Let $\af$ be as in the statement and let
$$
k=\max\{|\af\circ \bt^{(1)}(s_i)|: i=1,2,...,k(X)\}.
$$

Now suppose that $\phi: C(X)\to M_n$ for some integer $n\ge 1$
for which
\beq\label{Bbot-6}
\mu_{tr\circ \phi}(O_r)\ge \sigma
\eneq
for all open balls $O_r$ with radius $r\ge\eta,$ where
$tr$ is the normalized tracial state on $M_n$ and
$$
k/n<d.
$$
We may write
that
\beq\label{Bbot-7}
\phi(f)=\sum_{i=1}^n f(\zeta_i)p_i\tforal f\in C(X),
\eneq
where $\{p_1, p_2,...,p_n\}$ is a set of mutually orthogonal rank one
projections and $\zeta_i\in X,$ $i=1,2,...,n.$
Put $L=kN(\dt, {\cal G}_1', {\cal P}').$ Note $n>L.$

Note that, if  $x\in {\rm ker}\rho_{C(Y_m)},$ then $h_{*0}(x)=0$ for any \hm\, $h: C(Y_m)\to M_n$
(for any integer $n\ge 1$).
Define $H: C(Y_m)\to M_L$ by $H(f)=\sum_{i=1}^L f(s(\zeta_i))p_i$ for $f\in C(Y_m).$ 

Choose $\bt\in
Hom_{\Lambda}(\underline{K}(C(Y_m)\otimes C(\T)), \underline{K}(\C))$ defined as
\beq\label{Bbot4-1}
\bt|_{\underline{K}(C(Y_m))}=[H]
\eneq
and
\beq\label{Bbot4-2}
\bt|_{{\boldsymbol{\bt}}(\underline{K}(C(Y_m)))}=
\af\circ [\imath_m]|_{{\boldsymbol{\bt}}(\underline{K}(C(Y_m)))}
\eneq
(see 2.10  of \cite{Lnmem} for the definition of $\boldsymbol{\bt}$).
Let ${\cal G}_1=\{g\otimes 1: g\in {\cal G}\}\cup\{z_1\}\subset
C(Y_m)\otimes C(\T),$ where $z_1=1\otimes z$ and $z$ is the identity
function on the unit circle.

It follows from Lemma 10.2 of \cite{LnApp} (and considering each component of $Y_m$ separately)
 that there exists a unital $\dt/4$-${\cal G}_1'$-multiplicative
\morp\, $\Phi': C(Y_m)\otimes C(\T)\to M_L$
such that
\beq\label{Bbot-5}
[\Phi']=\beta,
\eneq
One obtains a
$\dt/2$-${\cal G}$-multiplicative \morp\, $\Phi: C(X)\otimes C(\T)\to M_L$ such that
\beq\label{Bot-5-1}
\|\Phi(\imath_m(g)\otimes f)-\Phi'(g\otimes f)\|<\dt/2\tforal g\in {\cal G}'' \andeqn f\in \{1, z, z^*\},
\eneq
where ${\cal G}''\subset {\cal G}$ such that $\imath_m(g)\in {\cal G}.$ 
Furthermore
\beq\label{Bot-5+1}
[\Phi]|_{{\cal P}_1}&=&[H_1 ]|_{{\cal P}_1}\,\,\,{\rm (} {\cal P}_1\subset \underline{K}(C(X)) {\rm )}\\\andeqn
[\Phi]|_{\boldsymbol{\bt}({\cal P})}&=&\af|_{\boldsymbol{\bt}({\cal P})},
\eneq
where $H_1: C(X)\to M_L$ is  defined by $H_1(f)=\sum_{i=1}^L f(\zeta_i)p_i$ for $f\in C(X).$ 

Define $\phi_0': C(X)\to M_L$ by $\phi_0'(f)=\Phi(f\otimes 1_{C(\T)})$ for all
$f\in C(X).$
Define
$$
u_0=\Phi(1_{C(X)}\otimes z)(\Phi(1_{C(X)}\otimes z^*)\Phi(1_{C(X)}\otimes z))^{1/2}=\langle \Phi(1_{C(X)}\otimes z)\rangle.
$$
Then
$$
\|u_0-\Phi(1_{C(X)}\otimes z)\|<\dt_{C(X), {\cal P}, {\bf b}}.
$$


%
Define $\phi': C(X)\to M_{n-L}$ defined by
\beq\label{Bbot-8}
\phi'(f)=\sum_{i=L+1}^{n}f(\zeta_i)p_i\tforal f\in C(X).
\eneq
Define $\phi_1: C(X)\to M_n$ by
\beq\label{Bbot-9}
\phi_1(f)=\phi'(f)\oplus \phi_0(f)\tforal f\in C(X).
\eneq
Since $k/n<d\le \gamma (k/L),$ $L/n<\gamma.$ Therefore
one computes that
\beq\label{Bbot-10}
[\phi_1]|_{{\cal P}_1}=[\phi_2]|_{{\cal P}_1}\andeqn\\
|\tau\circ \phi(g)-\tau\circ \phi_1(f)|&<&\gamma
\tforal g\in {\cal H}.
\eneq
Moreover, since  $k/n<d\le (\sigma/2)k/L,$ $L/n<\sigma/2.$
Therefore, by (\ref{Bbot-6}),
$$
\mu_{tr\circ \phi_1}(O_r)\ge \sigma/2\andeqn \mu_{tr\circ \phi}(O_r)\ge \sigma/2
$$
for all $r\ge \eta.$

It follows from \ref{Uniqq} and \ref{Add1109-1} that there is a  unitary $w_1\in M_n$ such that
\beq\label{Bbot-11}
\|{\rm Ad}\, w_1\circ \phi_1(f)-\phi(f)\|<\ep/4
\tforal f\in {\cal F}.
\eneq
Put
\beq\label{Bbot-12}
u=w^*({\rm diag}(\overbrace{1,1,...,1}^{n-L}, u_0))w.
\eneq
One checks that this unitary $u$ meets all the requirements.

\end{proof}

The following is a folklore.

\begin{lem}\label{1keeptrace}
Let $X$ be a compact metric space, let $\eta_i>0$ and $\sigma_i>0$
($i=1,2,...,m$) with $\eta_1>\eta_2>\cdots >\eta_m$ and
$\sigma_1>\sigma_2>\cdots >\sigma_m,$ and let $0<\lambda_1, \lambda_2<1.$
There exists $\dt>0$ and a finite subset ${\cal G}\subset C(X)$ satisfying the following:

Suppose that $A$ is a unital \CA\, with $T(A)\not=\emptyset$ and suppose that $\phi, \psi: C(X)\to A$ are two unital positive linear maps such that
\beq\label{1keep-1-0}
\mu_{\tau\circ \phi}(O_r)\ge \sigma_j
\eneq
for all $r\ge \eta_j,$ $j=1,2,...,m,$ and
\beq\label{1keep-1}
|\tau\circ \phi(g)-\tau\circ \psi(g)|<\dt\tforal g\in {\cal G}.
\eneq
for all $\tau\in T(A).$
Then,
$$
\mu_{\tau\circ \psi}(O_r)\ge \lambda_1\sigma_j
$$
for all $r\ge 2(1+\lambda_2)\eta_j,$ $j=1,2,...,m,$ and for all
$\tau\in T(A).$
\end{lem}

\begin{proof}
To simplify the proof, without loss of generality, we will prove only for the case that $m=1.$ The general case follows by taking minimum of $m$ $\dt$'s and  the union
of $m$ ${\cal G}'s.$

There are $x_1,x_2,...,x_K\in X$ such that
$$
\cup_{k=1}^K O(x_k,{\eta})\supset X.
$$
There are $f_1,f_2,...,f_K\in C(X)$ with $0<f_k\le 1$ such that
$f_k(x)=1$ if $x\in O(x_k,{\eta})$ and
$f_k(x)=0$ if ${\rm dist}(x, x_k)>(1+\lambda_2)\eta.$
Choose $\dt=(1-\lambda_1)\sigma_1$ and ${\cal G}=\{f_1,f_2,...,f_K\}.$

Now suppose that $\phi, \psi: C(X)\to A$ are two unital positive linear maps which satisfy the assumption  (\ref{1keep-1-0}) and (\ref{1keep-1}).

Let $x\in X$ and consider $O(x,r)$ for some $r\ge 2(1+\lambda_2)\eta.$
Then there exists $x_k$ such that
${\rm dist}(x, x_k)<\eta.$ This implies that
$$
O(x_k,(1+\lambda_2)\eta)\subset O(x,r).
$$
Thus
\beq
\mu_{\tau\circ \psi}(O(x,r))
&\ge &\tau\circ \psi(f_k)> \tau\circ \phi(f_k)-(1-\lambda_1)\sigma_1\\
&\ge &  \mu_{\tau\circ \phi}(O(x_k,\eta))-(1-\lambda_1)\sigma_1\\
&\ge & \lambda_1\sigma_1.
\eneq
for all $\tau\in T(A).$

\end{proof}

\begin{rem}
{\rm Note that in the above lemma, we insist that $\dt$ and ${\cal G}$ do not depend on $\phi.$
Otherwise one can have better estimates.}
\end{rem}

\begin{lem}\label{keeptrace}
Let $X$ be a compact metric space, let $\Delta: (0,1)\to (0,1)$ be a nondecreasing function, let $\eta>0$ and let $0<\lambda_1, \lambda_2<1.$
There exists $\dt>0$ and a finite subset ${\cal G}\subset C(X)$ satisfying the following:

Suppose that $A$ is a unital \CA\, with $T(A)\not=\emptyset$ and suppose that $\phi, \psi: C(X)\to A$ are two unital positive linear maps such that
\beq\label{Keep-1}
\mu_{\tau\circ \phi}(O_r)\ge \Delta(r)
\eneq
for all $r\ge \eta$ and
\beq\label{Keep-2}
|\tau\circ \phi(g)-\tau\circ \psi(g)|<\dt\tforal g\in {\cal G}.
\eneq
Then,
$$
\mu_{\tau\circ \psi}(O_r)\ge \lambda_1\Delta(r/2(1+\lambda_2))
$$
for all $r\ge 2(1+\lambda_2)\eta.$
\end{lem}

\begin{proof}
Let $\eta>0,$ $\Delta$ and  $0<\lambda_1, \lambda_2<1$ be given.
Choose $\lambda_0>0$ such that $0<\lambda_0<\lambda_2.$
Let $1>r_1>r_2>\cdots >r_N>0$ such that $\eta>r_N$ and
$$
r_{i+1}/r_i>{1+\lambda_0\over{1+\lambda_2}}, \,\,\,i=1,2,....,N-1.
$$
Put $\eta_j=r_j$ and $\sigma_j=\Delta(\eta_j),$ $j=1,2,...,N-1.$

Let $\dt>0$ and ${\cal G}$  be required by \ref{1keeptrace}
for $\eta_j$ and $\sigma_j$ ($j=1,2,...,N$),$ \lambda_1$ and $\lambda_2.$

Now suppose that
$\phi, \psi$ satisfy (\ref{Keep-1}) and (\ref{Keep-2}).
By applying \ref{1keeptrace}, we conclude that
$$
\mu_{\tau\circ \psi}(O_r)\ge \lambda_1\sigma_j
$$
for all $\tau\in T(A)$ and all $r\ge 2(1+\lambda_0)\eta_j,$
$j=1,2,...,N.$

Now suppose that $r\ge 2(1+\lambda_2)\eta>2(1+\lambda_0)\eta.$
Then
$$
{r\over{2(1+\lambda_0)}}>\eta.
$$
We may assume that, for some $j,$
$$
\eta_j>{r\over{2(1+\lambda_0)}}>\eta_{j+1}.
$$
Then
\beq
\mu_{\tau\circ \psi}(O_r) &>& \lambda_1\sigma_{j+1}=\lambda_1\Delta(\eta_{j+1})\\
&\ge & \lambda_1\Delta(\eta_j({1+\lambda_0\over{1+\lambda_2}}))\\
&\ge & \lambda_1\Delta({r\over{2(1+\lambda_2)}})
\eneq
for all $\tau\in T(A).$

\end{proof}

\begin{lem}\label{easy}
Let $u\in CU(M_n(C([0,1]))$ be a unitary such that
\beq\label{Le-1}
\|u(0)u(t)^*-1\|<1\tforal t\in [0,1].
\eneq
Suppose that $u(0)u(1)^*=\exp(\sqrt{-1}h)$ with
$\|h\|<2\arcsin (1/2).$
Then
$$
{\rm Tr}(h)=0.
$$
\end{lem}

\begin{proof}
By the assumption (\ref{Le-1}),  there  exists a selfadjoint element $b\in M_n(C([0,1])$ such that
$$
u(0)u(t)^*=\exp(\sqrt{-1}b(t))\andeqn \|b\|<2\arcsin(1/2).
$$
However, $u(0)u(t)^*\in CU(M_n(C([0,1])).$ It follows that
$$
({1\over{2\pi}}){\rm Tr}(b(t))\in \Z.
$$
Therefore
${\rm Tr}(b(t))$ is a constant.
Since $b(0)=0,$
$$
({1\over{2\pi}}){\rm Tr}(b(t))=0\tforal t\in [0,1].
$$
Note that $h=b(1).$ Therefore
$$
{\rm Tr}(h)=0.
$$
\end{proof}

%

\begin{thm}\label{NMT1}
Let $X$ be a compact metric space, let ${\cal F}\subset C(X)$ be a finite subset and let $\ep>0$ be a positive number.
There exists $\eta_1>0$ satisfying the following:
for any $\sigma_1>0,$ there exists $\eta_2>0$ satisfying the following:
for any $\sigma_2>0,$ there exists $\eta_3>0$ satisfying the following:
for any $\sigma_3>0,$ there exists $\eta_4>0$ satisfying the following:
For any  $\sigma_4>0,$ there exists $\gamma_1>0,$ $\gamma_2>0,$
$\dt>0,$ a finite subset ${\cal G}\subset C(X)$  and a finite subset ${\cal P}\subset \underline{K}(C(X))$
 a finite subset ${\cal H}\subset C(X)$ and a finite subset ${\cal U}\subset U_c(K_1(C(X)))$
 for which $[{\cal U}]\subset {\cal P}$ satisfying the following:
For any two unital $\dt$-${\cal G}$-multiplicative \morp s
$\phi, \psi: C(X)\to M_n(C([0,1])$  such that
\beq\label{NT-1-1}
[\phi]|_{\cal P}=[\psi]|_{\cal P}=[h]|_{\cal P}
\eneq
for some unital \hm\, $h: C(X)\to M_n(C([0,1])),$
\beq\label{NT-1}
\mu_{\tau\circ \phi}(O_r)\ge \sigma_i,\,\,\,
\mu_{\tau\circ \psi}(O_r)\ge \sigma_i,
\eneq
for all $\tau\in T(M_n(C([0,1])))$ and for all $r\ge \eta_i,$ $i=1,2,3,$
\beq\label{NT-2}
|\tau\circ \phi(g)-\tau\circ \psi(g)|<\gamma_1
\tforal g\in {\cal H}\tand\\
{\rm dist}(\overline{\langle \phi(u)\rangle},\overline{\langle \psi(u)\rangle})<\gamma_2\tforal u\in {\cal U},
\eneq
there exists a unitary $W\in M_n(C([0,1]))$ such that
\beq\label{NT-3}
\|W\phi(f)W^*-\psi(f)\|<\ep\tforal f\in {\cal F}.
\eneq

\end{thm}

(Note, as stated in \ref{d1}, $\phi$ and $\psi$ in (\ref{NT-2}) are in fact $\phi\otimes {\rm id}_{M_k}$ and
$\psi\otimes {\rm id}_{m_k}$ for some integer $k\ge 1.$ This will be used in the proof below.)

\begin{proof}
Put $B=M_n(C([0,1])).$
We may write $C(X)=\lim_{n\to\infty} (C(Y_n), \imath_n),$ where $Y_n$ is a finite CW complex.
Let $\ep>0$ and a finite subset ${\cal F}\subset C(X)$ be given.  Without loss of generality, we may assume
that ${\cal F}\subset \imath_n(C(Y_n))$ for  some $n.$
Let $\eta_1'>0$ (in place of $\eta$) be required by
\ref{Lmem} for $\ep/32$ (in place of $\ep$) and ${\cal F}.$

Let $\eta_1=\eta_1'/3.$
Let $\sigma_1>0$ and let $\sigma_1'=\sigma_1/2>0.$ Let
 $\dt_1>0$ (in place of $\ep$), ${\cal G}_1\subset C(X)$ (in place of ${\cal G}$) be a finite subset and let ${\cal P}_0\subset \underline{K}(C(X))$ (in place of ${\cal P}$) be a finite subset required by \ref{Lmem}
for $\ep/32$ (in place of $\ep$), ${\cal F},$ $\eta_1'$ and $\sigma_1'.$
We may assume that $\dt_1<\ep/32.$

There exists a finite CW complex $Y,$ a unital \hm\, $\imath: C(Y)\to C(X)$  and a finite subset
${\cal F}'\subset C(Y)$
such that $\imath({\cal F}')={\cal F}$ and $[\imath](\underline{K}(C(Y)))\supset {\cal P}_0$
(by choosing $Y=Y_n$ for some large $n$).

Let $0<\dt_2<\dt_{C(Y), \bf b}$ and ${\cal G}_2'\supset {\cal F}_{C(Y), \bf b}$ such that
$(\dt_2, {\cal G}_2')$ forms a $KK$-pair for $C(Y).$  Let ${\cal P}_0'\subset \underline{K}(C(Y))$ be such that
$\dt_{C(Y), \bf b}=\dt_{C, {\cal P}_0', \bf b}.$ To simplify the notation, without loss of generality,
we may assume that $[\imath]({\cal  P}_0')={\cal P}_0.$  Put ${\cal G}_2=\imath({\cal G}_2').$


Denote by $z\in C(\T)$ the identity function on the unit circle.
We may also assume that, for any $\dt_2$-$\{z,1\}\times {\cal G}_2$-multiplicative
\morp\, $\Lambda: C(\T)\otimes C(Y)\to C$ (for any unital \CA\, $C$
with $T(C)\not=\emptyset$), $[\Lambda]$ is well defined and
$$
\tau([\Lambda(g)])=0
$$
for all $g\in Tor(K_1(C(Y)))$ (which is a finite subgroup).

Furthermore, we may assume that $\dt_2$ is so small  that
if $\|uv-vu\|<3\dt_2,$ then the Exel formula
$$
\tau({\rm bott}_1(u,v))={1\over{2\pi\sqrt{-1}}}(\tau(\log(u^*vuv^*))
$$
holds in any unital \CA\, $C$ with tracial rank zero and any $\tau\in T(C)$ (see Theorem 3.6 of \cite{Lnamj}). Moreover
if $\|v_1-v_2\|<3\dt_2,$ then
$$
{\rm bott}_1(u,v_1)={\rm bott}_1(u,v_2).
$$

Let ${\cal U}=\{g_1, g_2,..., g_{k(X)}\}\subset U_c(K_1(C(X))$ be a finite subset
such that $\{[g_1], [g_2],...,[g_{k(X)}]\}$ forms  a set of generators for the finitely generated subgroup generated by
${\cal P}_0\cap K_1(C(X)).$  We assume that $m(X)\ge 1$ is an integer and
$g_i\in U(M_{m(X)}(C(X)).$
We may further assume that there are $g_j'$ ($j=1,2,...,k(X)$) in $U_c(K_1(C(Y)))$ such that
$\imath(g_j)=g_j',$ $j=1,2,...,k(X)$ (here again we identify a set of unitaries with its image in $U(C(Y))/CU(C(Y)))$).
Furthermore, we may assume that $g_1', g_2',...,g_{k(X)}'$ generate $K_1(C(Y)).$
Let ${\cal U}_0\subset C(X)$ be a finite subset
such that
$$
{\cal U}=\{(a_{i,j}): a_{i,j}\in {\cal U}_0\}.
$$

Let $\dt_u=\min\{1/256m(X)^2, \dt_1/16m(X)^2, \dt_2/16m(X)^2\}$ and
${\cal G}_u={\cal F}\cup{\cal G}_1\cup{\cal G}_2\cup {\cal U}_0.$

Let $\eta_2'>0$ (in place of $\eta$) and $\{s_1, s_2,...,s_{k'}\}\subset K_1(C(Y))$ be required by \ref{Bbot} for
$\dt_u$ (in place of $\ep$) and ${\cal G}_u$ (in place of ${\cal F}$).
Without loss of generality, we may assume that
$\{s_1, s_2,...,s_{k'}\}=\{g_1', g_2',...,g_{k(X)}'\}$ which generates $K_1(C(Y)).$
Put $\eta_2=\eta_2'/3.$

Let $\sigma_2>0$ and let $\sigma_2'=\sigma_2/2.$  Let $1>d>0$ be required by \ref{Bbot} for $\min\{\dt_1/4,\dt_2/4\}$ (in place $\ep$), ${\cal G}_u$ (in place of ${\cal F}$), $\eta_2$ and $\sigma_2'.$

Let $\dt_3>0$ (in place of $\dt$) and let ${\cal G}_3\subset C(\T)\otimes C(X)$ (in place of ${\cal G}$) be required by Lemma 10.3 of \cite{LnApp} for $d/8$ (in place of $\sigma$) and $\T\times X$ (in place of $X$).
Without loss of generality, we may assume
that
$$
{\cal G}_3=\{z\otimes g: g\in {\cal G}_4'\}\cup\{1\otimes g: g\in {\cal G}_4\},
$$
where
${\cal G}_4'\subset C(X)$ is a finite subset (by  choosing a smaller $\dt_3$ and large ${\cal G}_3$).

Let $\ep_1'>0$ (in place of $\dt$) and let ${\cal G}_4''\subset C(X)$ (in place of ${\cal G}$) be a finite subset
required by \ref{1keeptrace} for $\eta_1,\eta_2,$ $\sigma_1, \sigma_2,$ $1/2$ (in place of $\lambda_1$) and
$1/4$ (in place of $\lambda_2$).

Let $\ep_1''=\min\{d/27m(X)^2, \dt_u/2,  \dt_3/2m(X)^2, \ep_1'/2m(X)^2\}$ and
let ${\bar \ep}_1>0$ (in place of $\dt$) and ${\cal G}_5\subset C(X)$ (in place of ${\cal F}_1$) be a finite subset
required by 2.8 of \cite{Lnmem} for $\ep_1''$ (in place of $\ep$) and ${\cal G}_u\cup {\cal G}_4'\cup {\cal G}_4''$
(and $C(X)$ in place of $B$).
Put
$$
\ep_1=\min\{\ep_1', \ep_1'',{\bar \ep_1}\}.
$$

Let $\eta_3'>0$ (in place of $\eta$) be required by \ref{Uniqq} for
$\ep_1/4$ (in place $\ep$) and ${\cal G}_5$ (in place of ${\cal F}$).

Let $\eta_3''>0$ (in place of $\eta_1$) be required by \ref{Lijram}
for $\ep_1/4$ (in place of $\ep$) and ${\cal G}_5$
(in place of ${\cal F}$).
Let $\eta_3=\min\{\eta_3', \eta_3''\}.$
Let $\sigma_3>0.$ Let $\gamma_1>0$ (in place of $\gamma$), $\dt_4>0,$ ${\cal G}_6\subset C(X)$ (in place of ${\cal G}$), ${\cal H}\subset C(X)$ be a finite subset and let ${\cal P}_1\subset \underline{K}(C(X))$ (in place of ${\cal P}$) be required by \ref{Uniqq} for $\ep_1/4$ (in place of $\ep$), ${\cal G}_5$ (in place of ${\cal F}$), $\eta_3$ (in place $\eta$) and $\sigma_3$ (in place
$\sigma$).
Let $\eta_4>0$ (in place of $\eta_2$) be required by \ref{Lijram} for
$\ep_1/4$ (in place of $\ep$), ${\cal G}_5$ (in place of ${\cal F}$),
$\eta_3$ (in place of $\eta_1$), $\sigma_3$ (in place of $\sigma_1$).
Let $\sigma_4>0.$ Let $\dt_5>0,$ ${\cal G}_7\subset C(X)$
(in place of ${\cal G}$),${\cal P}_2\subset \underline{K}(C(X))$ (in place of ${\cal P}$) be required by \ref{Lijram}.

Let $\dt=\min\{\ep_1/4, \dt_4, \dt_5\},$
${\cal G}={\cal G}_5\cup {\cal G}_6\cup {\cal G}_7\cup {\cal H}$ and
${\cal P}={\cal P}_0\cup {\cal P}_1\cup {\cal P}_2.$
Let
$\gamma_2<\min\{d/16m(X)^2, \dt_u/9m(X)^2, 1/256m(X)^2\}.$
We may assume that $(\dt, {\cal G}, {\cal P})$ is a $KL$-triple.

Now suppose that $\phi, \psi: C(X)\to B$
are two unital $\dt$-${\cal G}$ multiplicative \morp s which satisfy the assumption for the above
$\eta_i,$ $\dt_i$ ($i=1,2,3,4$), $\gamma_i$ ($i=1,2$), ${\cal P},$ ${\cal U}$ and ${\cal H}.$

Choose a partition
$$
0=t_0<t_1<\cdots <t_N=1
$$
such that
\beq\label{NT-3-1}
\|\pi_t\circ \phi(g)-\pi_{t_{i-1}}\circ \phi(g)\|<\ep_1/4\andeqn
\|\pi_t\circ \psi(g)-\pi_{t_{i-1}}\circ \psi(g)\|<\ep_1/4
\eneq
for all $g\in {\cal G}$ and
for all $t\in [t_{i-1}, t_i],$ $i=1,2,...,N.$
By applying \ref{Uniqq}, for each $i,$ there exists a unitary
$w_i\in M_n$ such that
\beq\label{NT-3+n}
\|w_i\pi_{t_i}\circ \phi(g)w_i^*-\pi_{t_i}\circ \psi(g)\|<\ep_1/4\tforal g\in {\cal G}_5
\eneq
and, by \ref{Lijram},  there are unital \hm s $h_{i,1}, h_{i, 2}: C(X)\to M_n$ such that
\beq\label{NT-4}
\|\pi_{t_i}\circ \phi(g)-h_{i,1}(g)\|<\ep_1/4\andeqn
\|\pi_{t_i}\circ \psi(g)-h_{i,2}(g)\|<\ep_1/4
\tforal g\in {\cal G}_5,
\eneq
$i=0,1,2,...,N.$
Moreover (by also applying \ref{1keeptrace}),
\beq\label{NT-4+}
\mu_{tr\circ h_{i,j}}(O_r)\ge \sigma_k'
\eneq
for all $r\ge \eta_k',$ $k=1,2,$ $j=1,2$ and $i=1,2,...., N.$
Let  $\omega_j\in M_{m(X)}(B)$ be a unitary
such that $\omega_j\in CU(M_{m(X)}(B)$ and
\beq\label{NT-5}
\|\langle\phi(g_j^*)\rangle \langle \psi(g_j)\rangle  -\omega_j\|<\gamma_2,\,\,\,i=1,2,...,k(X).
\eneq
Write
$$
\omega_j=\exp(\sqrt{-1}a_j)
$$
for some selfadjoint element $a_j\in M_{m(X)}(M_n(C([0,1]))),$ $j=1,2,...,k(X).$
Then
$$
{n(t\otimes {\rm Tr}_{m(X)})(a_j(s))\over{2\pi}}\in \Z
$$
($s\in [0,1]$), where $t$ is the normalized trace on $M_n.$
It follows that
the above is a constant. In particular,
\beq\label{NT-5+1}
n(t\otimes {\rm Tr}_{m(X)})(a_j(t_i))=n(t\otimes {\rm Tr}_{m(X)})(a_j(t_{i-1})),
\eneq
$i=1,2,...,N$ and $j=1,2,...,k(X).$

Let $W_i=w_i\otimes {\rm id}_{M_{m(X)}},$ $i=0,1,....,N.$
Then
\beq\label{NT-5+2}
\|(h_{i, 1}\otimes {\rm id}_{M_{m(X)}})(g_j^*)W_i(h_{i, 1}\otimes {\rm id}_{M_{m(X)}})(g_j)W_i^*-\omega_j(t_i)\|<3m(X)^2\ep_1+2\gamma_2<1/32
\eneq

It follows from (\ref{NT-5}) that there exists selfadjoint elements $b_{i,j}\in M_{nm(X)}$ such that
\beq\label{NT-5+3}
\exp(\sqrt{-1}b_{i,j})=\omega_j(t_i)^*(h_{i, 1}\otimes {\rm id}_{M_{m(X)}})(g_j^*)W_i(h_{i, 1}\otimes {\rm id}_{M_{m(X)}})(g_j)W_i^*
\eneq
such that
\beq\label{NT-5+4}
\|b_{i,j}\|<2\arcsin (3m(X)^2\ep_1/4+2\gamma_2),\,\,\,j=1,2,...,k(X),\,i=0,1,...,N.
\eneq
Note that
\beq\label{NT-5+5}
(h_{i, 1}\otimes {\rm id}_{M_{m(X)}})(g_j^*)W_i(h_{i, 1}\otimes {\rm id}_{M_{m(X)}})(g_j)W_i^*
=\omega_j(t_i)\exp(\sqrt{-1}b_{i,j}),
\eneq
$j=1,2,...,k(X)$ and $i=0,1,...,N.$

Then,
\beq\label{NT-5+5+}
{n\over{2\pi}}(t\otimes {\rm Tr}_{M_{m(X)}})(b_{i,j})\in \Z,,\,\,\,j=1,2,...,k(X),\,i=0,1,...,N.
\eneq

Let
$$
\lambda_{i,j}={n\over{2\pi}}(t\otimes {\rm Tr}_{M_{m(X)}})(b_{i,j})
$$
$j=0,1,2,...,k(X),$ $i=1,2,...,N.$
Note that $\lambda_{i,j}\in \Z.$

Define
$\af_i^{(0,1)}: K_1(C(Y))\to \Z$  by
mapping $g_j'$ to $\lambda_{i,j},$ $j=1,2,...,m(X)$ and $i=0,1,2,...,N.$
We write $K_0(C(\T)\otimes C(Y))=K_0(C(Y))\oplus {\boldsymbol{\bt}}(K_1(C(Y)))$
(see  2.10 of \cite{Lnmem} for the definition
of ${\boldsymbol{\bt}}$).
Define $\af_i: K_*(C(\T)\otimes C(Y))\to K_*(M_n)$ as follows
\beq\label{NT-6}
\af_i|_{K_0(C(\T)\otimes C(Y))}([1])=n,\\
\af_i|_{{\rm ker}\rho_{C(Y)}}=0,\\
\af_i|_{{\boldsymbol{\bt}}(K_1(C(Y)))}=\af_i\circ {\boldsymbol{\bt}}|_{K_1(C(Y))}=\af_i^{(0,1)},\\
\af_i|_{K_1(C(\T)\otimes C(Y))}=0,\\
\eneq
By the Universal Coefficient Theorem (\cite{RS}), there exists an element $\af_i\in KK(C(\T)\otimes C(Y), \C)$ such that
$\af_i|_{K_*(C(\T)\otimes C(Y)}=\af_i$ as defined above, $i=1,2,...,N.$
We estimate that
$$
\|(w_{i-1}^*w_i)h_{i-1,1}(g)-h_{i-1,1}(g)(w_{i-1}^*w_i)\|<\ep_1\tforal g\in {\cal G}_5.
$$
Let $\Lambda_i: C(\T)\otimes C(X)\to M_n$ be a unital \morp\, given
by the pair $w_{i-1}^*w_i$ and $h_{i-1,1},$ $i=1,2,...,N$ (see 2.8 of \cite{Lnmem}).
Denote $V_{i,j}=h_{i,1}\otimes {\rm id}_{M_{m(X)}}(g_j),$ $j=1,2,...,m(X)$ and $i=0,1,2,...,N.$
Note that
\beq\label{NT-6+0}
\|W_{i-1}V_{i-1,j}^*W_{i-1}^* V_{i-1,j}V_{i-1,j}^*W_iV_{i-1,j}W_i^*-1\|<1/16\\
\|W_{i-1}V_{i-1,j}^*W_{i-1}^*V_{i-1,j}V_{i,j}^*
W_i V_{i,j}W_i^*-1\|<1/16
\eneq
and there is a continuous path $Z(t)$
of unitaries  such that $Z(0)=V_{i-1,j}$ and $Z(1)=V_{i,j}.$
We obtain a continuous path
$$
W_{i-1}V_{i-1,j}^*W_{i-1}^*V_{i-1,j}Z(t)^*W_i Z(t)W_i^*
$$
which is in $CU(M_{nm(X)})$
for all $t\in  [0,1].$
It follows that
$$
(1/2\pi\sqrt{-1})(t\otimes {\rm Tr}_{M_{m(X)}})[\log(W_{i-1}V_{i-1,j}^*W_{i-1}^*V_{i-1,j}Z(t)^*W_i Z(t)W_i^*)]
$$
is a constant. In particular,
\beq\label{NT-6-0+}
(1/2\pi\sqrt{-1})(t\otimes {\rm Tr}_{M_{m(X)}})(\log(W_{i-1}V_{i-1,j}^*W_{i-1}^*W_i V_{i-1,j}W_i^*))\\\label{NT-6-0++}
=(1/2\pi\sqrt{-1})(t\otimes {\rm Tr}_{M_{m(X)}})(\log(W_{i-1}V_{i-1,j}^*W_{i-1}^*V_{i-1, j}V_{i,j}^*W_i V_{i,j}W_i^*)).
\eneq
Also
\beq\label{NT-6+1}
&&W_{i-1}V_{i-1,j}^*W_{i-1}^*V_{i-1, j}V_{i,j}^*W_i V_{i,j}W_i^*\\
&&=(\omega_j(t_{i-1})\exp(\sqrt{-1}b_{i-1,j}))^*\omega_j(t_i)
\exp(\sqrt{-1}b_{i,j})\\\label{NT-6+1+}
&&=\exp(-\sqrt{-1}b_{i-1,j})\omega_j(t_{i-1})^*\omega_j(t_i)\exp(\sqrt{-1}b_{ij}).
\eneq
Note that, by (\ref{NT-5}) and (\ref{NT-5+2}),
\beq
\|\omega_j(t_{i-1})^*\omega_j(t_i)-1\|<3(3\ep_1'+2\gamma_2)<3/32,
\eneq
$j=1,2,...,m(X),$ $i=1,2,...,N.$
By \ref{easy}
\beq\label{LeasyApp}
(t\otimes {\rm Tr}_{m(X)})(\log(\omega_j(t_{i-1})^*\omega_j(t_i)))=0.
\eneq

It follows that (by the Exel formula, using (\ref{NT-6-0++}), (\ref{NT-6+1+}) and (\ref{LeasyApp}))
\beq\label{NT-6+2}
&&\hspace{-0.4in}(t\otimes {\rm Tr}_{m(X)})({\rm bott}(V_{i-1,j}, W_{i-1}^*W_i))\\
 &=&
({1\over{2\pi \sqrt{-1}}})(t\otimes {\rm Tr}_{m(X)})(\log(V_{i-1,j}^*W_{i-1}^*W_iV_{i-1,j}W_i^*W_{i-1}))\\
&=&({1\over{2\pi \sqrt{-1}}})(t\otimes {\rm Tr}_{m(X)})(\log(W_{i-1}V_{i-1,j}^*W_{i-1}^*W_iV_{i-1,j}W_i^*))\\
&=&({1\over{2\pi \sqrt{-1}}})(t\otimes {\rm Tr}_{m(X)})(\log(W_{i-1}V_{i-1,j}^*W_{i-1}^*V_{i-1,j}V_{i,j}^*
W_iV_{i,j}W_i^*))\\
&=& ({1\over{2\pi \sqrt{-1}}})(t\otimes {\rm Tr}_{m(X)})(\log(\exp(-\sqrt{-1}b_{i-1,j})\omega_j(t_{i-1})^*\omega_j(t_i)\exp(\sqrt{-1}b_{ij}))\\
&=& ({1\over{2\pi \sqrt{-1}}})[(t\otimes {\rm Tr}_{k(n)})(-\sqrt{-1}b_{i-1,j})+(t\otimes {\rm Tr}_{k(n)})(\log(\omega_j(t_{i-1})^*\omega_j(t_i))\\
&&\hspace{1.4in}+(t\otimes {\rm Tr}_{k(n)})(\sqrt{-1}b_{i,j})]\\
&=&{1\over{2\pi}}(t\otimes {\rm Tr}_{k(n)})(-b_{i-1,j}+b_{i,j}).
\eneq
In other words,
\beq\label{NT-+}
{\rm bott}(V_{i-1,j}, W_{i-1}^*W_i))=-\lambda_{i-1,j}+\lambda_{i,j}
\eneq
$j=1,2,...,m(X)$ and $i=1,2,...,N.$

Define $\bt_0=0,$
$\bt_1=[\Lambda_1]-\af_1+\af_0+\bt_0$ and
\beq\label{NT-+1}
\bt_i=[\Lambda_i]-\af_{i}+\af_{i-1}+\bt_{i-1},
\,\,\,i=1,2,...,N.
\eneq
 Define  $\kappa_0=0$ and $\kappa_i=\af_i+\bt_i,$ $i=1,2,...,N.$
 Note that $\af_i, \bt_i, \kappa_i\in
KK(C(\T)\otimes C(Y)), \C).$
We  compute that
\beq\label{Sch-1}
\bt_1(g_j')&=&[\Lambda_1](g_j')-\lambda_{1,j}+\lambda_{0,j}=0,\\
\bt_2(g_j')&=&[\Lambda_2](g_j')-\lambda_{2,j}+\lambda_{1,j}+\bt_1(g_j')=0
\andeqn\\
\bt_i(g_j')&=&0,\,\,\,i=1,2,...,N\andeqn
\,\,\,j=1,2,...,k(X).
\eneq
It follows that
\beq
|\tau\otimes {\rm Tr}_{m(X)}(\kappa_i([g_i]))|&=&
|\lambda_{i,j}/n| <d/2,
\eneq
$j=1,2,...,N$ and $i=1,2,...,k(X).$
By applying \ref{Bbot}, there is, for each $i=1,2,...,N,$
a unitary $z_i\in M_n$ such that
\beq\label{NT-7}
\|[z_i,\,h_{i,1}(g)]\|<\dt_u\tforal g\in {\cal G}_u\\
\andeqn
{\rm Bott}(z_i, h_{i,1}\circ \imath)=\kappa_i,\,\,\,i=0,1,2,...,N-1.
\eneq
Let $U_i=z_{i-1}w_{i-1}^*w_iz_i^*,$ $i=1,2,...,N.$
Then
\beq\label{NT-8}
\|[U_i, h_{i-1,1}(g)]\|<\min\{\dt_1, \dt_2\}\tforal g\in {\cal G}_u,
\eneq
$i=1,2,...,N.$
Moreover
\beq
{\rm Bott}(U_i, h_{i-1,1}\circ \imath)&=&
{\rm Bott}(z_{i,1},\, h_{i-1,1}\circ \imath)+{\rm Bott}(w_{i-1}^*w_i, h_{i-1,1}\circ \imath)\\
&&+{\rm Bott}(z_i^*,\, h_{i-1,1}\circ\imath)\\
&=& \kappa_{i-1}+[\Lambda_i]-\kappa_i\\
&=& \af_{i-1}+\bt_{i-1}+[\Lambda_i]
-\af_i-\bt_i\\
&=& \af_{i-1}+\bt_{i-1}+[\Lambda_i]
-\af_i-([\Lambda_i]-\af_{i}+\af_{i-1}+\bt_{i-1})\\
&=& 0
\eneq
$i=0,1,2,...,N-1.$
It follows that
$$
{\rm Bott}(U_i,\,h_{i-1,1})|_{\cal P}=0,\,\,\,i=1,2,...,N-1.
$$
By applying \ref{Lmem}, there exists a continuous path of unitaries, $\{U_i(t): t\in [t_{i-1}, t_i]\}$ such that
\beq\label{NT-11}
U_i(t_{i-1})=1,\,\,\, U_i(t_i)=z_{i-1}w_{i-1}^*w_iz_i^*\andeqn\\
\|U(t)h_{i-1,1}(f)U(t)^*-h_{i-1,1}(f)\|<\ep/32
\eneq
for all $f\in {\cal F}$ and for all $t\in [t_{i-1}, t_i],$ $i=1,2,...,N.$
Define
$W\in B$ by
$$
W(t)=w_{i-1}z_{i-1}^*U_i(t)\tforal t\in [t_{i-1}, t_i],\,\,\,i=1,2,...,N.
$$
Note that $W(t_{i-1})=w_{i-1}z_{i-1}^*,$ $i=1,2,...,N,$ and
$W(1)=w_Nz_N^*.$
One checks that, by (\ref{NT-4}), (\ref{NT-11}), (\ref{NT-7}), for $t\in [t_{i-1}, t_i],$
\beq
&&\hspace{-0.4in}\|W(t)\pi_t\circ \phi(f)W(t)^*-\pi_t\circ \psi(f)\| \\
&<&\|W(t)\pi_t\circ \phi(f)W(t)^*-W(t)\pi_{t_{i-1}}\circ \phi(f)W(t)^*\|\\
&&\hspace{0.2in}+
\|W(t)\pi_{t_{i-1}}\circ \phi(f)W(t)^*-W(t)h_{i-1,1}(f)W(t)^*\|\\
&&\hspace{0.4in}+
\|W(t)h_{i-1,1}(f)W(t)^*-W(t_{i-1})h_{i-1,1}(f)W(t_{i-1})^*\|\\
&&\hspace{0.6in}+\|W(t_{i-1})h_{i-1,1}(f)W(t_{i-1})^*-w_{i-1}h_{i-1,1}(f)w_{i-1}^*\|\\
&&\hspace{0.8in} +\|w_{i-1}h_{i-1,1}(f)w_{i-1}^*-w_{i-1}\pi_{t_{i-1}}\circ \phi(f)w_{i-1}^*\|\\
&&\hspace{1in} +\|w_{i-1}\pi_{t_{i-1}}\circ \phi(f)w_{i-1}^*-\pi_{t_{i-1}}\circ \psi(f)\|\\
&&\hspace{1.2in}+\|\pi_{t_{i-1}}\circ \psi(f)-\pi_t\circ \psi(f)\|\\
&<& \ep_1/4+\ep_1/4+\ep/32+\dt_u+\ep_1/4+\ep_1/4+\ep_1/4< \ep
\eneq
for all $f\in {\cal F}.$

\end{proof}

\begin{rem}
{\rm
By an argument used in \ref{MT2}, we can remove the part of the assumption in (\ref{NT-1-1}) that  $[\phi]|_{\cal P}$ is the same as $[h]|_{\cal P}$ for some \hm. At present, we do not use that form of the statement.
}
\end{rem}

\section{Preparation for the proof}

\begin{lem}\label{torsionlength}
There is an integer $K>0$ satisfying the following condition:
Suppose that $u\in M_n(C([0,1])$ for some integer $n\ge K.$ Then,
for any integer $k>0$ and any $L>0,$ if ${\rm cel}(u^k)\le L,$ then
${\rm cel}(u)\le 2\pi/K +L/k+6\pi.$

\end{lem}

\begin{proof}
(See the proof of 6.10 of \cite{Lncltr1}.) Let $K=K(1)$ be given  in Lemma 3.4 of \cite{Ph1}.
Let $n\ge K.$ 
It follows from Lemma 3.3 (1) of \cite{Ph1} that there exists a
selfadjoint element $a\in M_n(C([0,1]))$ with $\|a\|\le L$ such that
$$
{\rm det}(\exp(ia) u^k)(t)=1
$$
for every $t\in [0,1].$ 
Fix one of such integer $n.$ So
$$
{\rm det}((\exp(i a/k)u)^k)(t)=1
$$
for all $t\in [0,1].$
This implies that, for each $t\in [0,1],$
\beq\label{torsionl-1}
{\rm det}(\exp(i a(t)/k)u(t))=\exp( 2\pi i l(t)/k)
\eneq
for some integer $l(t)\le k.$ Suppose that $b(t)=-2\pi l(t)/k.$ Then
$b(t)$ is a real valued continuous function on $[0,1],$ whence
it is a constant. Note that
$$\exp(i (b(t)/n)\exp(ia/k)=\exp(i(b(t)/n+a/k)).$$ Then
\beq\label{torsionl-2}
{\rm det}(\exp(i(b(t)/n)\exp(ia/k)u)=1
\eneq
(for all $t\in [0,1]).$ By 3.4 and 3.1 of \cite{Ph1},
\beq\label{torsionl-3}
{\rm cel}(u)\le 2\pi/K+L/k+6\pi.
\eneq

\end{proof}

\begin{thm}\label{Olduniq}
Let $X$ be a compact metric space. Let  $\ep>0$ and let ${\cal F}\subset C(X)$  be a finite subset. Suppose that
$\lambda: U_c(K_1(C(X)))\to \R_+$ is a map.
There exist $\dt>0,$ a finite subset
${\cal G}\subset C(X),$ a finite subset ${\cal P}\subset
\underline{K}(C(X)),$  a finite subset of unitaries
${\cal U}\subset U_c(K_1(C(X)))$ and an integer $L>0$ satisfying the following condition:
if $\phi, \psi: C(X)\to C([0,1], M_n)$ (for some integer $n\ge 1$)
 are two
unital $\dt$-${\cal G}$-multiplicative \morp s such that
\beq\label{Olduni-1}
[\phi]|_{\cal P}=[\psi]|_{\cal P}\tand {\rm dist
}(\overline{\langle\phi(u)\rangle},\overline{\langle\psi(u)\rangle})\le \lambda(u)
\eneq
for all $u\in {\cal U},$
then there is a \hm\, $\Phi: C(X)\to
M_L(M_n(C([0,1])))$ with finite dimensional range and a unitary
$U\in M_{L+1}(M_n(C([0,1], M_n)))$ such that
\beq\label{Olduni-2}
\|U^*{\rm diag}(\phi(f), \Phi(f))U-{\rm diag}(\psi(f),
\Phi(f))\|<\ep
\eneq
for all $f\in {\cal F}.$
\end{thm}

\begin{proof}
This follows from Theorem 3.2 of \cite{GL}. One takes
$B=M_n(C([0,1])).$ Note that $B$ has stable rank one and
$K_0$-divisible rank $T,$ where $T: N\times N$ is defined by
$T(k,m)=[m/k]+1.$ Let $K$ be the constant described in Lemma 3.4 of
\cite{Ph1} (for $d=1$).
Pick a point $\xi\in X.$ If $n\ge K,$ we continue the argument below. If $n<K,$
define $\phi_0: C(X)\to M_{K-n}(C([0,1]))$ by $\phi_0(f)=f(\xi){\rm id}_{M_{K-n}}$ for all $f\in C(X).$
Replacing $\phi$ and $\psi$ by $\phi\oplus \phi_0$ and $\psi\oplus \phi_0$ and late absorbing
$\phi_0,$ we see that we may assume that $n\ge K.$

Then, by \ref{torsionlength}, $B$ has exponential divisible rank $E(L, k),$ where
$E(l,k)\le 2\pi/K+L/k+6\pi.$ It is also easy to see that $cer(B)\le
2.$
Then define  $\Lambda: U(M_{\infty}(C(X))\to \R_+$ as follows:

Let $\Pi: U(M_{\infty}(C(X)))\to K_1(C(X))$ be the quotient map and let
$J: K_1(C(X))\to \cup_{n=1}^{\infty} U(M_n(C(X)))/CU(C(X))$ be the splitting map (see \ref{Harp-Scand}). If $v\in U(M_n(C(X)))$ and $\Pi(v)\not=0,$ define $v_0=v(J\circ \Pi(v^*))_c,$ where $J\circ \Pi(v^*)_c\in U_c(K_1(C(X))).$
Define $\Lambda: U(M_{\infty}(C(X)))\to \R_+$ as follows:
\beq
\Lambda(v)&=&2{\rm cel}(v)+1\,\,\, {\rm if}\,\,\, v\in \cup_{n=1}^{\infty} U_0(M_n(C(X)))\andeqn\\
\Lambda(v)&=&\lambda(J\circ \Pi(v)_c)+6\pi+2{\rm cel}(v_0)+1\,\,\,{\rm if}\,\,\, \Pi(v)\not=0,
\eneq
where ${\rm cel}(v)$ and ${\rm cel}(v_0)$ is the exponential length
of $v$ and $v_0$ in $\cup_{n=1}^{\infty} U_0(M_n(C(X))),$ respectively.
%

Note that, for any finite subset ${\cal V}\subset U(M_m(C(X)))$ (for some integer $m\ge 1$), if $\dt$ is sufficiently small and ${\cal G}$ is sufficiently large (depends only on ${\cal V}$),
\beq\label{Old-n1}
{\rm cel}(\phi(v)\psi(v)^*)\le 2{\rm cel}(v)+1/4\le \Lambda(v) \rforal v\in {\cal V}\andeqn \Pi(v)=0.
\eneq
Otherwise, if $\Pi(v)\not=0,$  $v=v_cv_0$ for some $v_c\in U_c(K_1(C(X)))$ and $v_0\in U_0(M_{\infty}(C(X))).$ Thus, if $v_c\in {\cal U},$ $\dt$ is sufficiently small and
${\cal G}$ is sufficiently large (depends only on ${\cal V}$ and ${\cal U}$),
\beq\label{Old-n2}
{\rm cel}(\phi(v)\psi(v^*))&=&{\rm cel}(\phi(v_cv_0)\psi(v_0^*v_c^*))\\
&\le&{\rm cel}(\phi(v)\psi(v)^*)+1/4+{\rm cel}(\phi(v_c)\psi(v_c^*)\\
&\le & 2{\rm cel}(v)+1/4+1/4+ \lambda(v_c)+6\pi\le \Lambda(v).
\eneq

Therefore we can apply Theorem 3.2 of \cite{GL} directly (and the point-evaluation $f\mapsto f(\xi) {\rm id}_{M_k}$ will be absorbed into $\Phi$).

\end{proof}

The following is a folklore. It is a special case of Theorem 3.2 of \cite{GL}. It also follows from \ref{Lijram}.
We state here for the convenience for our proofs.

\begin{lem}\label{Foldunique}
Let $X$ be a compact metric space, let $\ep>0$ and let ${\cal F}\subset C(X)$ be a finite subset.
There exists $\dt>0,$ a finite subset ${\cal G}\subset C(X)$  and a finite subset ${\cal P}\subset \underline{K}(C(X))$ which forms a
$KL$-triple for $C(X)$ and an integer $N$ satisfying the following:
Suppose that $\phi: C(X)\to F$ is a unital $\dt$-${\cal G}$-multiplicative \morp, where $F$ is a finite dimensional
\CA\, such that
$$
[\phi]|_{\cal P}=[H]|_{\cal P}
$$
for some unital \hm\, $H: C(X)\to F.$ Then there exists a unital \hm\, $\Phi: C(X)\to M_N(F)$ and
a unital \hm\, $h: C(X)\to M_{N+1}(F)$ such that
$$
\|\phi(f)\oplus \Phi(f)-h(f)\|<\ep
$$
for all $f\in {\cal F}.$
\end{lem}

The following is an easy  folklore.

\begin{lem}\label{Folk}
Let $X$ be a compact metric space, $\ep>0$ be a positive number and let ${\cal F}\subset C(X)$ be a finite subset.
Then, there exists a finite subset $\{x_1, x_2,...,x_m\}\subset X$  
satisfying the following:
For any unital \hm\, $\phi: C(X)\to A$ (for any unital \CA\, $A$)  with finite dimensional range and any 
$K\ge m,$ 
there is a unital \hm\, $\psi: C(X)\to M_K$ with finite dimensional range and a unitary $u\in M_{K+1}(A)$ 
such that
$$
\|u^*(\phi(f)\oplus \psi(f))u-H(f)\|<\ep\rforal f\in {\cal F},
$$
where $H(f)=\sum_{i=1}^m f(x_i)p_i$ for all $f\in C(X)$  and 
where $p_1, p_2,...,p_m\in M_{k+1}(A)$ is  a set of mutually orthogonal projections such that
$1_A\lesssim p_i$ (for all $i$) and $\sum_{i=1}^m p_i=1_{M_{K+1}}.$
\end{lem}

\begin{proof}
Choose $\eta>0$ such that
\beq\label{44-1}
|f(x)-f(x')|<\ep/2\rforal f\in {\cal F},
\eneq
provided that ${\rm dist}(x, x')<2\eta.$ 
Choose a finite $\eta/2$ net $\{x_1, x_2,...,x_m\}$ of $X.$ 
Write $\phi(f)=\sum_{i=1}^n f(y_i)e_i$ for all $f\in C(X),$ where 
$\{y_1,y_2,...,y_n\}\subset X$ and $\{e_1, e_2,...,e_n\}\subset A$ is a set of mutually orthogonal projections.
Let $Y_1, Y_2,...,Y_{m}$ be  mutually disjoint subsets of $\{y_1, y_2,...,y_n\}$ such 
that $Y_j\in O(x_j, \eta),$ $j=1,2,...,m.$ Note that $Y_j$ could be empty.
  We rewrite 
\beq\label{44-2}
\phi(f)=\sum_{i=1}^{m} (\sum_{y_j\in Y_i} f(y_j)e_j)
\eneq
Denote $E_i=\sum_{y_j\in Y_i} e_j.$   Note 
that if $Y_i=\emptyset,$  $E_i=0.$ Now let $K\ge m.$ 
Choose mutually orthogonal $p_1', p_2',...,p_{K+1}'$ in $M_{K+1}(A)$  such that each $p_i$ is unitarily equivalent to
$1_A.$  Let $E_i'\le p_i'$ be such that $E_i'$ is unitarily equivalent to $E_i,$ $i=1,2,...,m.$
Define  $p_i=p_i',$ $i=1,2,...,m-1$ and  $p_m=\sum_{j=m}^{K+1} p_j'$ and define 
$q_i=p_i-E_i',$ $i=1,2,...,m.$ 
Define 
\beq\label{44-3}
\psi(f)=\sum_{i=1}^{m} f(x_i)q_i\andeqn\\
H(f)=\sum_{i=1}^m f(x_i)p_i
\eneq
One then ready to verify that $\psi$ and $H$ meet the requirements.
\end{proof}

{\it Added in proof: The author would like to thank Junping Liu who brought the issue of the above lemma.  He 
also benefited from a conversation with Guihua Gong about the same issue.}

\begin{lem}\label{2Olduniq}
Let $X$ be a compact metric space. Let $\lambda: \bigcup_{n=1}^{\infty}U(M_n(C(X)))\to \R_+$ be a map.
For any $\ep>0$ and any finite
subset ${\cal F}\subset C(X),$ there exist $\dt>0$ a finite subset
${\cal G}\subset C(X),$ a finite subset ${\cal P}\subset
\underline{K}(C(X))$ and a finite subset of unitaries ${\cal
U}\subset U_c(C(X)),$  a finite subset
$\{x_1, x_2,...,x_m\}\subset X$ and an integer $L>0$ satisfying the following condition:
if $\phi,\,\psi: C(X)\to A$ (for any  unital separable simple \CA\,$A$  with tracial rank at most one)  are
unital $\dt$-${\cal G}$-multiplicative \morp s such that
\beq\label{2Olduni-1}
[\psi]|_{\cal P}=[\phi]|_{\cal P}\tand\\\label{2Olduniq-2}
 {\rm
dist}(\langle\phi(u)\rangle,\langle \psi(u)\rangle)\le \lambda(u)
\eneq
for all $u\in {\cal U},$
 then, for any set of mutually orthogonal projections
 $p_1, p_2,...,p_m\in M_{K}(A)$ ($K\ge mL$) with
 $[p_i]\ge L[1_A],$ $i=1,2,...,m,$  and $\sum_{i=1}^m p_i=1_{M_K(A)},$ there is
 a unitary
$U\in M_{K+1}(A)$ such that
\beq\label{Olduni-2+}
\|U^*{\rm diag}(\phi(f), H(f))U-{\rm diag}(\psi(f),
H(f))\|<\ep
\eneq
for all $f\in {\cal F},$ where
$H(f)=\sum_{i=1}^m f(x_i)p_i$ for all $f\in C(X).$
\end{lem}

\begin{proof}
The proof follows  exactly the same way as that of \ref{Olduniq}. Note that  it follows from
\cite{Lnexp} that $M_j(A)$ has  exponential rank $1+\ep$ for every integer $j\ge 1.$  Also,  by
\cite{Lncltr1}, $A$ has stable rank one,  $K_0$-divisible rank  one, exponential divisible rank $E(L, k)=L/k+8\pi+1$
(see 6.10 of \cite{Lncltr1}, or derive it from \ref{torsionlength} directly). Thus Theorem 3.2 of \cite{GL} 
(together with \ref{Folk}) can also be applied as in the proof of \ref{Olduniq}. 

\end{proof}



\begin{lem}\label{smalltrace}
Let $X$ be a compact metric space  and let
 $s_1,s_2,...,s_m\in {\rm ker}\rho_{C(X)}$ be a finite subset.
 For any $d>0,$ there is $\dt>0$ and ${\cal G}\subset C(X)$ satisfying the following:
For any unital \CA\, $A$ with $T(A)\not=\emptyset$ and any unital $\dt$-${\cal G}$-multiplicative
\morp\, $L: C(X)\to A,$  one has that
\beq\label{smalltr-1}
\tau([L](s_j))<d\tforal \tau\in T(A),\,\,\,j=1,2,...,m.
\eneq

\end{lem}

\begin{proof}
There is an integer $m_0\ge 1$ and projections
$p_i, q_i\in M_{m_0}(C(X))$ such that
$$
[p_i]-[q_i]=s_i,\,\,\,i=1,2,...,m.
$$
Note that, for any $\tau\in T(A),$
$$
\tau(p_i)=\tau(q_i),\,\,\,i=1,2,...,m.
$$

Now suppose the lemma is false. Then there is $d_0>0,$ a sequence of unital \CA s $A_n$ with $T(A_n)\not=\emptyset,$ a sequence
of $\dt_n$-${\cal G}_n$-multiplicative \morp s $L_n: C(X)\to A_n$ with
$\sum_{n=1}^{\infty}\dt_n<\infty$ and $\cup_{n=1}^{\infty}{\cal G}_n$ is dense in $C(X)$ such that, for some $\tau_n\in T(A_n)$ and $1\le j\le m,$
\beq\label{smallt-2-2}
|\tau_n(L_n\otimes {\rm id}_{M_{m_0}})(p_j-q_j))|\ge d_0
\eneq
for all $n.$
Define $L: C(X)\to \prod_{n=1}^{\infty} A_n$
by $L(f)=\{L_n(f)\}$ for all $f\in C(X).$ Let $\pi: \prod_{n=1}^{\infty} A_n\to \prod_{n=1}^{\infty} A_n/\oplus_{n=1}^{\infty} A_n$ be the quotient map. Then $\pi\circ L: C(X)\to \prod_{n=1}^{\infty} A_n/\oplus_{n=1}^{\infty} A_n$ is a unital \hm.
Therefore, for any tracial state $t\in T(\prod_{n=1}^{\infty} A_n/\oplus_{n=1}^{\infty} A_n),$
\beq\label{smallt-2-}
t((\pi\circ L)\otimes {\rm id}_{M_{m_0}})(p_j-q_j))=0.
\eneq
Let $T_n: \prod_{n=1}^{\infty}A_n\to \C$ be defined by $T_n(a)=\tau_n(\pi_n(a))$ for all $a\in \prod_{n=1}^{\infty}A_n,$
where $\pi_n: \prod_{n=1}^{\infty} A_n\to A_n$ is the projection to the $n$-coordinate. Then $T_n$ is a tracial state.  Note that, for any $a\in \oplus_{n=1}^{\infty} A_n,$
\beq\label{smallt-2}
\lim_{n\to\infty} T_n(a)=0.
\eneq
Let $T$ be a limit point of $\{T_n\}.$ Then, by (\ref{smallt-2}), $T$ defines a tracial state
on $\prod_{n=1}^{\infty} A_n/\oplus_{n=1}^{\infty} A_n.$
Therefore, by (\ref{smallt-2-}),
$$
T(((\pi\circ L_{*0})\otimes {\rm id}_{M_{m_0}})(p_j-q_j))=0.
$$
It then follows that, for some subsequence $\{n_k\},$
$$
\lim_{k\to\infty} \tau_{n_k}((L_n\otimes {\rm id}_{M_{m_0}})(p_j-q_j))=0.
$$
This contradicts with (\ref{smallt-2-2}). The lemma follows.

\end{proof}

When $K_i(C(X))$ ($i=0,1$) is finitely generated, the following follows from
10.2 of \cite{LnApp}. We make a modification so it also applies to
the case that $K_i(C(X))$ ($i=0,1$)  is not finitely generated.

\begin{lem}\label{102}
Let $X$ be a compact metric space.
For any $\dt>0,$ any finite subset ${\cal G}\subset C(X)$ and any finite subset
${\cal P}\subset \underline{K}(C(X))$ for which the intersection of ${\rm ker}\rho_{C(X)}$ and the subgroup generated by ${\cal P} $ is generated by $g_1, g_2,...,g_k$ such that $(\dt, {\cal G}, {\cal P})$ is a $KL$-triple, there exists an integer
$N(\dt, {\cal G}, {\cal P})$ satisfies the following:

For any unital $\dt$-${\cal G}$-multiplicative \morp\, $L: C(X)\to B,$
where $B=M_n,$ or $B=M_n(C([0,1]))$ (for any integer $n\ge 1$)
with
$K=\max\{|L(g_i)|: i=1,2,...,k\},$
There exists an integer $N(K)\ge 1$ satisfying the following:
for any integer $N\ge N(K)/n,$ there exists a unital $\dt$-${\cal G}$-multiplicative \morp\, $L_0: C(X)\to M_{nN}\subset M_N(B)\subset $ such that
\beq\label{102-1}
{N(K)\over{{\rm max}\{K,1\}}} &\le & N(\dt, {\cal G}, {\cal P})\tand\\
([L]+[L_0])|_{\cal P} &= &[H]|_{\cal P}
\eneq
for some unital \hm\, $H: C(X)\to M_{1+N}(B)$ with finite dimensional range.

\end{lem}

\begin{proof}
Write
$C(X)=\lim_{n\to\infty}C(Y_n),$ where each $Y_n$ is a finite CW complex.  We use $\imath_m: C(Y_m)\to C(X)$ for the \hm\, given by the inductive limit system. Without loss of generality, we may assume that
${\cal G}\subset \imath_m(C(Y_m))$ for some $m\ge 1.$ Let ${\cal G}'\subset C(Y_m)$ be a finite subset such that
$\imath_m({\cal G}')={\cal G}.$
We may further assume that ${\cal P}\subset [\imath_m](\underline{K}(C(Y_m)))$ and ${\cal P}'\subset \underline{K}(C(Y_m))$ is a finite subset such that
 $[\imath_m]({\cal P}')={\cal P}.$ As defined, we also assume that
 $(\dt, {\cal G}')$ is a $KK$-triple for $C(Y_m).$

Let $\imath_m(C(Y_m))\cong C(Y),$ where $Y$ is a compact subset of $Y_m.$ Note that $\imath_m$ induces an embedding $\imath: C(Y)\to C(X).$
Denote by $s: X\to Y$ the surjective continuous map given by
$\imath,$ i.e, $\imath(f)(y)=f(s(y))$ for all $f\in C(Y).$

Suppose that $Y_m$ is a finite disjoint union of connected finite CW complexes $Z_1, Z_2,...,Z_l.$  One can choose $\xi_i\in Z_i$ such that
$\xi_i\in Y,$ $i=1,2,...,l.$  There are $s_1,s_2,...,s_k\in
\cup_{i=1}^l K_0(C(Z_i\setminus \{\xi_i\}))$ such that
$[\imath_m](s_i)=g_i,$ $i=1,2,...,k.$
Write $K_0(C(Z_i\setminus \{\xi_i\}))=\Z^{k(i)}\oplus G_i,$ where
$G_i$ is the torsion subgroup.
Since $K_0(\C)=\Z$ and $K_1(\C)=\{0\},$ any \hm\, from $K_i(C(Y_m))$ into $K_i(\C)$ vanishes on $Tor(K_i(C(Y_m))),$ $i=0,1.$
To simplify the notation, without loss of generality, we may assume
that $s_1, s_2,...,s_k$ are the standard generators
for $\oplus_{i=1}^l\Z^{(k(i)}.$
We may assume that ${\cal G}_i'\subset C(Z_i)$ is a finite subset
such that $\oplus_{i=1}^l {\cal G}_i'={\cal P}'$ and
${\cal P}_i'\subset \underline{K}(C(Y_m))$ is a finite subset such that
$\oplus_{i=1}^l {\cal P}_i'={\cal P}'.$

Applying 10.2 of \cite{LnApp} to each component $Z_i,$
we obtain an integer $N_i(\dt, {\cal G}_i', {\cal P}_i')$ given by
10.2 of \cite{LnApp}.
Let $N(\dt, {\cal G}, {\cal P})=\sum_{i=1}^l N_i(\dt, {\cal G}_i', {\cal P}_i').$

Now let $L: C(X)\to B$ be a $\dt$-${\cal G}$-multiplicative \morp\,.
 Put $L'=L\circ \imath_m: C(Y_m)\to M_n.$
 Let $\kappa\in Hom_{\Lambda}(\underline{K}(C(Y_m)), \underline{K}(B))$
be given by $L'.$
Let $\kappa_1\in KK(B, \C)$ be given by a point-evaluation, if $B=M_n(C([0,1]),$ or $\kappa_1$ is given by the identity, if $B=M_n.$
In either cases, one may view $\kappa_1$ is an identity on $\underline{K}(B)=\underline{K}(\C).$
Put
$$
K=\max\{|\kappa(s_j)|: j=1,2,...,k\}.
$$
It follows from 10.2 of \cite{LnApp} that there exists an integer $N(K)\ge 1$ and
unital $\dt$-${\cal G}'$-multiplicative \morp\, $L_0': C(Y_m)\to M_{K(n)}$ such that
\beq\label{102-n1}
{N(K)\over{{\rm max}\{K, 1\}}} &\le &  N(\dt, {\cal G}, {\cal P})
\andeqn\\
{[L_0']}|_{\underline{K}(C_0(Z_i\setminus\{\xi_i\}))} &= &
-(\kappa_1\times \kappa)|_{\underline{K}(C_0(Z_i\setminus \{\xi_i\}))},\,\,\,i=1,2,...,l.
\eneq
 If $N\ge N(K)/n,$ by adding some point-evaluation, if necessary,
we may assume that $L_0'$ maps $C(Y_m)$ into a \SCA\,  $D\cong M_{nN}$ and $D$ is a \SCA\, of $M_N(B)$ with $1_D=1_{M_N(B)}.$
 Then, viewing $L_0'$ maps $C(Y_m)$ into $M_N(B),$
 \beq\label{102-n2}
 \kappa+ [L_0']|_{\underline{K}(C_0(Z_i\setminus\{\xi_i\}))}=0.
 \eneq
There is a point-evaluation $h_0: C(Y_m)\to M_{n+N}(B)$ at  $\{\xi_1, \xi_2,...,\xi_l\}$  such that
\beq\label{102-n3}
[L'\oplus L_0']=[h_0].
\eneq
We may write
$$
h_0(f)=\sum_{i=1}^l f(\xi_i)p_i\tforal f\in C(Y_m),
$$
where $p_1, p_2,...,p_l$ are mutually orthogonal projections in $M_{N+1}(B).$
There is a unital \morp\, $L_0: C(X)\to M_N(B)$ such that
$$
L_0\circ \imath_m|_{C(Y_m)}=L_0'.
$$
Note that $L_0$ is $\dt$-${\cal G}$-multiplicative.
Define $H: C(X)\to M_{N+1}(B)$ by
$$
H(f)=\sum_{i=1}^l f(s(\xi_i))p_i\tforal f\in C(X).
$$
Then
$$
[L\oplus L_0]|_{\cal P}=[H]|_{\cal P}.
$$

\end{proof}

\begin{lem}\label{pointevaluation}
Let $X$ be a compact metric space, let $\ep>0,$ let ${\cal F}\subset
C(X)$ be a finite subset. There exists a finite subset
$\{x_1,x_2,...,x_m\}\subset X$  ($m\ge 1$) satisfying the following:
 for any unital \hm\,
$h_0:C(X)\to C([0,1], M_n)$  with finite dimensional range,
\beq\label{pointev-1}
\|(h_0\oplus h_1)(f)-\sum_{i=1}^{m}f(x_i)p_i\|<\ep\tforal f\in {\cal
F},
\eneq
where $h_1: C(X)\to C([0,1], M_{(m-1)n})$ is a unital \hm\, with finite
dimensional range and $\{p_1,p_2,...,p_{m}\}$ is a set of mutually orthogonal rank $n$ projections.

\end{lem}

\begin{proof}

 Let $\eta>0$
such that
$$
|f(x)-f(x')|<\ep/4\rforal f\in {\cal F},
$$
provided that ${\rm dist}(x, x')<\eta.$ Let $\{x_1,
x_2,...,x_m\}$ be  an $\eta$-dense subset of $X.$
Suppose that $h_0: C(X)\to C([0,1], M_n)$ is a unital \hm\, with finite dimensional range. Then there are $y_1, y_2,...,y_n\in X$ and
mutually orthogonal rank one projections $e_1, e_2,...,e_n$ such that
\beq\label{pointev-2}
h_0(f)=\sum_{i=1}^n f(y_i)e_i\rforal f\in C(X).
\eneq
Divide $\{y_1, y_2,...,y_n\}$ into $N$ disjoint subsets  $Y_1, Y_2,..., Y_N$
with $1\le N\le m$ such that
\beq\label{pointev-3}
{\rm dist}(y_i,x_j)<\eta,
\eneq
if $y_i\in Y_j.$  Let $E_j=\sum_{y_i\in Y_j} e_i$ and denote by $R_j$ the rank of $E_j,$ $j=1,2,...,N.$
Choose mutually orthogonal projections $q_1,q_2,..., q_m\in C([0,1], M_{(m-1)n})$ such that
rank of $q_j=n-R_j,$ $j=1,2,...,N$ and rank $q_j=n$ if $N<j\le n.$ Note $\sum_{j=1}^mq_j=1_{M_{(m-1)n}}.$
Define $h_1: C(X)\to C([0,1], M_{(m-1)n})$ by
$$
h_1(f)=\sum_{j=1}^n f(x_j)q_j\tforal f\in C(X).
$$
Let $p_j=E_j+q_j$ if $1\le j\le N$ and $p_j=q_j$ if $N<j\le m.$  Note that $p_j$ has rank $n$ for $j=1,2,...,m.$
One then checks that
$$
\|(h_0\oplus h_1)(f)-\sum_{k=1}^{m}f(x_k)p_k\|<\ep\rforal f\in {\cal F}.
$$

\end{proof}

\begin{lem}\label{Dig1}
Let $X$
be a compact metric space, let ${\cal P}\subset \underline{K}(C(X))$ be a finite subset and let
$G$ be the subgroup generated by ${\cal P}.$
Suppose $\Delta: (0, 1)\to (0,1)$ is  a nondecreasing function, $\eta>0$
and $\lambda_1, \lambda_2>0$ are given. Suppose that
 $g_1, g_2,...,g_k$ are generators of   $G\cap {\rm ker} \rho_{C(X)}.$

Suppose that $L,\, \Lambda: C(X)\to A$ (for some unital separable
simple \CA\, with tracial rank at most one) are two $ \dt$-${\cal
G}$-multiplicative \morp s for which $[L](g_i)$ and $[\Lambda](g_i)$  are well defined
($i=1,2,...,k$), where $\dt$ is a positive number and ${\cal G}$ is
a finite subset of $C(X),$
\beq\label{small-1}
|\tau([L](g_i))|<\sigma\tand |\tau([\Lambda](g_i))|<\sigma\tforal \tau\in T(A),\,\,\,i=1,2,...,k,
\eneq
for some $1>\sigma>0,$ and
\beq\label{small-1+}
\mu_{\tau\circ L}(O_r)\ge \Delta(r),\,\,\, \mu_{\tau\circ \Lambda}(O_r)\ge \
\Delta(r)\\\label{small-1++}
\eneq
for all $\tau\in T(A)$ and for all $r\ge \eta.$

 Then, for any $\ep>0$ and any finite
subset ${\cal F},$ any mutually orthogonal projections
$e_1,e_2,...,e_N,$
any finite subset ${\cal
H}\subset A$ and $R_0>1,$  there exists a projection $p\in A$ and a unital \SCA\,
$B=\oplus_{j=1}^m C(X_j, M_{r(j)}),$ where $X_j=[0,1],$ or $X_j$ is
a single point, with $1_B=p,$ mutually orthogonal
projections $e_1', e_2',...,e_N'\in B$ and a unital $(\dt+\ep)$-${\cal
G}$-multiplicative \morp s $\psi_1, \psi_2: C(X)\to B$ such that
\beq\label{small-1+1}
&&\|L(f)-[(1-p)L(f)(1-p)+\psi_1(f)]\|<\ep,\\\label{sm-1+2}
&&\|\Lambda(f)-[(1-p)\Lambda(f)(1-p)+\psi_2(f)]\|<\ep
\tforal f\in {\cal F},\\\label{sm-1+3}
&&\tau(1-p)<\eta,\,\,\,\tau(e_i')\ge \min\{(1-\lambda_1)\tau(e_i):\tau\in T(A)\} \tforal \tau\in T(A)\\\label{sm-1+4}
&& r(j)\ge R_0,\,\,\,j=1,2,...,k,\\\label{sm-1+5}
&& \|pe_ip-e_i'\|<\ep, t_{j,x}(e_i')\ge  \min\{(1-\lambda_1)\tau(e_i):\tau\in T(A)\} \\\label{sm-1+6}
&&|t_{j,x}([\psi_1](g_i))|<(1+\lambda_1)\sigma,
|t_{j,x}([\psi_2](g_i))|<(1+\lambda_1)\sigma\\\label{sm-1+7}
&&\,\,\,\,j=1,2,...,k\andeqn x\in X_j\\\label{small-1+10}
&&\hspace{-0.4in}\mu_{t_{j,x}\circ \psi_1}(O_r)\ge (1-\lambda_1)\Delta(r/2(1+\lambda_2)),\,\,\,\,
\mu_{t_{j, x}\circ \psi_2}(O_r)\ge (1-\lambda_1)\Delta(r/2(1+\lambda_2))
\eneq
 for all $r\ge 2(1+\lambda_2)\eta$
(We use  $t_{j,x}$ for $\tau_{j,x}\otimes {\rm Tr}_R$ on $B\otimes M_R,$
where  $t_{j,x}(f)=t\circ f(x)$ for all $f\in C(X_j, M_{r(j)}),$ for all $x\in X_j$ and $t$ is
the normalized trace on $M_{r(j)}$ and
${\rm Tr}_R$ is the standard trace on $M_R.$)

 Moreover, for any $\ep_0>0,$ one may assume that
$$
\|pa-ap\|<\ep_0\tand pap\in_{\ep_0} B\tforal a\in{\cal H}.
$$

If furthermore, $[L]|_{\cal P}=[\Lambda]|_{\cal P}$ (in $KK(C(X), A),$ then, by taking smaller $\dt$ and
larger ${\cal G},$ depending only on ${\cal P},$ one may further require that
$$
[\psi_1]|_{\cal P}=[\psi_2]|_{\cal P}\,\,\,{\rm in}\,\,\, KK(C(X), B).
$$
\end{lem}

\begin{proof}
The proof is a modification  of  that of Lemma 9.7 of \cite{LnApp}.  We repeat many arguments here.  Let $p_j, q_j\in
M_R(C(X))$ such that
$$
[p_j]-[q_j]=g_j,\,\,\,j=1,2,...,k,
$$
for some integer $R\ge 1.$
There exists of a sequence of projections $p_n\in A$ such that
\beq\label{small-4}
\lim_{n\to\infty}\|cp_n-p_nc\|=0\tforal c\in A,
\eneq
and there exists a sequence of \SCA s
$B_n=\oplus_{j=1}^{m(n)}C(X_{j,n}, M_{r(j,n)})$ (where
$X_{j,n}=[0,1]$ or $X$ is a single point) with $1_{B_n}=p_n$ such
that
\beq\label{sm-5}
&&\lim_{n\to\infty}{\rm dist}(p_ncp_n, B_n)=0\andeqn\\
&&\lim_{n\to\infty} \sup_{\tau\in T(A)}\{\tau(1-p_n)\}=0.
\eneq
Moreover, by 3.3 of \cite{Lncltr1}, we may also assume
that $r(j,n)\ge R_0$ for all $j.$
For sufficiently large $n,$ there exists a \morp\, $L_n': p_nAp_n
\to B_n$ such that
$$
\lim_{n\to\infty}\|L_n'(a)-p_nap_n\|=0 \tforal a\in A
$$
(see 2.3.9 of \cite{Lnbk}).
There are (see 2.55 and 2.5.6 of \cite{Lnbk}) mutually orthogonal projections
$e_{i,n}\in B_n$ such that
$$
\lim_{n\to\infty}\|p_ne_ip_n-e_{i,n}'\|=0,\,\,\,i=1,2,...,N.
$$
We have
\beq\label{sm-5+}
\lim_{n\to\infty}\|L(f)-[(1-p_n)L(f)(1-p_n)+L_n'\circ
L(f)]\|&=&0\andeqn\\
\lim_{n\to\infty}\|\Lambda(f)-[(1-p_n)\Lambda(f)(1-p_n)+L_n'\circ
\Lambda(f)]\|&=&0
\rforal f\in C(X).
\eneq
Define  $L_{n,R}': M_R(A)\to M_R(A)$ by $L_n'\otimes {\rm
id}_{M_R}$ and $L_R: M_R(C(X))\to M_R(A)$ by
$L_R=L\otimes {\rm id}_{M_R}.$
Suppose that  there exists a subsequence
$\{n_k\},$ $\{j_k\}$ and $\{x_k\}\in [0,1]$ such that
\beq\label{sm-6}
|t_{j_k,x_k} (L_{n_k,R}'\circ L_R(p_i-q_i))|\ge (1+\lambda_1)\sigma
\eneq
for all $k.$
Define a state $T_k: A\to \C$ by $T_k(a)=t_{j_k, x_k}(a),$
$k=1,2,....$ Let $T$ be a limit point. Note $T_k(1_A)=1.$
Therefore $T$ is a state on $A.$ Then, by (\ref{sm-6}),
\beq\label{TTL-8+}
|T([L](g_i))|\ge (1+\lambda_1)\sigma.
\eneq
However, it is easy to check that $T$ is a tracial state. This
contradicts with (\ref{small-1}).
Put $\psi_1=L'_{n}\circ L$ and $\psi_2=L'_{n}\circ \Lambda$ for some
large $n.$ Then we have shown (for the choice of large $n$)
that (\ref{small-1+1}), (\ref{sm-1+2}) and (\ref{sm-1+7}) hold.

A similar argument shows that, for some sufficiently large $n,$
$$
t_{j,x}(e_{i,n}')\ge \min\{(1-\lambda_1)\tau(e_i): \tau\in T(A)\},
$$
$i=1,2,...,N,$ for all $x\in X_j$ and $j=1,2,...,m(n).$

Moreover,  a similar argument shows that, for any finitely many  $f_1, f_2,....,f_N\in C(X)$
such that $0<f_i\le 1,$ $i=1,2,...,N,$ we may assume (by choosing large $n$) that
\beq
t_{j,x}\circ \psi_1(f_k)\ge (1-\lambda_1/2) \min\{\tau(L(f_k)): \tau\in T(A)\}\andeqn\\
t_{j,x}\circ \psi_2(f_k)\ge (1-\lambda_1/2) \min\{\tau(\Lambda(f_k)): \tau\in T(A)\}
\eneq
for all $x\in X_j,$ $j=1,2,...,m.$ By choosing sufficiently many (but finitely many) $f_j\,'s,$  using the argument
in the proof of \ref{1keeptrace}, we may assume that
\beq
\mu_{t_{j,x}\circ \psi_i}(O_r)\ge (1-\lambda_1)\Delta(\eta/2(1+\lambda_2))
\eneq
for all $r\ge 2(1+\lambda_2)\eta$ and for all $x\in X,$ $j=1,2,...,m$ and $i=1,2.$

So the first part of the lemma follows by choosing
$B$ to be $B_n,$ $p$ to be $p_n$ and $\psi_1$ to be $L_n'\circ L$
and $\psi_2$ to be $L_n'\circ \Lambda$ for
some sufficiently large $n.$ Note, by (\ref{small-4}) and (\ref{sm-5}), for any $\ep_0>0$ and any
 finite subset ${\cal H},$
we can assume that
$$
\|pa-ap\|<\ep_0\andeqn pap\in_{\ep_0} B
$$
for all $a\in {\cal H}.$

To see the last part of the lemma holds, taking a commutative
\CA\, $C$ and considering the maps $L\otimes {\rm id}_{M_l(C)}$ and
$\Lambda\otimes {\rm id}_{M_l(C)}$ from $C(X)\otimes M_l(C)$ into
$A\otimes M_l(C).$
There is $K_0\ge 1$ such that, if $x\in Tor(K_i(C(X)))\cap G,$  then
$K_0x=0.$
Let $C_1, C_2,...,C_{K_0!}$ be unital commutative \CA\, such that
$K_0(C_j)=\Z\oplus \Z/j\Z$ and $K_1(C_j)=\{0\},$ $j=1,2,...,K_0!.$
Let $l\ge 1$ be an integer  such that a set of generators of
$K_0(C(X)\otimes C_j)\cap G$ and $K_1(C(X)\otimes C_j)\cap G$ can be represented by
projections and unitaries in $M_l(C(X)\otimes C_j)=C(X)\otimes C_j\otimes M_l,$
$j=1,2,...,K_0!.$

Choose $0<\ep_1<\ep$ and a finite subset ${\cal F}_1\supset {\cal F}$
(which depends on $K_0$ and $l$ above).
Then one applies the first part of the lemma for this $\ep_1$ and ${\cal F}_1.$
For a finite subset of projections
$E_1, E_2,...,E_K\in C(X)\otimes C_j\otimes M_l,$
if in addition that $[L]|_{\cal P}=[\Lambda]|_{\cal P}$ (with sufficiently small $\dt$ and
sufficiently large ${\cal G}$),
there are partial isometries $W_1, W_2,...,W_K\in A\otimes  C_j\otimes M_{l+R}$ for some integer $R\ge 0$ such that
$$
W_iW_i^*=E_i'\oplus {\rm id}_{M_R(C_j)} \andeqn W_i^*W_i=E_i''\oplus {\rm id}_{M_R(C_j)},
$$
where $E_i'$ and
$E_i''$ are two projections such
that
$$
\|E_i'-L\otimes {\rm id}_{M_l(C_j)}(E_i)\|<1/16\andeqn \|E_i''-\Lambda\otimes {\rm id}_{M_l(C_j)}(E_i)\|<1/16,
$$
$i=1,2,...,K.$

Fix $\ep_0>0.$
One then chooses a large ${\cal H}$ so that
$$
\|pa-ap\|<\ep_0\andeqn pap\in_{\ep_0} B
$$
imply that
\beq
&&\hspace{-0.4in}\|(p\otimes {\rm id}_{M_{l+R}(C_j)})W_i-W_i(p\otimes {\rm id}_{M_R(C_j)})\|<\ep_1,\\
 &&\hspace{-0.4in}\|(p\otimes {\rm id}_{M_{l+R}(C_j)})(E_i'\otimes {\rm id}_{M_R(C_j)})-(E'_i\otimes {\rm id}_{M_R(C_j)})(p\otimes {\rm id}_{M_{l+R}(C_j)})\|<\ep_1\andeqn\\
 &&\hspace{-0.4in}\|(p\otimes {\rm id}_{M_{l+R}(C_j)})(E_i''\otimes {\rm id}_{M_R(C_j)})-(E_i''\otimes {\rm id}_{M_R(C_j)})(p\otimes {\rm id}_{M_{l+R}(C_j)})\|<\ep_1, \\
 &&\hspace{-0.4in}\|(p\otimes {\rm id}_{M_l(C_j)})E_i'-E_i'(p\otimes {\rm id}_{M_l(C_j)})\|<\ep_1\andeqn\\
 &&\hspace{-0.4in}\|(p\otimes {\rm id}_{M_l(C_j)})E_i''-E_i''(p\otimes {\rm id}_{M_l(C_j)})\|<\ep_1
 \eneq
 as well as
 \beq
 &&(p\otimes {\rm id}_{M_l(C_j)})E_i'(p\otimes {\rm id}_{M_l(C_j)})\in_{\ep_1}B\otimes M_l(C_j),\\
  &&(p\otimes {\rm id}_{M_l(C_j)})E_i''(p\otimes {\rm id}_{M_l(C_j)})\in_{\ep_1}B\otimes M_l(C_j),\\
&& (p\otimes {\rm id}_{M_{l+R}(C_j)})W_i(p\otimes {\rm id}_{M_{l+R}(C_i)})(p\otimes {\rm id}_{M_{l+R}(C_j)})\in_{\ep_1} B\otimes M_{l+R}(C_j),\\
 && (p\otimes {\rm id}_{M_{l+R}(C_j)})(E_i' \otimes {\rm id}_{M_R(C_j)} )(p\otimes {\rm id}_{M_{l+R}(C_j)})\in _{\ep_1} B\otimes M_{l+R}(C_j)\andeqn\\
&& (p\otimes {\rm id}_{M_{l+R}(C_j)})(E_i'' \otimes {\rm id}_{M_R(C_j)} )(p\otimes {\rm id}_{M_{l+R}(C_j)})\in _{\ep_1} B\otimes M_{l+R}(C_j).
 \eneq
 It follows that (with small $\ep_1$)
 there are projections $e_i', e_i''\in B\otimes M_l(C_j)$ such that
 $$
 \|e_i'-(p\otimes {\rm id}_{M_l(C_j)})E_i'(p\otimes {\rm id}_{M_l(C_j)})\|<2\ep_1,\\
  \|e_i''-(p\otimes {\rm id}_{M_l(C_j)})E_i''(p\otimes {\rm id}_{M_l(C_j)})\|<2\ep_1\andeqn
  $$
$$
[e_i]=[e_i']\,\,\,{\rm in}\,\,\, K_0(B).
$$
Therefore, one has
$$
[\psi_1]([E_i])=[\psi_2]([E_i]),\,\,\,i=1,2,...,K.
$$
 From this, one concludes that one may require that
 $$
 [\psi_1\otimes {\rm id}_{C_j)}]|_{K_0(C(X)\otimes C_j)\cap G}=[\psi_2\otimes {\rm id}_{C_j}]|_{K_0(C(X)\otimes C_j)\cap G},
 $$
 $j=1,2,...,K_0!.$ A similar argument shows that one may also require that
 $$
 [\psi_1\otimes {\rm id}_{C_j}]|_{K_1(C(X)\otimes C_j)\cap G}=[\psi_2\otimes {\rm id}_{C_j}]|_{K_1(C(X)\otimes C_j)\cap G},
 $$
 $j=1,2,...,K_0!.$ It follows that one may require that
 $$
 [\psi_1]=[\psi_2]\,\,\,{\rm in}\,\,\, KK(C(X), B).
 $$

\end{proof}



\section{The main results}

\begin{lem}\label{MT2}
Let $X$ be a compact metric space, let $\ep>0,$ $\ep_0>0,$ let $\{x_1,
x_2,...,x_m\}\subset X,$ let ${\cal F}\subset C(X)$ be a finite
subset and let $\Delta: (0,1)\to (0,1)$ be an increasing map with
$\lim_{t\to 0} \Delta(t)=0.$
Let ${\cal P}\subset \underline{K}(C(X))$ be a finite subset,
$K\ge 1$ be an integer and let $\eta_0>0.$  Then, there exists $\eta>0,$ $\dt>0,$ a finite subset
${\cal G}\subset C(X)$  satisfying the following: For any unital
$\dt$-${\cal G}$-multiplicative \morp s $L, \Lambda: C(X)\to A$ for some unital separable simple
\CA\, $A$ with tracial rank at most one
for which
\beq\label{MT2-1}
[\Lambda]|_{\cal P}=[L]|_{\cal P}
\tand
\mu_{\tau\circ L}(O_r), \mu_{\tau\circ L}(O_r)\ge \Delta(r)
\eneq
for all open balls $O_r$ with radius $1>r\ge \eta, $ and, for any $\ep_{00}>0$ and any finite subset ${\cal H}\subset A,$
 there exist mutually orthogonal projections\\ $P_1, P_2, P_3,p_1, p_2,...,p_m\in A$ with
 $P_1\oplus P_2\oplus P_3\oplus \sum_{i=1}^m p_i=1_A,$
 \beq\label{MT2-2}
 \tau(P_3)>1-\ep_0\tforal \tau\in T(A)\tand\\\label{MT2-2+}
 K[P_1\oplus P_2]\le  [p_i],\,\,\,\,\,\,i=1,2,...,m,
 \eneq
 and there exists a unital $\ep$-${\cal F}$-multiplicative \morp\,
 $\psi: C(X)\to P_2BP_2$ whose range contained in a finite dimensional \SCA,
 unital $\ep$-${\cal F}$-multiplicative \morp s $H_1, H_2: C(X)\to  P_3BP_3\subset P_3AP_3,$ where $P_2, P_3,$ $ p_1,p_2,...,p_m\in B,$
 $B=\oplus_{j=1}^N B_j,$ and
 $B_j=C(X_j, M_{r(j)})$ ($X_j=[0,1],$ or $X_j$ is a point) with
 $$
 [H_1]|_{\cal P}=[H_2]|_{\cal P}=[h_0]|_{\cal P},
 $$
 for some unital \hm\, $h_0: C(X)\to C,$
 where $C=P_3BP_3,$ $C=\oplus_{j=1}^NC_j$ and $ C_j=C(X_j, M_{r'(j)}),$
 and a unitary $W\in A$ such that
\beq\label{MT2-3}
&&\|L(f)-[(P_1L(f)P_1\oplus H_1(f)\oplus \psi(f)\oplus \sum_{i=1}^m f(x_i)p_i]\|<\ep\\
&&\hspace{-0.4in}\|{\rm Ad}\, W \circ \Lambda(f)-[P_1({\rm Ad}\, W\circ \Lambda)(f)P_1\oplus H_2(f)\oplus \psi(f)\oplus \sum_{i=1}^m f(x_i)p_i]\|<\ep\\
&&\hspace{-1.2in}
\tforal
f\in {\cal F},\,\,\,\,\,\,\hspace{0.4in}\\
&& \mu_{t_{j,x}\circ H_i}(O_r)\ge \Delta(r/3)/2\tand
t(P_2+\sum_{i=1}^mp_i)<\ep_0
\eneq
for all $r\ge \eta_0,$ $x\in X_j,$ where $t_{j,x}$ is the
composition of the point-evaluation at $x$ and the normalized trace on $M_{r'(j)},$ $j=1,2,...,k,$ and for all $t\in T(B),$ and
\beq
\|P_1a-aP_1\|<\ep_{00},\,\, \, (1-P_1)a(1-P_1)\in_{\ep_{00}} B \rforal a\in{\cal H}\cup L( {\cal F})\cup \Lambda({\cal F}),
\eneq
where $1_B=1-P_1,$
Moreover,
\beq\label{MT2-3+1}
[P_1LP_1]|_{\cal P}=[P_1({\rm Ad}\, W\circ \Lambda) P_1]|_{\cal P}.
\eneq
\end{lem}

\begin{proof}


Let $\ep, \ep_0,$ $\{x_1, x_2,...,x_m\}\subset X,$ a finite subset ${\cal F}\subset X,$ a finite subset ${\cal P}\subset \underline{K}(C(X)),$ $\Delta,$ $K\ge 1$ and $\eta_0>0$ be as described.
We may assume that $(\ep, {\cal F}, {\cal P})$ is a $KL$-triple  for $C(X)$
 and $0<\ep_0, \ep<1/16.$

Let $\dt_1>0$ (in place of $\dt$), ${\cal G}_1\subset C(X)$ be a finite subset (in place of ${\cal G}),$
${\cal P}_1\subset {\underline{K}}(C(X))$ (in place of ${\cal P}$)
be a finite subset
and $K_1$ be an integer
(in place of $L$)
for $\min\{\ep/16, \ep_0/16\}$ and ${\cal F}$
required by \ref{Foldunique}.

We may also assume, without loss of generality, that
 ${\cal P}\subset {\cal P}_1$ and $(\dt_1, {\cal G}_1, {\cal P}_1)$ forms a $KL$-triple.
We may further assume that, if $L', L'': C(X)\to C$ (for any unital \CA\, $C$) are $\ep$-${\cal G}_1$-multiplicative
\morp s and
$$
\|L'(f)-L''(f)\|<\ep\tforal f\in {\cal G}_1,
$$
then
$$
[L']|_{{\cal P}_1}=[L'']|_{{\cal P}_1}.
$$


Let $G$ be the subgroup generated by ${\cal P}_1$ and let
$s_1,s_2,..., s_{k_0}$ be a set of generators of $G\cap {\rm ker}\rho_{C(X)}.$

Let $\ep_2=\min\{\ep/64, \ep_0/64, \dt_1/2, \dt'/2\}$ and
${\cal G}_2={\cal F}\cup {\cal G}_1\cup {\cal G}'.$


Let $N_1=N(\ep_2,{\cal G}_2, {\cal P}_1)$ be as in \ref{102}
where
$\dt$ is replaced by $\ep_2,$ ${\cal G}$ is replace by ${\cal G}_2$ and
${\cal P}$ is replaced by ${\cal P}_1.$

Let $N_2$ (in place $m$) and $\{y_1,y_2,...,y_{N_2}\}$ (in place of
$\{x_1,x_2,...,x_m\}$) be as in \ref{pointevaluation} for
$\ep_2$ (in place $\ep$) and ${\cal G}_2$ (in place of ${\cal F}$). One may assume that $N_2>m$ and
$y_j=x_j,$ $j=1,2,...,m.$

Choose $\eta'>0$ satisfying the following:
$$
|f(x)-f(x')|<\ep_2/16
$$
for all $f\in {\cal G}_2,$ if ${\rm dist}(x, x')<2\eta'.$ Moreover, we may assume
that
$$
O_{4\eta'}(y_j)\cap O_{4\eta'}(y_i)=\emptyset\,\,\, {\rm if}\,\,\, i\not=j.
$$

Choose $\eta''>0$ such that  $\eta''<\eta_0/4$ and
$$
\Delta(\eta'')<{\Delta(\eta_0/4)\over{256N_2}}.
$$
Choose
$$
\eta=\min\{\eta_1/4, \eta_2/4, \eta'/4, \eta''/4\}.
$$

Let $\dt_2>0$ (in place of $\dt$) and let ${\cal G}_3\subset C(X)$ be a finite subset required by Lemma 9.6 of \cite{LnApp}
for $\ep_2/2$ (in place of $\ep$), ${\cal G}_2$ (in place of ${\cal F}$), $\eta$ and $1/256$ (in place of $r$).

Choose ${\bar K}$ to be an integer which is greater than
the integer $K$  given by this lemma. We may assume that $K>4.$
Choose $d>0$
such that
\beq\label{MT2-10}
16d{\bar K}N_1N_2^2(K_1+1)<\ep_0\Delta(\eta)/2^{11}.
\eneq

It follows from \ref{smalltrace}  that there are $\dt_3>0$ and a finite subset ${\cal G}_4\subset C(X)$ such that, for any unital $\dt_3$-${\cal G}_4$-multiplicative \morp\, $\Psi: C(X)\to C$ (for any unital   \CA\, $C$ with $T(C)\not=\emptyset$),
$$
|\tau([\Psi](s_i))|<d/8\tforal \tau\in T(C).
$$


Let $\dt=\min\{\ep_2/4, \dt_2/4, \dt_3/4, \dt_4\}$ and ${\cal G}=\cup_{i=1}^5{\cal G}_i.$

%

Now suppose that $L, \Lambda: C(X)\to A,$ where $A$ is a unital simple \CA\, with tracial rank at most one,
satisfy the assumptions of the theorem for the above $\dt,$  ${\cal G},$  $\eta$ and $\Delta.$

By Theorem 9.6 of \cite{LnApp} and by the choice of $\eta,$ there exists projections $Q_1, Q_2\in A$
and two sets of mutually orthogonal projections
$\{E_1, E_2,...,E_{N_2}\}$ in $(1-Q_1)A(1-Q_1)$ and $\{E_1', E_2',...,E_{N_2}'\}$ in $(1-Q_2)A(1-Q_2)$ such that
$\sum_{i=1}^{N_2}E_i=1-Q_1,$ $\sum_{i=1}^{N_2}E_i'=(1-Q_2),$
\beq\label{MT2-11}
\|L(g)-[Q_1L(g)Q_1\oplus \sum_{i=1}^{N_2}g(y_i)E_i)]\|<\ep_2/2,\\
\|\Lambda(g)-[Q_2\Lambda(g)Q_2\oplus \sum_{i=1}^{N_2}g(y_i)E_i']\|<\ep_2/2
\eneq
for all $g\in {\cal G}_2,$
\beq\label{MT2-12}
&&\Delta(\eta)>\tau(E_i)\ge (1-1/2^{12})\Delta(\eta)\andeqn\\ &&\Delta(\eta)>\tau(E_i')\ge (1-1/2^{12})\Delta(\eta),\,\,\,i=1,2,...,N_2,
\eneq
for all $\tau\in T(A).$
Since $A$ has tracial rank at most one (see Lemma 9.9 of \cite{LnApp}), there is, for each $i,$
a projection $q_i\le E_i$ such that $[q_i]\le [E_i']$ and
$$
\tau(q_i)\ge (1-1/2^{11})\Delta(\eta)\tforal \tau\in T(A).
$$
Let $q_i'\le E_i'$ such that $[q_i]=[q_i'],$ $i=1,2,...,N_2.$
Let $Q_0=1-\sum_{i=1}^{N_2}q_i$ and $Q_0'=1-\sum_{i=1}^{N_2}q_i'.$
Then, we have that
\beq
\|L(g)-[Q_1L(g)Q_1\oplus \sum_{i=1}^{N_2}g(y_i)(E_i-q_i)+\sum_{i=1}^{N_2}g(y_i)q_i]\|<\ep_2/2\\
\|\Lambda(g)-[Q_2\Lambda(g)Q_2\oplus \sum_{i=1}^{N_2}g(y_i)(E_i'-q_i')+\sum_{i=1}^{N_2}g(y_i)q_i'\|<\ep_2/2
\eneq
for all $g\in {\cal G}_2.$
Since $[q_i]=[q_i'],$ $i=1,2,..., N_2,$ there is a unitary $W_1\in A$ such that
$$
W_1^*q_i'W_1=q_i,\,\,\,i=1,2,...,N_2\andeqn W_1^*Q_0'W_1=Q_0.
$$
Define $L_1: C(X)\to Q_0AQ_0$ by
$L_1(f)=Q_1L(f)Q_1\oplus \sum_{i=1}^{N_2}f(y_i)(E_i-q_i)$ for all $f\in C(X).$
Define $\Lambda_1: C(X)\to Q_0AQ_0$ by
$\Lambda_1(f)=W_1^*(Q_2\Lambda(f)Q_1\oplus \sum_{i=1}^{N_2}f(y_i)(E_i'-q_i'))W_1$ for all
$f\in C(X).$ Then $L_1$ and $\Lambda_1$ are $\ep_2$-${\cal G}_4$-multiplicative and
\beq\label{MT2-13}
\|L(g)-[L_1(g)+\sum_{i=1}^{N_2}g(y_i)q_i]\|<\ep_2/2\andeqn\\
\|{\rm Ad}\, W_1\circ \Lambda(g)-[\Lambda_1(g)+\sum_{i=1}^{N_2}g(y_i)q_i]\|<\ep_2/2
\eneq
for all $g\in {\cal G}_2.$
Note that
\beq\label{MT2-15}
[L_1](s_i)=[L](s_i)\andeqn [\Lambda_1](s_i)=[\Lambda](s_i),\,\,\,\,i=1,2,...,k_0.
\eneq
We compute that
$$
\mu_{\tau\circ Q_1LQ_1}(O_r)\ge \Delta(r)-N_2\Delta(\eta)\ge   255\Delta(r)/256
$$
for all $\tau\in T(A)$ and $r\ge \eta_0/4.$  It follows that
\beq\label{nMT2-15}
\mu_{\tau\circ L_1}(O_r)\ge 255\Delta(r)/256
\eneq
for all $\tau\in T(A)$ and $r\ge \eta_0/4.$
Similarly,
\beq\label{nMT2-16}
\mu_{\tau\circ \Lambda_1}(O_r)\ge 255\Delta(r)/256
\eneq
for all $\tau\in T(A)$ and $r\ge \eta_0/4.$


Let $\theta<
{\ep_0\Delta(\eta/2)\over{8N_1N_2{\bar K}}}.$
Let $\ep_{00}>0$ and let ${\cal H}\subset A$ be a finite subset. Define
$$
{\cal H}_1={\cal H}\cup L({\cal G})\cup L_1({\cal G})\cup \Lambda_1({\cal G})\cup {\rm Ad}\, W_1\circ \Lambda({\cal G})\cup\{P_1, q_1,q_2,...,q_{N_2}\}.
$$
Let $0<\dt_0<\min\{\ep_2/2, \dt/4, \ep_{00}\}$ and put
\beq\label{MT2-14--}
L'(f)=L_1(f)\oplus \sum_{i=1}^{N_2}f(y_i)q_i\andeqn\\
\Lambda'(f)=\Lambda_1(f)\oplus \sum_{i=1}^{N_2}f(y_i)q_i
\eneq
for all $f\in C(X).$
Since $A$ has tracial rank at most one, by \ref{Dig1},
there exists a projection
$Q_3\in A$  and $B=\oplus_{j=1}^{N} B_j$ with $1_B=Q_3,$
where $B_j=M_{r(j)}(C(X_j))$ and $X_j=[0,1],$ or $X_j$ is a point, such that
\beq\label{MT2-14}
\|Q_3 a-aQ_3\|<\dt_0,\,\,\, Q_3aQ_3\in_{\dt_0} B\tforal a\in {\cal H}_1,\\\label{MT2-14+1}
\|L'(f)-[(1-Q_3)L'(f)(1-Q_3)\oplus L_3(f)]\|<\ep_2/2N_2,\\
\|\Lambda'(f)-[(1-Q_3)\Lambda(f)(1-Q_3)\oplus \Lambda_3(f)]\|<\ep_2/2N_2\rforal f\in {\cal G},\\\label{MT2-14+1+1}
 \tau(1-Q_3)<\theta\rforal \tau\in T(A),\\
|T_{j,x}([L_3](s_i))|<(1+1/128)(d/8), \,\,\,
|T_{j,x}([\Lambda_3](s_i)|<(1+1/128)(d/8)\\\label{MT2-14+2}
\mu_{T_{j,x}\circ L_3}(O_r)\ge 3\Delta(r/3)/4,\,\,\,\mu_{T_{j,x}\circ \Lambda_3}(O_r)\ge 3\Delta(r/3)/4
\eneq
for all $r\ge \eta_0,$  and
for all $j$ and $x\in X_j,$ where
 $T_{j,x}$ is the normalized trace of $M_{r(j)}$ at $x\in X_j,$ and
\beq\label{MT2-rank}
r(j)>{2^{15}{\bar K}N_1N_2^2(K_1+1)\over{\ep_0 \Delta(\eta/2)}},\,\,\,j=1,2,...,N.
\eneq
 Moreover,
 \beq\label{MT2-16-1}
 [L_3]|_{{\cal P}_1}=[\Lambda_3]|_{{\cal P}_1}\,\,\,{\rm in}\,\,\, KL(C(X), B).
  \eneq
Therefore
\beq\label{MT2-16}
\|Q_3L'(f)Q_3-L_3(f)\|<\ep_2/2N_2\andeqn
\|Q_3\Lambda'(f)Q_3-\Lambda_3(f)\|<\ep_2/2N_2
\eneq
for all $f\in {\cal G}.$
By \ref{Dig1},  we further obtain mutually orthogonal projections
$e_1,e_2,...,e_{N_2}\in B$ such that
\beq\label{MT2-17}
\|L_3(f)-[EL_3(f)E\oplus \sum_{i=1}^{N_2} f(y_i)e_i]\|<\ep_2/2\\
\|\Lambda_3(f)-[E\Lambda_3(f)E\oplus \sum_{i=1}^{N_2}f(y_i)e_i]\|<\ep_2/2
\eneq
for all $f\in {\cal G},$ where $E=1_B-\sum_{i=1}^{N_2}e_i.$ Moreover we require that
\beq\label{MT2-18}
\Delta(\eta)/2>T_{j,x}(e_i)\ge (1-1/250)\Delta(\eta),\,\,\,\Delta(\eta)/2>\tau(e_i)\ge (1-1/250)\Delta(\eta)
\eneq
for all $x\in X_j,$  where $T_{j,x}$ is the normalized trace evaluated at $x,$ and for all $\tau\in T(A),$ $j=1,2,...,N_2.$

Define $L_4=EL_3E$ and $\Lambda_4=E\Lambda_3E.$
We compute, by (\ref{MT2-14+2}) and (\ref{MT2-18})
\beq\label{MT2-18+}
\mu_{T_{j,x}\circ L_4}(O_r) &\ge& 3\Delta(r/3)/4-N_2\Delta(\eta)\ge \Delta(r/3)/2\andeqn \\\label{MT2-18+2}
\mu_{T_{j, x}\circ \Lambda_4}(O_r)&\ge & \Delta(r/3)/2.
\eneq
for all $r\ge \eta_0$
and for all $x\in X_j$ and $j=1,2,...,N.$

We compute that
\beq\label{NMT-02}
[L_4](s_j)=[\Lambda_4](s_j)=[L](s_j)=[\Lambda](s_j),\,\,\,j=1,2,...,k_0.
\eneq
Put $C'_j=E(M_{r(j)}(C(X_j))E=M_{r''(j)}(C(X_j)),$ where $r''(j)\le r(j),$
put $L_{4,j}=\pi_j\circ L_4$ and $\Lambda_{4,j}=\pi_j\circ \Lambda_4,$ where
$\pi_j: EBE\to EB_jE$ is the projection, $j=1,2,...,N.$

Let
$$
d_j=\max\{|T_{j,x}([L_4](s_k))|: k=1,2,...,k_0\},
$$
$j=1,2,...,N.$
Note that $d_j\le d/4,$ $j=1,2,...,N.$

It follows from \ref{102} that
there exist unital $\ep_2$-${\cal G}_2$-multiplicative \morp s
$L_{0,j}, {\bar L_{0,j}}: C(X)\to M_{J_j}(C(X_j))$ whose ranges are contained in  finite dimensional \SCA s, where
\beq\label{MT2-23}
J_j=d_jN_1r(j)\le dN_1 r(j)
\eneq
such that
\beq\label{nMT-01}
[L_{4,j}\oplus L_{0,j}]|_{{\cal P}_1}=[H_{0,j}]|_{{\cal P}_1}\andeqn [L_{0,j}\oplus {\bar L}_{0,j}]|_{{\cal P}_1}=[h_{0,j}]|_{{\cal P}_1}
\eneq
for some unital \hm s $H_{0,j}: C(X)\to M_{r''(j)+J_j}(C(X_j))$ and $h_{0,j}: C(X)\to M_{2J_j}(C(X_j))$
with finite dimensional range.

By applying \ref{Foldunique}, we obtain a unital \hm\, $h_{1,j}: C(X)\to M_{2J_jK_1}(C(X_j))$
with finite dimensional range and
a unital \hm\, $H_{1,j}: C(X)\to M_{2J_j(K_1+1)}(C(X_j))$ with finite dimensional range
such that
\beq\label{nMT-03}
\|(L_{0,j}\oplus {\bar L}_{0,j}\oplus h_{1,j})(f)-H_{1,j}(f)\|<\min\{\ep/16, \ep_0/16\}\tforal g\in {\cal F}.
\eneq
It follows from \ref{pointevaluation} that there are mutually orthogonal rank $2J_j(K_1+1)$
projections $q_{i,j}'\in M_{N_22J_1(K_1+1)}(C(X_j))$  and a unital \hm\, $h_{2, j}: C(X)\to
M_{(N_2-1)(2J_1(K_1+1)}(C_j)$ with finite dimensional range such that
\beq\label{nMT-04-}
\|H_{1,j}(f)\oplus h_{2,j}(f)-\sum_{i=1}^{N_2} f(y_i)q_{i,j}'\|<\ep_2\tforal f\in {\cal F},
\eneq
$j=1,2,...,N.$

There is, for each $i$ and $j,$ by (\ref{MT2-18}) and (\ref{MT2-10}), a projection $p_{i,j}'\le e_i$ such that
\beq\label{MT2-19}
\ep_0/128N_2\ge  T_{j,x}(p_{i,j}')\ge \ep_0/256N_2
\ge 16d{\bar K}N_1N_2(K_1+1)
\eneq
for $x\in X_j,$ $j=1,2,...,N$ and $i=1,2,...,N_2.$

Put $L_0=\sum_{j=1}^NL_{0,j},$ ${\bar L}_0=\sum_{j=1}^N{\bar L}_{0,j},$
$h_i=\sum_{j=1}^N h_{i,j},$ $i=1,2.$
There is a projection $q_{i,j}\le p_{i,j}'$ in $B_j$ such that
$[q_{i,j}]=[q_{i,j}']$ in $K_0(B_j),$ $j=1,2,...,N.$
Put $p_i''=\sum_{j=1}^N q_{i,j},$ $i=1,2,...N_2.$ Then
\beq\label{MT2-19+}
T_{j,x}(p_i'')<2J_j(K_1+1)/r(j)\tforal x\in X_j.
\eneq

Thus we obtain a unitary $W_0\in B$ such that
\beq\label{nMT-04}
\|{\rm Ad}\, W_0\circ (L_0\oplus {\bar L}_0\oplus  h_1\oplus h_2)(f)-\sum_{i=1}^{N_2} f(y_i)p_i''\|<\ep_2+\ep/16
\eneq
for all $f\in {\cal F}.$

Now define
$$H_1(f)=L_4(f)\oplus {\rm Ad}\, W_0\circ L_0\oplus   h_1(f)\oplus h_2(f)\oplus \sum_{k=1}^{m}f(y_k)(e_i-p_i')\oplus \sum_{j=m+1}^{N_2}f(y_j)(e_i-p_i'')$$
for all $f\in C(X),$
$\psi={\rm Ad}\, W_0\circ {\bar L}_0.$
Let $P_1=(1-Q_3),$ $P_2=W_0^*{\bar L}_0(1_{C(X)})W_0,$
$P_3=H_1(1_{C(X)}).$ $p_j=p_j'-p_j'',$ $j=1,2,...,N_2.$
Then we estimate that, by (\ref{MT2-17}) and (\ref{nMT-04}),
\beq
&&\hspace{-0.6in}\|L_3(f)-(H_1\oplus \psi(f)\oplus\sum_{j=1}^mf(x_j)p_j)\|\\
&<&\|L_3(f)-(L_4(f)\oplus \sum_{i=1}^{N_2}f(y_j)e_j)\| +\\
&&\hspace{-0.6in}\|L_4(f)\oplus \sum_{j=1}^{N_2}f(y_j)e_j
-L_4(f)\oplus {\rm Ad}\, W_0(L_0\oplus {\bar L}_0\oplus h_1\oplus h_2)(f)\oplus
\sum_{j=1}^{m}f(y_j)p_j)\|\\\label{nMT2-30}
&<& \ep_2/2+\ep_2+\ep/16=3\ep_2/2+\ep/16
\eneq
for all $f\in {\cal F}.$
It follows from (\ref{MT2-13}), (\ref{MT2-14--}), (\ref{MT2-14+1}) and (\ref{nMT2-30}) that
\beq
&&\|L(f)-[P_1L(f)P_1\oplus\psi(f)\oplus \sum_{j=1}^mf(x_j)p_j\oplus H_1(f)\|\\
&<&\|L(f)-L'(f)\|+\|L'(f)-(1-Q_3)L'(f)(1-Q_3)\oplus L_3(f)\|\\
&&+
\|(1-Q_3)L'(f)(1-Q_3)\oplus L_3(f)-P_1L(f)P_1\oplus L_3(f)\|\\
&&+\|
P_1L(f)P_1\oplus L_3(f)-[P_1L(f)P_1\oplus H_1(f)\oplus \psi(f)\oplus \sum_{j=1}^m f(x_j)p_j] \|\\
&<& \ep_2/2+\ep_2/2+3\ep_2/2+\ep/16<\ep
\eneq
for all $f\in {\cal F}.$
Define
$$
H_2(f)=\Lambda_4\oplus {\rm Ad}\, W_0\circ L_0\oplus  h_1(f)\oplus h_2(f)\oplus \sum_{k=1}^{N_2}f(y_k)(e_i-p_i')\oplus \sum_{j=m+1}^{N_2}f(y_j)(e_j-p_j'').
$$
Similarly, we also have
\beq
\|{\rm Ad}\, W_1\circ \Lambda(f)-
[P_1({\rm Ad}\, W_1\circ \Lambda(f))P_1\oplus \psi(f)\oplus \sum_{j=1}^m f(x_j)p_j\oplus H_2(f)]\|<\ep
\eneq
for all $f\in {\cal F}.$

Note also that, (by (\ref{MT2-10}) and (\ref{MT2-19+}))
\beq\label{MT2-20}
{\ep_0\Delta(\eta)\over{128N_2}}\ge T_{j,x}(p_i)&=&T_{j,x}(p_j'-p_j'')\ge {\ep_0\Delta(\eta)\over{256N_2}}-2J_j(K_1+1)/r(j)\\
&\ge & 14{\bar K}dN_1N_2(K_1+1)
\eneq
for all $\tau\in T(A)$ and $i=1,2,...,N_2.$
Therefore, by (\ref{MT2-14+1+1}),
\beq\label{MT2-21+}
\tau(p_i)\ge (1-\theta)14d{\bar K}N_1N_2(K_1+1)\tforal \tau\in T(A).
\eneq
Note that, by (\ref{MT2-23}),
\beq\label{NMT2-21}
T_{j,x}(P_2)\le d_jN_1\andeqn \tau(P_1)=\tau(1-Q_3)<\theta
\eneq
for all $x\in X_j,$ $j=1,2,...,N$ and for all $\tau\in T(A).$ 
It follows that
\beq\label{NMT2-22}
\tau(p_i)>K\tau(P_2)+K\tau(P_1) \,\tforal \tau\in T(A),\,\,\,i=1,2,...,m.
\eneq
This gives (\ref{MT2-2+}).
To obtain (\ref{MT2-2}), we note, by (\ref{MT2-19}) and (\ref{MT2-10}) that
\beq\label{NMT2-23}
\tau(P_1)+\tau(P_2)+\sum_{i=1}^m \tau(p_i)&<&
\theta+dN_1+{\ep_0\over{128N_2}}m\\
&<& {\ep_0\Delta(\eta/2)\over{4N_1N_2 {\tilde{K}}}}+ {\ep_0\Delta(\eta/2)\over{256{\tilde K}}}+{\ep_0\over{128}}<\ep_0/2
\eneq
for all $\tau\in T(A).$
We also have that
\beq\label{NMT2-24}
T_{j,x}(P_2)+\sum_{i=1}^mT_{j,x}(p_i)<\ep_0/2 \tforal x\in X_j.
\eneq
Thus
$$
t(P_2+\sum_{i=1}^mp_i)<\ep_0/2
$$
for all $t\in T(B).$ This implies (\ref{MT2-2}).
We write $C=P_3BP_3=\oplus_{j=1}^N C_j,$ where $C_j=M_{r'(j)}(C(X_j)),$ $j=1,2,...,N.$
Finally, from (\ref{MT2-18+}) and (\ref{MT2-18+2}),
\beq\label{NMT2-25}
\mu_{t_{j,x}\circ H_i}(O_r)\ge \Delta(r/3)/2
\eneq
for all $r\ge \eta_0$ and $x\in X_j,$ where $t_{j,x}$ is the normalized trace of $M_{r'(j)}$ evaluated at $x\in X_j,$ $j=1,2,...,N$ and $i=1,2.$
The lemma follows.

\end{proof}

\begin{lem}\label{length}
Let $C$ be a separable unital \CA\, with $T(C)\not=\emptyset,$  let ${\cal U}\subset U_c(K_1(C))$ be a finite
subset, ${\cal F}\subset C$ be a finite subset and let $\lambda>0.$
There exists $\dt>0$ and a finite subset ${\cal G}\subset C$ satisfying the following:
Suppose that $L_1, L_2: C\to A$ (for some unital \CA\ $A$) are two  $\dt$-${\cal G}$-multiplicative \morp s
such that
\beq\label{length-1}
{\rm dist}(\overline{\langle L_1(u)\rangle}, \overline{\langle L_2(u)\rangle})\le \Gamma
\eneq
for all $u\in {\cal U}$ and for some $\Gamma>0.$
There exists a finite subset ${\cal H}\subset A$ and $\sigma>0$ such that,
if $p\in A$ is a projection such that
$$
\|pa-ap\|<\sigma, pap\in_\sigma B \tforal a\in {\cal H},
$$
where $1_B=p$ and $B\subset pAp$ is a unital \SCA,
 and $\Lambda_1, \Lambda_2: C\to B$ are two $2\dt$-${\cal G}$-multiplicative \morp s
 such that
 $$
 \|pL_i(g)p-\Lambda_i(g)\|<\sigma\tforal g\in {\cal G},
 $$
then
$$
{\rm dist}(\overline{\langle \Lambda_1(u)\rangle}, \overline{\langle \Lambda_2(u)\rangle})\le \Gamma+\lambda\,\,\,{\rm (}\,in\,\,\, B{\rm)}
$$
for all $u\in {\cal U}.$

Moreover,
\beq\label{length-2}
|\tau\circ \Lambda_1(f)-\tau\circ \Lambda_2(f)|\le \lambda+\max\{|t\circ L_1(f)-\tau\circ L_2(f)|: f\in {\cal F},\,\,
t\in T(A)\}
\eneq
for all $f\in {\cal F}$ and $\tau\in T(B).$
\end{lem}

\begin{proof}
Let ${\cal U}$ and ${\cal F}$ be fixed.
Then there is an integer $k\ge 1$ such that every element in ${\cal U}$ is represented by a unitary in $M_k(C).$
To simplify notation, replacing $A$ by $M_k(A),$ replacing $B$ by $M_k(B),$ and later replacing $L_i$  by
$L_i\otimes {\rm id}_{M_k},$ $i=1,2,$ without loss of generality, we may assume that ${\cal U}$ is actually in $U(A).$
We choose $\dt$ and ${\cal G}$ such that
for any $2\dt$-${\cal G}$-multiplicative \morp s $L$ from $C,$
$
\langle L(u) \rangle
$
is well defined for all $u\in {\cal U}.$

Now let  $L_1$ and $L_2$ be as described (for the above choice of $\dt$ and ${\cal G}$).
Suppose that $ {\cal U}=\{u_1, u_2,...,u_m\}.$ Then there are $w_1, w_2,...,w_m\in CU(A)$ such that
$$
\|\langle  L_1(u_i)\rangle \langle L_2(u_i^*)\rangle-w_i\|\le \Gamma+\lambda/2.
$$
It is clear that, if ${\cal H}$ is sufficiently large (containing at least
$L_1(u)$ and $L_2(u)$ for all $u\in {\cal U}$ and many  other elements in $U(A)$)
and $\sigma$ is sufficiently small, $\langle pL_ip(u_j)\rangle $ are well defined and
$$
\|\langle pL_ip(u_j)\rangle-\langle  \Lambda_i(u_j)\rangle \|<\lambda/16
$$
($j=1,2,...,m$ and $i=1,2$)
and there are unitaries $v_1, v_2,...,v_m\in CU(B)$ such that
$$
\|pw_ip-v_i\|<\lambda/16,\,\,\,i=1,2,...,m.
$$
It follows that
\beq
{\rm dist}(\overline{\langle \Lambda_1(u)\rangle}, \overline{\langle \Lambda_2(u)\rangle})\le \Gamma+\lambda
\eneq
for all $u\in {\cal U}.$

Similarly, for each $f\in {\cal F},$ there are $x_1(f), x_2(f),..., x_{f(m)}(f)\in A$ such that
\beq
&&\|L_1(f)-\sum_{i=1}^{f(m)} x_i(f)^*x_i(f)\|<\lambda/8\andeqn\\
&&\|L_2(f)-\sum_{i=1}^{f(m)} x_i(f)x_i(f)^*\|<M+\lambda/8
\eneq
where
$M=\max\{|\tau\circ L_1(f)-\tau\circ L_2(f)|: f\in {\cal F},\,\,\tau\in T(A)\}$ (see \cite{CP}).
We compute that, with sufficiently large ${\cal H}$ and small $\sigma,$ there are
$y_1(f), y_2(f),...,y_{f(m)}(f)\in B$ such that
\beq
&&\|\Lambda_1(f)-\sum_{i=1}^{f(m)} y_i(f)^*y_i(f)\|<\lambda/4\andeqn\\
&&\|\Lambda_1(f)-\sum_{i=1}^{f(m)} y_i(f)^*y_i(f)\|<M+\lambda/4
\eneq
for all $f\in {\cal F}.$
This implies that
$$
|\tau\circ \Lambda_1(f)-\tau\circ \Lambda_2(f)|<M+\lambda
$$
for all $f\in {\cal F}$ and for all $\tau\in T(B).$

\end{proof}

\begin{thm}\label{MT1}
Let $X$ be a compact metric space and let $\Delta: (0,1)\to (0,1)$ be a non-decreasing  function with $\lim_{t\to 0}\Delta(t)=0.$
Let $\ep>0$ and ${\cal F}\subset C(X)$ be a finite subset.
Then there exists $\eta>0,$ $ \dt>0,$ a finite subset ${\cal G}\subset C(X),$
a finite subset ${\cal H}\subset C(X)_{s.a.},$  a finite subset
 ${\cal P}\subset \underline{K}(C(X)),$ a finite
subset ${\cal U}\subset U_c(K_1(C(X))),$ $\gamma_1>0$ and $\gamma_2>0$ satisfying the following:
Suppose that $L_1, L_2: C(X)\to A$  are two unital
$\dt$-${\cal G}$-multiplicative \morp s for some unital simple
\CA\, $A$ of tracial rank at most one such that
\beq\label{MT1-1}
[L_1]|_{\cal P}&=&[L_2]|_{\cal P},\\
|\tau\circ L_1(h)-\tau\circ L_2(h)| &<&\gamma_1\tforal h\in {\cal H},\\
{\rm dist}(\langle \overline{L_1(u)}\rangle, \langle \overline{L_2(u)}\rangle)&<&\gamma_2\tforal u\in {\cal U}\\
\mu_{\tau\circ L_i}(O_r) &>&\Delta(r)
\eneq
for all $\tau\in T(A)$ and for all $r\ge \eta.$
Then there exists a unitary $W\in A$ such that
\beq\label{MT1-2}
\|{\rm Ad}\, W\circ L_1(f)-L_2(f)\| &<& \ep\tforal f\in {\cal F}.
\eneq

\end{thm}

\begin{proof}
Fix $\ep>0$ and a finite subset ${\cal F}\subset C(X).$
Let $\eta_1>0$ be as in \ref{NMT1} for $\ep/2$ (in place of $\ep$)
and ${\cal F}.$ Let $\sigma_1=\Delta(\eta_1/3)/3.$
Let $\eta_2>0$ be as in \ref{NMT1} for $\ep/2$ (in place of $\ep$),
${\cal F},$ $\eta_1$ and $\sigma_1.$
Let $\sigma_2=\Delta(\eta_2/3)/3.$ Let $\eta_3>0$ be as in \ref{NMT1} for $\ep/2$ (in place of $\ep$),
${\cal F},$ $\eta_1,$ $\sigma_1,$ $\eta_2$ and $\sigma_2.$
Let $\sigma_3=\Delta(\eta_3/3)/3.$ Let $\eta_4>0$ be as in \ref{NMT1} for $\ep/2$ (in place of $\ep$),
${\cal F},$ $\eta_1,$ $\sigma_1,$ $\eta_2,$ $\sigma_2,$ $\eta_3$ and
$\sigma_3.$ Let $\sigma_4=\Delta(\eta_4/3)/3.$

Let $\gamma_1'>0$ ( in place of $\gamma_1$), $\gamma_2'>0$ (in place of $\gamma_2$), $\dt_1$ (in place of $\dt$), ${\cal G}_1\subset C(X)$
(in place of ${\cal G}$) be a finite subset, ${\cal P}_1\subset \underline{K}(C(X))$ (in place of ${\cal P}$) be a finite subset, ${\cal H}\subset C(X)_{s.a.}$  be a finite subset and ${\cal U}_1\subset U_c(K_1(C(X)))$ (in place of ${\cal U}$) be a finite subset as required by \ref{NMT1} for
$\ep/2,$ ${\cal F},$ $\eta_i$ and $\sigma_i$ ($i=1,2,3,4$).
Let $N\ge 1$ be an integer such that every unitary in ${\cal U}_1$ is in $M_N(U(C(X))).$

Let $\Delta_1=\Delta/2.$
Let $\dt_2>0$ (in place of $\dt$) and ${\cal G}_2\subset C(X)$ (in place of ${\cal G}$) be  required by \ref{keeptrace} for $\Delta_1$ (in place of $\Delta$), ${\cal U}_1$ (in place of ${\cal U}$), $\eta_4/2$ (in place $\eta$), 15/16 (in place of $\lambda_1$)
and $1/32$ (in place of $\lambda_2$).

Let $\dt_3>0$ (in place of $\dt$), ${\cal G}_3\subset C(X)$ (in place of ${\cal G}$) be a finite subset, ${\cal P}_2\subset \underline{K}(C(X))$
(in place of ${\cal P}$) be a finite subset
$\{x_1, x_2,..., x_m\}\subset X,$
${\cal U}_2\subset U_c(K_1(C(X))$ (in place of ${\cal U}$)  and
$K\ge 1$ (in place of $L$) be an integer required by \ref{2Olduniq} for $\ep/2$ (in place of $\ep$),
${\cal F}$ and $\gamma_2'$ (in place of $\lambda$).

Let $\dt_4>0$ (in place of $\dt$), ${\cal G}_4\subset C(X)$ (in place of ${\cal G}$) be a finite subset
 required by Lemma \ref{length} for
$\gamma_2'/8$ (in place of $\lambda$), ${\cal U}_1\cup {\cal U}_2$ (in place of ${\cal U}$) and ${\cal H}$ (in place of ${\cal F}$).

Let $\dt_5=\min\{\ep/4, \dt_i: 1\le i\le 4\}$
${\cal G}_5={\cal F}\cup \cup_{i=1}^4{\cal G}_i,$  ${\cal U}={\cal U}_1\cup {\cal U}_2,$
$\gamma_1=\gamma_1'/8$ and $\gamma_2=\gamma_2'/8.$
Put $\ep_0=\min\{\gamma_1'/8N, \gamma_2'/8N\}$ and $\eta_0=\min\{\eta_i/4: 1\le i\le 4\}.$

Let  $\eta>0,$ $\dt_6>0$ (in place of $\dt$), ${\cal G}_6\subset C(X)$ (in place of ${\cal G}$) be a finite subset
required by \ref{MT2} for $\dt_5$ (in place of $\ep$), $\ep_0,$ $\{x_1, x_2,...,x_m\},$ ${\cal G}_5$ (in place of ${\cal F}$), $\Delta,$ $K$ and $\eta_0.$

Define $\dt=\min\{\dt_6, \dt_5\},$  ${\cal G}={\cal G}_6\cup {\cal G}_5$ and ${\cal P}={\cal P}_1\cup {\cal P}_2.$

Now suppose that $L_1,L_2: C(X)\to A$ are two unital $\dt$-${\cal G}$-multiplicative \morp s, where $A$ is a unital simple \CA\, of tracial rank at most one, which satisfy the assumption
for the above defined $\Delta,$ $\eta,$ ${\cal H},$ ${\cal U},$ $\gamma_1$ and $\gamma_2.$

Let ${\cal H}_1\subset A$ (in place of ${\cal H}$)  be a finite subset  and $\sigma>0$ for $L_1,$ $L_2,$ ${\cal U},$
$\lambda_2'/8$ (in place of $\Gamma$) and
$\min\{\lambda_1'/8,\lambda_2'/8\}$ (in place of $\lambda$) (for $C=C(X)$) be required by \ref{length}.
Let $\ep_{00}=\sigma.$

Let $\dt_7=\min\{\sigma/2, \dt\}.$

By applying \ref{MT2}, for ${\cal H}_1$ (in place of ${\cal H}$), there exist mutually orthogonal projections
$$P_1, P_2, P_3, p_1, p_2,...,p_m\in A$$ with
$P_2, P_3, p_1,p_2,...,p_m\in B,$ $P_1+P_2+P_3+\sum_{i=1}^m p_i=1_A,$
\beq\label{NMT1-10}
\tau(P_3)>1-\ep_0\andeqn K([P_1]+[P_2])\le [p_i],\,\,\,i=1,2,...,m,
\eneq
a unital $\dt_5$-${\cal G}_5$-multiplicative \morp\, $\psi: C(X)\to P_2BP_2$ whose range is contained in a finite dimensional \SCA,
unital $\dt_5$-${\cal G}_5$-multiplicative \morp s $H_1, H_2: C(X)\to P_3BP_3\subset P_3AP_3,$ where
$1_B=1-P_1,$ $B=\oplus_{j=1}^N B_j,$ $B_j=C(X_j, M_{r(j)})$ ($X_j=[0,1],$ or $X_j$  is a point) with
\beq\label{NMT1-11}
[H_1]|_{\cal P}=[H_2]|_{\cal P}=[h_0]|_{\cal P}
\eneq
for some unital \hm\, $h_0: C(X)\to B$ and a unitary $W_0\in A$ such that
\beq\label{NMT1-12}
\|L_1(g)-[P_1L_1(g)P_1\oplus \psi(g)\oplus \sum_{i=1}^m g(x_i)p_i\oplus H_1(g)]\|<\dt_7\andeqn\\\label{NMT1-12+}
\|{\rm Ad}\, W_0\circ L_2(g)-[P_1({\rm Ad}\, W_0\circ L_2(g))P_1\oplus \psi(g)\oplus
\sum_{i=1}^m g(x_i)p_i\oplus H_2(g)]\|<\dt_7
\eneq
for all $g\in {\cal G}_5,$
\beq\label{NMT1-13}
\mu_{t\circ H_i}(O_r)\ge \Delta(r/3)/2
\eneq
for all $r\ge \eta_0$ and for all $t\in T(C),$ where
$C=P_3BP_3,$
\beq\label{NMT1-13+}
T(P_2\oplus \sum_{i=1}^m p_i)<\ep_0
\eneq
for all $T\in T(B),$ and,
\beq\label{NMT1-14}
&&\|P_1a-aP_1\|<\ep_{00}\andeqn\\
 &&(1-P_1)a(1-P_1)\in_{\ep_{00}} B\tforal a\in {\cal H}_1\cup L_1({\cal G}_5\cup L_2({\cal G}_5).
\eneq
Moreover
\beq\label{NMT1-15}
[P_1L_1P_1]|_{\cal P}=[P_1{\rm Ad}\, W_0\circ L_2P_1]|_{\cal P}.
\eneq
Put $\Psi_1=P_1L_1P_1\oplus \psi$ and $\Psi_2=P_1{\rm Ad}\, W_0\circ L_2P_1\oplus \psi.$
By the choice of ${\cal H}_1,$
$$
{\rm dist}(\overline{\langle P_1L_1P_1(u)\rangle}, \overline{\langle P_1{\rm Ad}\, W_0\circ L_2P_1(u)\rangle})<\gamma_2'/4+\gamma_2'/2=\lambda_2
$$
for all $u\in {\cal U}.$  Let $D$ be a finite dimensional \SCA\, of $P_2B P_2$ such that
$\psi(C(X))\subset D.$ Then
\beq\label{NMT1-n1}
\langle \psi(u)\rangle \in CU(P_2BP_2)
\tforal u\in {\cal U}.
\eneq
It follows that
\beq\label{NMT1-16}
{\rm dist}(\overline{\langle \Psi_1(u)\rangle}, \overline{\langle \Psi_2(u)\rangle})<\lambda_2
\eneq
for all $u\in {\cal U}.$ By the choices of $K$ and $\{x_1, x_2,...,x_m\},$ there exists a unitary
$$W_1\in (P_1+P_2+\sum_{i=1}^m p_i)A(P_1+P_2+\sum_{i=1}^m p_i)$$ such that
\beq\label{NMT1-17}
\|W_1^* (\Psi_2(f)\oplus \sum_{i=1}^m f(x_i)p_i)W_1-\Psi_1(f)\oplus \sum_{i=1}^m f(x_i)p_i\|<\ep/2\tforal f\in {\cal F}.
\eneq

Define $\Phi_1: C(X)\to B$ by
$$
\Phi_1(f)=\psi(f)\oplus \sum_{i=1}^m f(x_i)p_i\oplus H_1(f)
$$
for all $f\in C(X)$ and  define
$\Phi_2: C(X)\to B$ by
$$
\Phi_2(f)=\psi(f)\oplus \sum_{i=1}^m f(x_i)p_i\oplus H_2(f)
$$
for all $f\in C(X).$

By the choice of $\dt_4,$ ${\cal G}_4$ and ${\cal H}_1$, and applying \ref{length}, we obtain that
\beq\label{NMT1-18}
{\rm dist}(\overline{\langle \Phi_1(u)\rangle}, \overline{\langle \Phi_2(u)\rangle})<\lambda_2'/8\,\,\,{\rm (in}\,\,\, B{\rm )}
\eneq
for all $u\in {\cal U}$ and
\beq\label{NMT1-19}
|T\circ \Phi_1(g)-T\circ \Phi_2(g)|<\lambda_1'/4
\eneq
for all $g\in {\cal H}$ and for all $T\in T(B).$

Combining with  (\ref{NMT1-13+}), we obtain that
\beq\label{NMT1-20}
|t\circ H_1(f)-t\circ H_2(f)|\le \lambda_1'/4+2\ep_0<\lambda_1'
\eneq
for all $f\in {\cal H}$ and for all $t\in T(C).$
Using the de la Harp-Skandalis determinant, combining
(\ref{NMT1-n1}),
(\ref{NMT1-18}) and (\ref{NMT1-13+}), we compute that
\beq\label{NMT1-21}
{\rm dist}(\overline{\langle H_1(u)\rangle},
\overline{\langle H_2(u)\rangle})<\lambda_2'/4+2N\ep_0<\lambda_2'.
\eneq
for all $u\in {\cal U}.$
Then, by (\ref{NMT1-11}) and by applying \ref{NMT1}, there exists a unitary $W_2\in C$ such that
\beq\label{NMT1-23}
\|{\rm Ad}\, W_2\circ H_2(f)-H_1(f)\|<\ep/2
\eneq
for all $f\in {\cal F}.$
Define $W=W_0(W_1\oplus W_2).$ Then, by (\ref{NMT1-12+}), (\ref{NMT1-17}) and (\ref{NMT1-23}), we finally obtain that
$$
\|{\rm Ad}\, W\circ L_2(f)-L_1(f)\|<\ep
$$
for all $f\in {\cal F}.$

\end{proof}

\begin{df}\label{PMP}
Let $X$ be a compact metric space and $P\in M_r(C(X))$ be a projection, where $r\ge 1$ is an integer. Put $C=PM_r(C(X))P.$
Suppose $\tau\in T(C).$ It is known  that there exists a probability measure $\mu_\tau$ on $X$ such that
$$
\tau(f)=\int_X  t_x(f(x))d\mu_\tau(x)\rforal f\in C
$$
where  $t_x$ is the normalized trace on $P(x)M_rP(x)$ for all $x\in X$ (see 2.17 of \cite{Lncrell}).

Suppose that $Y$ is a finite CW complex, $r\ge 1$ is an integer and $P\in M_r(C(Y))$ is a projection.
Let $X\subset Y$ be a compact subset. Let $\pi: M_r(C(Y))\to M_r(C(X))$ be the quotient map defined
by $\pi(f)=f|_X$ for all $f\in M_r(C(Y)).$
\end{df}

\begin{cor}\label{CPM}
Suppose that $Y$ is a finite CW complex, $r\ge 1$ is an integer and $P\in M_r(C(Y))$ is a non-zero projection.
Define $C=\pi(PM_r(C(Y))P)$ as defined above.
Then Theorem \ref{MT1} holds when $C(X)$ is replaced by $C$ and using the measure defined in \ref{PMP}.
\end{cor}

\begin{proof}
Clearly the corollary holds if $C=M_r(C(X)).$

To prove the general case, we may assume that $Y$ is connected.
Then there is an integer $d\ge 1$ and a projection $Q\in M_d(PM_r(C(Y))P)$ such that
$QM_r(PM_r(C(Y))P)Q\cong M_k(C(Y))$ for some integer $k\ge 1.$  Put $C_1=QM_r(PM_r(C(Y))P)Q.$
Then there exists a projection $Q_1\in M_{k_1}(C_1)$ and a unitary $W\in M_{dk_1}(PM_r(C(Y))P)$ such that
$W^*Q_1W=P.$
 Keep the notation
$\pi$ as in \ref{PMP}.
Note that, for any unital \morp\, $L_i: C\to A,$ we obtain a unital \morp\, $L_i\otimes {\rm id}_{M_d}:
M_r(C)\to M_r(A).$ Put $\psi_{i,1}=(L_1\otimes {\rm id}_{M_r})|_{\pi(C_1)}$ and
$\psi_{i,2}=(\psi_2\otimes {\rm id}_{M_{k_1}})|_{\pi(Q_1)M_{k_1}(C_1)\pi(Q_1)}.$ We see that the corollary follows
by first considering $\psi_{i,1}$ ($i=1,2$) and then $\psi_{i,2}$ ($i=1,2$).
\end{proof}

\begin{df}\label{Full}
{\rm
Let $A$ be a unital \CA\, and let $C$ be another \CA. Let $L: C\to A$ be a positive linear map.
Let $\Theta: C_+\setminus\{0\} \to \N\times \R_+$ be a map.  We write
$\Theta(c)=(N(\Theta(c)), R(\Theta(c)))$ for $c\in C_+\setminus \{0\},$ where $N(\Theta(c))\in \N$ and
$R(\Theta(c))\in R_+.$
Suppose that ${\cal S}\subset C_+$ is a subset. We say the map $L$
is ${\cal S}$-$\Theta$-full, if, for each $s\in {\cal S},$ there are
$x_1, x_2,...,x_{N(\Theta(s))}$ such that $\|x_j\|\le R(\Theta(s)),$ $j=1,2,...,N(\Theta(s))$ and
\beq\label{full-1}
1_A=\sum_{j=1}^{N(\Theta(s))} x_j^*L(s)x_j.
\eneq
}
\end{df}

 The following is known and easy to prove.  Only part (1) is actually used in this paper.  Both hold
 for more general unital simple \CA s.  For example, the class of unital separable simple
 \CA s which satisfy the strict comparison property for positive elements.

\begin{lem}\label{Lfull}
Let $X$ be a compact subset of a finite CW complex $Y,$ let $P\in M_r(C(Y))$ be a projection, where $r\ge 1$ is an integer, and let $\pi: M_r(C(Y))\to M_r(C(X))$ be defined by $\pi(f)=f|_X,$ Put $C=\pi(PM_r(C(Y))P).$

{\rm (1)} Suppose that $\Theta: C_+\setminus \{0\}\to \N\times (\R_+\setminus \{0\})$ is a map. Then there exists
a non-decreasing map $\Delta: (0,1)\to (0,1)$ satisfying the following:
For any $\eta>0,$ there exists a finite subset ${\cal S}\subset C_+\setminus\{0\}$ such that,
if $A$ is a unital separable simple \CA\, with $TR(A)\le 1$ and if $L: C\to A$ is a unital ${\cal S}$-$\Theta$-full positive linear map, then
$$
\mu_{\tau\circ L}(O_r)\ge \Delta(r)\tforal  \tau\in T(A)
$$
for all open balls $O_r$ with radius $r\ge \eta.$

{\rm (2)}  Suppose that $\Delta: (0,1)\to (0,1)$ is a nondecreasing map. Then there exists a map $\Theta: C_+\setminus\{0\}
\to \N\times \R_+\setminus\{0\}$ satisfying the following:
For any finite subset ${\cal S}\subset C_+\setminus\{0\},$ there exists $\eta>0$ such that, if $A$ is a unital
separable simple \CA\, with $TR(A)\le 1$ and
$L: C\to A$ is a unital positive linear map  for which
$$
\mu_{\tau\circ L}(O_r)\ge \Delta(r)\tforal  \tau\in T(A)
$$
for all open balls $O_r$ with radius $r\ge \eta,$ then $L$ is ${\cal S}$-$\Theta$-full.

\end{lem}

\begin{thm}\label{MT4}
Let $C$ be a unital AH-algebra and let $\Theta: C_+\setminus \{0\}\to  \N\times \R_+$ be a map. Let $\ep>0,$  ${\cal F}\subset C$ be a finite subset.  There exists a finite subset ${\cal S}\subset A_+\setminus \{0\},$ $\dt>0,$ $\sigma_1>0,$ $\sigma_2>0,$ a finite subset ${\cal G}\subset C,$ a finite subset ${\cal P}\subset \underline{K}(C),$ a finite subset ${\cal H}\subset A_{s.a.}$ and a finite subset ${\cal U}\subset U_c(K_1(C))$ satisfying the following:
Suppose that $A$ is a unital separable simple \CA\, with $TR(A)\le 1$ and suppose that $\phi, \psi: C\to A$ are two unital
$\dt$-${\cal G}$-multiplicative \morp s such that
$\phi$ and $\psi$ are ${\cal S}$-$\Theta$-full,
\beq\label{MT4-1}
[\phi]|_{\cal P}&=&[\psi]|_{\cal P},\\
|\tau\circ \phi(g)-\tau\circ \psi(g)|&<&\sigma_1\tforal g\in {\cal H},\\
{\rm dist}(\overline{\langle\phi(u)\rangle},
\overline{\langle \psi(u)\rangle})&<&\sigma_2
\tforal u\in {\cal U}.
\eneq
Then there exists a unitary $w\in A$ such that
\beq\label{MT4-2}
\|{\rm Ad}\, w\circ \phi(f)-\psi(f)\|<\ep\tforal f\in {\cal F}.
\eneq
\end{thm}

\begin{proof}
Let $C=\lim_{n\to\infty} (C_n, \phi_n),$ where $C_n=P_nM_{r(n)}(C(Y_n))P_n,$ $X_n$ is a finite CW complex,
$r(n)\ge 1$ is an integer, $P_n\in M_{r(n)}(C(Y_n))$ is a projection and $\phi_n: C_n\to C_{n+1}$ is a unital \hm.
Let $\phi_{n, \infty}: C_n\to C$ be the unital \hm\, induced by the inductive limit system.
Then, for each $n,$ $\phi_{n, \infty}(C_n)\cong \pi_n(P_n)M_{r(n)}(C(X_n))\pi_n(P_n),$ where $X_n\subset Y_n$
is a compact subset and
$\pi_n: M_{r(n)}(C(Y_n))\to M_{r(n)}(C(X_n))$ is defined by
$\pi_n(f)=f|_{X_n}.$  Let $B_n=\pi_n(P_n)M_{r(n)}(C(X_n))\pi_n(P_n),$ $n=1,2,....$  Note that $B_n\subset B_{n+1},$ $n=1,2,....$
We may write $C=\overline{\cup_{n=1}^{\infty} B_n}.$
Let $\ep>0$ and ${\cal F}\subset C$ be a finite subset. Without loss of generality, we may assume
that ${\cal F}\subset  B_n$ for some integer $n\ge 1.$  From this it is clear that we can reduce the general case to the case
that $C=B_n.$ Then the result follows from \ref{CPM} and \ref{Lfull}.
\end{proof}

\begin{cor}\label{CM1}
Let $C$ be a unital AH-algebra and let $\Theta: C_+\setminus\{0\}: \N\times \R_+$ be a map. For any $\ep>0$ and any finite subset ${\cal F}\subset C,$ there exists $\sigma_1>0,$ $\sigma_2>0,$ a finite subset ${\cal S}\subset A_+\setminus\{0\},$
a finite subset ${\cal P}\subset \underline{K}(C),$
a finite subset ${\cal H}\subset A_{s.a.}$ and a finite subset
${\cal U}\subset U_c(K_1(C))$ satisfying the following:

Suppose that $A$ is a unital separable simple \CA\, with $TR(A)\le 1$ and suppose that $\phi,\,\psi: C\to A$ are two unital monomorphisms
which are ${\cal S}$-$\Theta$-full such that
\beq\label{CM1-1}
[\phi]|_{\cal P}&=&[\psi]|_{\cal P}\\
|\tau\circ \phi(g)-\tau\circ \psi(g)|&<&\sigma_1\tforal g\in {\cal H}\tand \tforal \tau\in T(A),\\
{\rm dist}(\phi^{\ddag}({\bar u}), \psi^{\ddag}(u))&<&\sigma_2\tforal u\in {\cal U}.
\eneq
Then there exists a unitary  $w\in A$ such that
\beq\label{CM1-2}
\|{\rm Ad}\, w\circ \phi(f)-\psi(f)\|<\ep\tforal f\in {\cal F}.
\eneq

\end{cor}

%
\begin{thm}\label{MT5}
Let $C$ be a unital AH-algebra and let $A$ be a unital separable simple \CA\, with $TR(A)\le 1.$ Suppose that
$\phi, \psi: C\to A$ are two unital monomorphisms. Then
$\phi$ and $\psi$ are approximately unitarily equivalent if and only if \beq\label{MT5-1}
[\phi]&=&[\psi]\,\,\,{\rm in}\,\,\, KL(C,A)\\
\phi_{\sharp}&=&\psi_{\sharp}\tand\\
\phi^{\ddag}&=&\psi^{\ddag}.
\eneq

\end{thm}

\begin{thm}\label{MT5C}
Let $C$ be a unital AH-algebra and let $A$ be a unital separable simple \CA\, with $TR(A)\le 1.$ Suppose that
$\phi, \psi: C\to A$ are two unital monomorphisms. Then
$\phi$ and $\psi$ are approximately unitarily equivalent if and only if \beq\label{MT5-1}
[\phi]&=&[\psi]\,\,\,{\rm in}\,\,\, KL(C,A)\\
\phi_{\sharp}&=&\psi_{\sharp}\tand\\
\phi^{\dag}&=&\psi^{\dag}.
\eneq

\end{thm}

Note that $[\phi]=[\psi],$ $\phi^{\dag}=\psi^{\dag}$ and
$\phi_{\sharp}=\psi_{\sharp}$ imply that
$\phi^{\ddag}=\psi^{\ddag}.$ Thus Theorem \ref{MT5} follows
from \ref{CM1} immediately.

\section{The range}

\begin{df}\label{Delta1}
Let $X$ be a compact metric space and let $C=PM_n(C(X))P,$ where
$P\in M_n(C(X))$ is a projection and $P(x)>0$ for all $x\in X,$  and let $A$ be a unital separable simple  \CA\, with $T(A)\not=\emptyset.$ Let $\gamma: T(A)\to T_{\rm f}(C)$ be a continuous affine map. For any $\tau\in T(A)$ and any non-empty open set $O\subset X,$ define
$$
\mu_{\gamma(\tau)}(O)=\sup\{\gamma(\tau)(f): 0\le f<1\andeqn {\rm supp}f\subset O\}.
$$
Since $\gamma(T(A))$ is compact, we conclude that
$$
\inf_{\tau\in T(A)} \mu_{\gamma(\tau)}(O)>0
$$
for every non-empty open subset $O\subset X.$

Fix $a\in (0,1).$ There are finitely many points
$x_1, x_2,..,x_m\in X$ such that $\cup_{n=1}^m O(x_i, a/2)\supset X.$
Let $O_a$ be an open ball of $X$ with center at a point $x$ and with radius $a.$ Then $O_a\supset O(x_i, a/2)$ for some $i.$
Define
\beq\label{Delta1-1}
\Delta_1(a)=\min_{\{1\le i\le m\}}\inf_{\tau\in T(A)}\mu_{\gamma(\tau)}(O(x_i, a/2))
\eneq
for all $a\in (0,1).$  It follows that
\beq\label{Delta1-2}
\mu_{\tau}(O_a)\ge \Delta_1(a)\tforal a>0.
\eneq
Note that, if $X$ is infinite, $\lim_{a\to 0}\Delta_1(a)=0.$
\end{df}

\begin{lem}\label{measure}
Let $C$ be as in \ref{Delta1} and $A$ be a unital separable simple \CA\, with $T(A)\not=\emptyset.$
Suppose that $\gamma: T(A)\to T_{\rm f}(C)$ is a continuous affine map.
For any $\eta>0,$ $0<\lambda_1, \lambda_2<1,$ there exists a finite subset ${\cal H}\subset C_{s.a.}$ and $\ep>0$ satisfying the following: for any unital positive linear map
$L: C\to A$ such that
\beq\label{measure-1}
|\tau\circ L(g)-\gamma(\tau)(g)|<\ep\tforal g\in {\cal H},
\eneq
then
\beq\label{measure-2}
\mu_{\tau\circ L}(O_r)\ge \lambda_1\Delta_1(a/2)/2(1+\lambda_2)\tforal a\ge \eta.
\eneq
\end{lem}

The proof of this is almost identical to that of \ref{keeptrace}. We omit it.
%
%

\begin{lem}\label{RanL1}
Let $X$ be a finite CW complex and let $A$ be an infinite dimensional unital simple \CA\, with $TR(A)\le 1.$
Let $C=PM_r(C(X))P$ ($r\ge 1$), where $P\in M_r(C(X))$ is a projection.
Suppose that $e\in A$ is a non-zero projection. Then, there exists a non-zero projection $e_0\le e$ and
a unital monomorphism $h: C\to e_0Ae_0.$
\end{lem}
\begin{proof}
Without loss of generality, we may assume that $X$ is connected. There are mutually orthogonal   and mutually
equivalent  non-zero projections
$e_1, e_2,...,e_r \le eAe.$ Put $e'=\sum_{i=1}^r e_i.$
It is well known that there exists a unital monomorphism $h_0: C(X)\to e_1Ae_1$ (see 9.5 of \cite{Lncltr1}).
This extends a monomorphism $h_1: M_r(C(X))\to e'Ae'\cong M_r(e_1Ae_1).$
Let $e_0=h_1(P).$ Define $h: C\to e_0Ae_0$ by $h=h_1|_C.$
\end{proof}

The following has been used. 

\begin{lem}\label{Cutdown}
Let $A$ be a unital simple separable \CA\, with $T(A)\not=\emptyset$ and let 
$p\in A$ be a non-zero projection.  Then there is an affine homeomorphism 
$\gamma_1: T(A)\to T(pAp)$ such that, for any extremal point $t\in T(A),$
$$
\gamma_1(t)(a)=t(a)/\tau(p)\rforal a\in pAp.
$$
\end{lem}

\begin{proof}
For each $\tau\in T(A),$ define 
$\gamma(\tau)(a)=\tau(a)/\tau(p)$ for all $a\in pAp.$ It is easy to check  that 
$\gamma: T(A)\to T(pAp)$ is bijective and bi-continuous. 
Let $\partial_e(A)$ and $\partial_e(pAp)$ denote the extremal points of $T(A)$ and $T(pAp),$ respectively. 
Then $\gamma$ maps $\partial_e(T(A))$ onto $\partial_e(pAp).$  To see this, 
let $\tau\in \partial_e(T(A)).$ Suppose that there are $\tau_1, \tau_2\in T(pAp)$ such that
$\gamma(\tau)=\af\tau_1+\bt \tau_2,$ where $0\le \af,\, \bt<1$ and $\af+\bt=1.$ 
Since $pAp$ is a unital hereditary \SCA, one can extend $\tau_1$ and $\tau_2$ to two traces on $A$ (see 5.2.7 of \cite{Ped}).
Note also that such extensions are unique, again, because $pAp$ is a hereditary \SCA.
Denote $s'$ the extension of $\af\tau_1+\bt\tau_2$ this way. 
Define $s(a)=s'(a)/s'(1_A).$ Then $s=\tau.$
On the other hand, define  $t_i(a)=\tau_i(a)/\tau_i(1_A)$ for all $a\in A$ and $i=1,2.$ 
Then
$$
s=s'/s'(1_A)=(\af\tau_1+\bt \tau_2)/s'(1_A)=({\af\tau_1(1_A)\over{s'(1_A)}})t_1+({\bt\tau_2(1_A)\over{s'(1_A)}})t_2.
$$
Note $s'(1_A)=\af\tau_1(1_A)+\bt \tau_2(1_A).$
Since $s=\tau$ is  assumed to be in  $\partial_e(T(A)),$ either 
$$
{\af\tau_1(1_A)\over{s'(1_A)}}=0,\,\,\, \text{or} \,\,\, {\bt\tau_2(1_A)\over{s'(1_A)}}=0,
$$
which forces either $\af=0$ or $\bt=0.$  This proves $\gamma$ maps $\partial_e(T(A))$ to 
$\partial_e(T(pAp)).$ Similar argument shows that $\gamma^{-1}$ maps 
$\partial_e(T(pAp))$ to $\partial_e(T(A)).$ This implies that $\gamma$ is a homeomorphism 
from $\partial_e(T(A))$ onto $\partial_e(T(pAp)).$ 

For each $\tau\in T(A),$ by the general theory of Choquet, there is a (positive) probability Borel measure 
$\mu$ on $\partial_e(T(A))$ such that 
$$
f(\tau)=\int_{\partial_e(T(A))} fd\mu\rforal f\in \Aff(T(A)).
$$
Define $\gamma_1(\tau)$ by 
$$
g(\gamma_1(\tau))=\int_{\partial_e(T(A))} g\circ \gamma_1 d\mu\rforal g\in \Aff(T(pAp).
$$
It is easy to check that $\gamma_1$ is an affine map from $T(A)$ to $T(pAp).$ 
Since $\gamma_1|_{\partial_e(T(A))}$ is a homeomorphism, $\gamma_1$ maps $T(A)$ onto $T(pAp)$ as 
an affine homeomorphism.


\end{proof}

\begin{df}\label{KL+}
{\rm
Let $C$ and $A$ be two unital \CA s. Denote by
$KK_e(C,A)^{++}$ the set of those elements $\kappa\in KK(C,A)$ such that
$$
\kappa([1_C])=[1_A]\andeqn \kappa(K_0(C)_+\setminus \{0\})
\subset K_0(A)_+\setminus \{0\}.
$$
Denote by $KL_e(C,A)^{++}$ the set of those elements $\kappa\in
KL(C,A)$ such that
$\kappa([1_C])=[1_A]$ and $\kappa(K_0(C)\setminus \{0\})\subset
K_0(A)_+\setminus \{0\}.$

Now suppose that $T_{\rm f}(C)\not=\emptyset$ and $A$ is a unital simple \CA\, with $T(A)\not=\emptyset.$ Let $\gamma: T(A)\to T_{\rm f}(C)$ be a continuous affine map.
We say $\kappa$ and $\gamma$ are compatible, if,
$\tau\circ \kappa([p])=\gamma(\tau)([p])$ for every projection
$p\in M_{\infty}(C).$ Let $\af: U(M_{\infty}(C))/CU(M_{\infty}(C))\to U(A)/CU(A)$ be a continuous \hm. By (\ref{HS-4}), there is a \hm\, $\af_0:
\Aff(C)/\overline{\rho_C(K_0(C))}\to
\Aff(A)/\overline{\rho_A(K_0(A))}$ induced by $\af$ and there is \hm\, $\af_1: K_1(C)\to K_1(A)$  induced by $\af.$  We say $\af$ and $\kappa$ compatible if $\kappa|_{K_1(C)}=\af_1,$ we say
$\kappa,$ $\gamma$ and $\af$ are compatible if $\kappa$ and $\gamma$ are compatible, $\kappa$ and $\af$ compatible and
the \hm\, induced by $\gamma$ is equal to $\af_0.$
}

\end{df}

\begin{lem}\label{Ran1}
Let $X$ be a finite CW complex, let $n\ge 1$ be an integer,
 let $C=PM_n(C(X))P,$ where $P\in M_n(C(X))$ is a projection, and let $A$ be a unital infinite dimensional separable simple \CA\, of tracial rank at most one. Suppose that $\kappa\in KK_e(C, A)^{++}$ and  $\gamma: T(A)\to T_{{\rm f}}(C)$ is a continuous affine map so that $\kappa$ and $\gamma$ are  compatible.
Let $\sigma>0$ and ${\cal H}\subset C_{s.a.}$ be a finite subset. Then there is a unital monomorphsm $h: C\to A$ such that
\beq
[h]=\kappa\andeqn\\
|\tau\circ h(c)-\gamma(\tau)(c)|<\sigma
\eneq
for all $c\in {\cal H}$ and all $\tau\in T(A).$

\end{lem}

\begin{proof}
To simplify the proof, without loss of generality, we may assume that
$X$ is connected.   We may also assume that $0<\sigma<1/2$ and ${\cal H}$ is in the unit ball of $C.$  There is a unital separable amenable simple \CA\, $B$ with $TR(B)=0$ which satisfies the UCT such that
$$
(K_0(B), K_0(B)_+, [1_B], K_1(B))=(K_0(A), K_0(A)_+, [1_A], K_1(A)).
$$
Let $[\imath]\in KK_e(B,A)^{++}$ be an invertible element which gives the above identity. Therefore there is $\kappa_0\in KK_e(C, B)^{++}$ such that
$$
\kappa=\kappa_0\times [\imath].
$$

Without loss of generality, we may assume that
${\cal H}$ is in the unit ball of $C.$

Let $p\in B$ with $\tau(p)<\sigma/8$ for all $\tau\in T(B).$
It follows from 6.2 of \cite{Lninv} that there is a nonzero projection
$p_0\le p,$ a finite dimensional \SCA\, $B_0\subset (1-p_0)B(1-p_0)$ with
$1_{B_0}=1-p_0,$ a unital \hm\, $h_1: C\to p_0Bp_0$ and a unital \hm\,
$h_2: C\to B_0$ such that
$$
[h_1+h_2]=\kappa_0.
$$
Let $e_0\in A$ be a projection such that $[e_0]=[\imath]([p_0]).$
Put $D=(1-e_0)A(1-e_0).$  Then $D$ is a unital simple \CA\, with $TR(D)\le 1.$
Let $\gamma_1^{-1}: T((1-e_0)A(1-e_0))\to T(A)$ be the affine homeomorphism given by \ref{Cutdown} such that 
$\gamma_1(\tau)(a)=\tau(a)/\tau(1-e_0)$ for all $a\in (1-e_0)A(1-e_0)$ and $\tau\in \partial_e(T(A)).$ 
Note that, for any $\tau\in T(A)$ and $f\in \Aff(T(A))$ with $\|f\|\le 1,$
\beq\label{2015-8/2}
\hspace{-0.2in}|f(\tau)-f(\gamma_1(\tau))|\le \int_{\partial_e(T(A))} |f-f\circ \gamma|d \mu <\sup\{t(e_0)/(1-t(e_0)): t\in T(A)\}<\sigma/7,
\eneq
where $f(\tau)=\int_{\partial_e(T(A))} fd\mu$ and where 
$\mu$ is a probability measure on $\partial_e(T(A)).$ 
Define $\gamma_1': T(D)\to T_{\rm f}(C)$ by $\gamma_1'=\gamma\circ \gamma_1^{-1}.$
It  follows from Lemma 9.5 of \cite{Lncltr1} that there exists a unital monomorphism $h_3: C\to D$ such that
\beq\label{ran1-1}
[h_3]=[h_2]\,\,\,{\rm in}\,\,\, KK(C, A).
\eneq
\beq\label{ran1-2}
|t(h_3(c))-\gamma_1'(t)(c)|<\sigma/8
\eneq
for all $c\in {\cal H}$ and for $t\in T(D).$
It follows from Theorem 5.4 of \cite{Lninv} that there is a unital monomorphism
$j:p_0Bp_0\to e_0Ae_0$
such that $[j]=[\imath].$

 Now define
$h: C\to A$ by
$h(c)=j\circ h_1(c)\oplus h_3(c)$ for all $c\in C.$ One computes that
$$
[h]=[\kappa]\andeqn
|\tau(h(c))-\gamma(\tau)(c)|<\sigma
$$
for all $c\in {\cal H}$ and for all $\tau\in T(A).$

\end{proof}

\begin{lem}\label{Ran2}
Let $C$ be as in \ref{Ran1} and let $A$ be a unital infinite dimensional separable simple \CA\, with $TR(A)\le 1.$  Suppose that  $\kappa\in KK_e(C,A)^{++},$
$\gamma: T(A)\to T_{\rm f}(C)$ is a continuous affine map and $\af:
U(M_{\infty}(C))/CU(M_{\infty}(C))\to U(A)/CU(A)$ such that
$\kappa, \gamma$ and $\af$ are compatible. Then, for any  $\sigma_1>0,$ $1>\sigma_2>0,$ any finite subset  ${\cal H}\subset C_{s.a.}$ and any finite subset ${\cal U}\subset U(M_N(C))$ (for some
integer $N\ge 1$),  there exists a unital monomorphism  $h:C\to A$ such that
\beq\label{ran2-1}
[h]=\kappa,\,\,\, |\tau\circ h(c)-\gamma(\tau)(c)|<\sigma_1\tforal \tau\in T(A)\\
\andeqn {\rm dist}(h^{\ddag}({\bar u}), \af({\bar u}))<\sigma_2\tforal u\in {\cal U}.
\eneq
\end{lem}

\begin{proof}
To simplify the notation, without loss of generality, we may assume that
$X$ is connected.
Furthermore, a standard argument shows that, we can further reduce the general case to the case that $C=C(X).$

We write
$K_1(C)=G_1\oplus Tor(K_1(C)),$ where $Tor(K_1(C))$ is the torsion subgroup of $K_1(C)$ and
$G_1$ is the free part.
Fix a point $\xi_0\in X,$ define
$$
C_0=\{f\in C: f(\xi_0)=0\}.
$$
Then $C_0\subset C$ is an ideal of $C$ and $C/C_0=\C.$
We write
$$
K_0(C)=\Z[1_C]\oplus K_0(C_0).
$$

Let $A_1$ be a unital separable amenable simple \CA\, with UCT and with $TR(A_1)=TR(A)\le 1$ such that
\beq\label{ran2-2}
&&\hspace{-0.6in}(K_0(A_1), K_0(A_1)_+, [1_{A_1}], T(A_1), \rho_{A_1})=(K_0(A), K_0(A)_+, [1_{A}], T(A), \rho_{A})\\\label{ran2-3}
&&\andeqn K_1(A_1)=G_1\oplus Tor(K_1(A)).
\eneq
To simplify notation, we may assume that ${\cal U}={\cal U}_0\cup {\cal U}_1,$
where ${\cal U}_0\subset U_0(M_N(C))$ and ${\cal U}_1\subset U_c(M_N(C))$ are finite subsets and $N\ge 1$ is an integer.
For each $u\in {\cal U}_0,$ write
$u=\prod_{i=1}^{n(u)} \exp(\sqrt{-1} a_i(u)),$ where $a_i(u)\in M_N(C)$ is a selfadjoint elements.
Write
$$
a_i(u)=(a_i^{(k,j)}(u))_{N\times N},\,\,i=1,2,...,n(u).
$$
Write
$$
c_{i,k,j}(u)={a_i^{(k,j)}+(a_i^{(k,j)})^*\over{2}}\andeqn
d_{i,k,j}(u)={a_i^{(k,j)}-(a_i^{(k,j)})^*\over{2i}}.
$$

Put
$$
M={\rm max} \{\|c\|,  \|c_{i,k,j}(u)\|, \|d_{i,k,j}(u)\|: c\in {\cal H}, u\in {\cal U}_0\}.
$$

Choose a non-zero projection $e\in A$  such that
$$
\tau(e)<{\sigma_1\over{8N^2(M+1)\max\{n(u): u\in {\cal U}_0\}}} \,\tforal \tau\in T(A).
$$
Let $e_0\in A_1$ be a projection such that $[e_0]=[e]$ using (\ref{ran2-2})  and let $A_2=(1-e_0)A_1(1-e_0).$
In what follows, we use the identification (\ref{ran2-2}).
Define
$\theta_1\in Hom(K_i(C), K_i(A_2))$
as follows:
On $K_0(C),$
define $\theta_1(m[1_C])=m[1-e_0]$ for all $m\in \Z,$
$\theta_1|_{K_0(C_0)}=\kappa|_{K_0(C_0)},$ on $K_1(C),$ define
\beq\label{ran2-5}
 \theta_1|_{Tor(K_1(C))}=\kappa|_{Tor(K_1(C))},\,\,\,
\theta_1|_{G_1}={\rm id}|_{G_1}.
\eneq
By the Universal Coefficient Theorem, there exists an element $\theta_1\in KL(C, A_2)$
which gives the above \hm s.
Let $\theta_2\in KL(A_1, A)$ which gives the identification (\ref{ran2-2}) and
$\theta_2|_{Tor(K_1(A_1))}={\rm id}|_{Tor(K_1(A_1))}$ and $\theta_2|_{G_1}=\kappa|_{G_1}.$
Let $\beta=\kappa-\theta_2\circ \theta_1.$ We compute that
\beq\label{ran2-9}
\beta([1_C])=[e],\,\,\,\beta|_{K_0(C_0)}=0,\\
\andeqn\beta|_{K_1(C)}=0.
\eneq
Thus $\beta\in KK_e(C, eAe)^{++}.$
It follows from \ref{Ran1} that there is a unital monomorphism $\phi_0: C\to eAe$ such that
$[\phi_0]=\beta.$

Choose
$$
{\cal H}_1={\cal H}\cup\{c_{i,k,j}(u), d_{i,k,j}(u): 1\le k, j\le N,1\le i\le n(u), u\in {\cal U}_0\}.
$$
It follows from \ref{Ran1} that there exists a unital monomorphism
$\phi_1: C\to A_2$ such that
\beq\label{Ran2-12}
[\phi_1]=\theta_1\andeqn \\
|\tau\circ \phi_1(f)-\gamma(\tau)(f)|<{\sigma_1\over{8N^2}}
\eneq
for all $f\in {\cal H}_1$ for all $\tau\in T(A).$
Note that, for $u\in {\cal U}_0,$
\beq\label{Ran2-13}
\Delta(u)=\overline{\sum_{i=1}^{n(u)}\widehat{a_j(u)}},
\eneq
where $\hat{a}(\tau)=\tau(a)$ for all $a\in A_{s.a.}.$
Since $\af$ and $\gamma$ are compatible, we then compute  that
\beq\label{ran2-14}
{\rm dist}(\phi_1^{\ddag}({\overline{u}}), \af({\overline{u}}))<\sigma_2/8
\eneq
for all $u\in {\cal U}_0.$
Denote by $U_c(G_1)$ the image of $G_1.$
Define $\chi: U_c(G_1)\to {\rm Aff}(T(A))/\overline{\rho_A(K_0(A))}$ by
\beq\label{Ran2-15}
\chi=\af|_{U_c(G_1))}-\phi_0^{\ddag}|_{U_c(G_1)}
-\phi_1^{\ddag}|_{U_c(G_1)}.
\eneq
Note that $U_c(G_1)\cong G_1.$ We identify $U_c(G_1)$ with the corresponding part in $U_c(K_1(A_2)).$
By defining $\chi$ on
$Tor(K_1(A_2))$ to be zero, we obtain a \hm\,
$\chi: U_c(K_1(A_2))\to {\rm Aff}(T(A))/\overline{\rho_A(K_0(A))}.$
It follows from Theorem 8.6 of \cite{Lninv}  that there exists a unital \hm\,
$h_1: A_2\to (1-e)A(1-e)$ such that
\beq\label{Ran2-16}
[h_1]=\theta_2,\,\,\, (h_1)_{\sharp}={\rm id}_{T(A)}\\
\andeqn
h_1^{\ddag}|_{U_c(A_2)}=\chi+\theta_2|_{K_1(A_2)},
\eneq
where we identify $K_1(A_2)$ with $U_c(A_2)=U_c(G_1)\oplus Tor(K_1(A))$ and
$U_c(K_1(A))$ with $K_1(A).$  We also identify ${\Aff(T(A_2))/\overline{\rho_A(K_0(A_2))}}$
with ${\Aff(T(A))/\overline{\rho_A(K_0(A))}}.$
Note that (by (\ref{Ran2-16}),
$$
h_1^{\ddag}|_{\Aff(T(A))/\overline{\rho_A(K_0(A))}}={\rm id}_{\Aff(T(A))/\overline{\rho_A(K_0(A))}}.
$$
Now define
\beq\label{Ran2-17}
h(f)=\phi_0(f)\oplus h_1\circ \phi_1(f)\tforal f\in C.
\eneq
It follows that
\beq\label{Ran2-18}
[h]&=&\kappa, \\
|\tau\circ h(f)-\gamma(\tau)(f)|&<&\sigma_1\tforal f\in {\cal H}\andeqn\\
{\rm dist}(h^{\ddag}(\bar{u}),\af(\bar{u}))&<&\sigma_2\tforal u\in {\cal U}.
\eneq

\end{proof}

\begin{lem}\label{Ran3}
Let $X$ be a compact subset of a finite CW complex $Y. $
Then there exists a sequence of finite CW complex $Y_n\supset Y_{n+1}$ each of which is
a compact subset of $Y$  and there exists a \morp\, $\phi_n: C(X)\to C(Y_n)$ such that
\beq\label{ran3-1}
\pi_n\circ \phi_n={\rm id}_{C(X)},\,\,\,n=1,2,...\andeqn\\
\lim_{n\to\infty}\|\phi_n(f)\phi_n(g)-\phi_n(fg)\|=0
\eneq
for all $f, g\in C(X),$ where $\pi_n: C(Y_n)\to C(X)$ is the quotient map.
\end{lem}

\begin{proof}
Let $d_n\searrow 0$ be a decreasing sequence of positive numbers.  There are finitely many open balls of $Y$ with center in $X$ and radius $d_n$
covers $X.$  Let $Z_n$ be the union of  closure of these balls. Then $Z_n$ is a compact subset
of $Y$ which is homeomorphic to a finite CW complex. We may assume that $Z_n\supset Z_{n+1}.$
Then (by, for example, The Effros-Choi Theorem),  there exists, for each $n,$ a
\morp\, $\psi_n: C(X)\to C(Z_n)$ such that
\beq\label{ran3-2}
\pi_n\circ \psi_n={\rm id}_{C(X)},
\eneq
where $\pi_n(f)=f|_X$ for $f\in C(Z_n),$ $n=1,2,....$

Let $\{{\cal F}_m\}\subset C(X)$ be a sequence  of increasing finite subsets of  the unit ball of $C(X)$ so that its union
is dense in the unit ball of $C(X).$ Choose $Y_1=Z_1$ and $\phi_1=\psi_1.$
Let ${\cal G}_1={\cal F}_1\cup \{fg: f, g\in {\cal F}_1\}.$
Choose $d_{n_2}$ such that
 \beq\label{ran3-3}
 |\psi_1(f)(x)-\psi_1(f)(x')|<1/4\,\tforal f\in {\cal G}_1,
 \eneq
provided that ${\rm dist}(x, x')<d_{n_2}$ for all $x, x'\in Z_1.$
By (\ref{ran3-2}),
\beq\label{ran3-4}
\psi_1(fg)(x)-\psi_1(f)(x)\psi_1(g)(x)=0
\eneq
for all $x\in X.$  Now for any $z\in X_{n_2},$ there exists $x\in X$ such that
${\rm dist}(x, z)<d_{n_1}.$ Therefore, by (\ref{ran3-4}) and (\ref{ran3-3}),
\beq\label{ran3-5}
&&\hspace{-1.5in}|\psi_1(fg)(z)-\psi_1(f)(z)\psi_1(g)(z)| \le
|\psi_1(fg)(z)-\psi_1(fg)(x)|\\
&&\hspace{-0.4in}+|\psi_1(fg)(x)-\psi_1(f)(x)\psi_1(g)(x)|\\
&&+|\psi_1(f)(x)\psi_1(g)(x)-\psi_1(f)(z)\psi_1(g)(z)|<3/4
\eneq
for all $f, g\in {\cal F}_1.$ Choose $Y_2=Z_{n_2}.$ Define $h_1: C(Z_1)\to C(Z_{n_2})$ defined by $h_1(f)=f|_{Z_{n_2}}$ for all $f\in C(Z_1).$ Define
$\phi_2: C(X)\to C(Y_2)$ by defining
$$
\phi_2(f)=h_1\circ \psi_1.
$$
Thus, by (\ref{ran3-5}),
\beq\label{ran3-6}
\|\phi_2(fg)-\phi_2(f)\phi_2(g)\|<3/4
\eneq
for all $f, g\in {\cal F}_1.$
Note that
\beq\label{ran3-7-1}
\pi_{n_2}\circ \phi_2={\rm id}_{C(X)}.
\eneq

Let ${\cal G}_2={\cal G}_1\cup {\cal F}_2\cup \{fg: f,g \in {\cal F}_2\}.$
Choose $d_{n_3}$ such that
 \beq\label{ran3-7}
 |\phi_2(f)(x)-\phi_2(f)(x')|<1/4^2\,\tforal f\in {\cal G}_2,
 \eneq
provided that ${\rm dist}(x, x')
<d_{n_3}$ for all $x, x'\in Y_2.$
By (\ref{ran3-2}), for any $x\in X,$
\beq\label{ran3-8}
\phi_2(fg)(x)=\phi_2(f)(x)\phi_2(g)(x)
\eneq
Then, for any $z\in Z_{n_3},$ there exists $x\in X$ such that ${\rm dist}(x, z)<d_{n_3}. $ Thus, by (\ref{ran3-8}) and (\ref{ran3-7}),
\beq\label{ran3-9}
|\phi_2(fg)(z)-\phi_2(f)(z)\phi_(g)(z)|<3/4^2\tforal z\in Z_{n_3}.
\eneq
Let $h_2: C(Y_2)\to C(Z_{n_3})$ be defined by $h_2(f)=f|_{Z_{n_3}}$ for all $f\in C(Y_2).$ Put $Y_3=Z_{n_3}.$ Define $\phi_3: C(X)\to C(Y_{n_3})$ by
$\phi_3(f)=h_2\circ \phi_2.$ Then
\beq\label{ran3-10}
\pi_{n_3}\circ \phi_3={\rm id}_{C(X)}.
\eneq
By (\ref{ran3-9}), we have
that
\beq\label{ran3-11}
\|\phi_3(fg)-\phi_3(f)\phi_3(g)\|<3/4^2
\eneq
for all $f, g\in {\cal F}_2.$ In this fashion, we obtain a sequence of
\morp s $\phi_k: C(X)\to C(Y_k),$ where $Y_k=Z_{n_k},$ such that
\beq\label{ran3-12}
\pi_{n_k}\circ \phi_k={\rm id}_{C(X)}\andeqn\\
\|\phi_k(fg)-\phi_k(f)\phi_k(g)\|<3/4^{k-1},
\eneq
for all $f, g\in {\cal F}_k,$ $k=1,2,....$
It follows that, for any $f,g\in C(X),$
\beq\label{ran3-13}
\lim_{k\to\infty}\|\phi_k(fg)-\phi_k(f)\phi_k(g)\|=0.
\eneq
\end{proof}

We have the following corollary.

\begin{cor}\label{Cran3}
Let $Y$ be a finite CW complex and $P\in M_r(C(Y))$ be a non-zero projection for some integer $r\ge 1.$
Let $X$ be a compact metric space of $Y$ and let $C=\pi(PM_r(C(Y))P),$ where
$\pi:  M_r(C(Y))\to M_r(C(X))$ is the quotient map defined by $\pi(f)=f|_X.$
Then there exists a sequence of finite CW complex $Y\supset Y_n\supset Y_{n+1}$ each of which is
a compact subset of $Y$  and there exists a \morp\, $\phi_n: C\to P_n(C(Y_n))P_n$ such that
\beq\label{Cran3-1}
\pi_n\circ \phi_n={\rm id}_{C},\,\,\,n=1,2,...\andeqn\\
\lim_{n\to\infty}\|\phi_n(f)\phi_n(g)-\phi_n(fg)\|=0
\eneq
for all $f, g\in C(X),$ where $P_n=P|_{Y_n}$ and $\pi_n: C(Y_n)\to C(X)$ is the quotient map defined by
$\pi_n(f)=f|_X$ for all $f\in C(Y_n).$
\end{cor}

\begin{lem}\label{Ran4}
Let $Y$ be a finite CW complex and $P\in M_r(C(Y))$ be a non-zero projection for some integer $r\ge 1.$
Let $X$ be a compact metric space of $Y$ and let $C=\pi(PM_r(C(Y))P),$ where
$\pi:  M_r(C(Y))\to M_r(C(X))$ is the quotient map defined by $\pi(f)=f|_X.$
Suppose that $A$ is  a unital infinite dimensional separable simple \CA\, with $TR(A)\le 1.$ For any $\kappa\in KL_e(C, A)^{++},$ any affine continuous map
$\gamma: T(A)\to T_{\rm f}(C)$ and any continuous \hm\, $\af: U(M_{\infty}(C))/CU(M_{\infty}(C))\to U(A)/CU(A)$ such that
$\kappa,$ $\gamma$ and $\af$ are compatible, then there is a unital
monomorphism $h: C\to A$ such that
\beq\label{ran4-1}
[h]=\kappa,\,\,\, h_{\sharp}=\gamma\andeqn h^{\ddag}=\af.
\eneq
\end{lem}

\begin{proof}
Let $Y_n,$ $P_n,$ $\pi_n$ and $\phi_n$ be as given by \ref{Cran3}.
Let $B_n=P_nM_r(C(Y_n))P_n.$ Let $\{q_n'\}$ be a sequence of non-zero projections such that
$\tau(q_n')<1/n$ for all $\tau\in T(A),$ $n=1,2,....$  Since $q_n'Aq_n'$ is a unital infinite dimensional simple \CA\, of tracial rank at most one, by \ref{RanL1},  there exists a non-zero projection
$q_n\le q_n'$ and  a unital monomorphism $\phi_{0,n}: B_n\to q_nAq_n,$ $n=1,2,...$

Define $\gamma_n: T(A)\to T_{\rm f}(B_n)$ by
\beq\label{ran4-4}
\gamma_1(\tau)(b)=
\tau(1-q_n)\gamma(\tau)(\pi_n(b))+
\tau\circ \phi_{0,n}(b)
\eneq
for all $b\in B_n$ and
for all $\tau\in T(A).$

Define $s_n: B_n\to B_{n+1}$ by $s_n(f)=f|_{Y_n}$ for all $f\in B_n,$ $n=1,2,....$

Let $\kappa_n=\kappa\circ [\pi_n]$ be
in $Hom_{\Lambda}(\underline{K}(B_n), \underline{K}(A))$ and let
$\af_n: U(M_{\infty}(B_n))/CU(M_{\infty}(C))\to
U(A)/CU(A)$ defined by $\af_n=\af\circ \pi_n^{\ddag}.$

Note that $K_i(B_n)$ is finitely generated ($i=0,1$).
Let $\ep_n>0,$ ${\cal F}_n\subset B_n$ be a finite subset and $Q_n\subset \underline{K}(B_n)$ be a finite subset
such that $(\ep_n, {\cal F}_n, {\cal P}_n)$ is a $\underline{K}$-triple and $(\ep_n, {\cal F}_n)$ is $KK$-pair for
$B_n,$ $n=1,2,....$
Put ${\cal Q}_n=[\pi_n]({\cal P}_n),$ $n=1,2,....$ We may assume that $[s_n]({\cal P}_n)\subset {\cal P}_{n+1}$ and
$\cup_{n=1}^{\infty}{\cal Q}_n=\underline{K}(C).$
Let ${\cal G}_n\subset C$ be a finite subset and let $\dt_n>0$ be such that
$(\dt_n, {\cal G}_n, {\cal Q}_n)$ is a $\underline{K}$-triple.

Choose, for each $n,$  a finite subset ${\cal F}_n\subset B_n$ such that
$s_n({\cal F}_n)\subset {\cal F}_{n+1}$ and $\cup_{n=1}^{\infty}\pi_n({\cal F}_n)$ is dense in
$C.$ Choose, for each $n,$  a finite subset ${\cal H}_n\subset (B_n)_{s.a}$ such that
$s_n({\cal H}_n)\subset {\cal H}_{n+1}$ and $\cup_{n=1}^{\infty}\pi_n({\cal H}_n)$ is dense in
$C_{s.a.}.$ Choose, for each $n,$ a finite subset ${\cal U}_n\subset U(M_{N(n)}(B_n))$ (for some integer
$N(n)$) such that
$s_n({\cal U}_n)\subset {\cal U}_{n+1}$ and $\cup_{n=1}^{\infty}\pi_n({\cal U}_n)$ is dense in $U(M_{\infty}(C)).$

It follows from \ref{Ran2} there is, for each $n,$ a unital monomorphism
$h_n: B_n\to A$ such that
\beq\label{ran4-5}
[h_n]&=&\kappa_n\\\label{ran4-6}
|\tau\circ h_n(f)-\gamma_n(\tau)(f)|&<&1/2^n\tforal f\in {\cal H}_n\andeqn \tforal \tau\in T(A)\\\label{ran4-7}
\andeqn {\rm dist}(h_n^{\ddag}({\bar u}), \af_n({\bar u}))&<&1/2^n\tforal u\in {\cal U}_n,
\eneq
$n=1,2,....$  Define $L_n=h_n\circ \phi_n.$
Note that
\beq\label{ran4-8}
\pi_n\circ \phi_n&=&{\rm id}_{C}\andeqn\\
\lim_{n\to\infty}\|\phi_n(fg)-\phi_n(f)\phi_n(g)\|&=&0\tforal f, g\in C.
\eneq
Thus, without loss of generality, we may assume that
$\phi_n$ is $(\dt_{n-1}, {\cal G}_{n-1})$-multiplicative.
It then follows that from (\ref{ran4-8}) and (\ref{ran4-5}) that
\beq\label{ran4-9}
[h_n\circ \phi_n]|_{{\cal Q}_n}=(\kappa_n)|_{{\cal Q}_n},\,\,\,n=1,2,....
\eneq
By (\ref{ran4-9}), (\ref{ran4-6}), (\ref{ran4-7}), combining \ref{measure}  and applying \ref{CPM},
we obtain a subsequence $h_{n_k}\circ \phi_{n_k}: C\to A$ and a sequence
of unitaries $\{u_k\}\subset A$ such that
\beq\label{ran4-10}
\|{\rm Ad}\, u_k\circ h_{n_{k+1}}\circ \phi_{n_{k+1}}(f)-{\rm Ad}\, u_{k-1}\circ h_{n_k}\circ \phi_{n_k}(f)\|<1/2^{k+1}
\eneq
for all $f\in {\cal F}_k.$
It follows that
$\{{\rm Ad}\, u_{k-1}\circ h_{n_k}\circ \phi_{n_k}(f)\}$ is Cauchy for all $f\in C.$
Define
$h: C\to A$ by
$$
h(f)=\lim_{k\to\infty}{\rm Ad}\, u_{k-1}\circ h_{n_k}\circ \phi_{n_k}(f)\tforal f\in C.
$$
It is ready to check that $h$ satisfy all requirements of the lemma.

\end{proof}

\begin{thm}\label{Ran5}
Let $C$ be a unital AH-algebra and let $A$ be a unital infinite dimensional separable simple \CA\, with $TR(A)\le 1.$ For any $\kappa\in KL_e(C, A)^{++},$ any affine continuous map
$\gamma: T(A)\to T_{\rm f}(C)$ and any continuous \hm\, $\af: U(M_{\infty}(C))/CU(M_{\infty}(C))\to U(A)/CU(A)$ such that
$\kappa,$ $\gamma$ and $\af$ are compatible, then there is a unital
monomorphism  $h: C\to A$ such that
\beq\label{ran5-1}
[h]=\kappa,\,\,\, h_{\sharp}=\gamma\andeqn h^{\ddag}=\af.
\eneq
\end{thm}

\begin{proof}
Write $C=\lim_{n\to\infty}(B_n, \psi_n),$ where $B_n=P_nM_{r(n)}(C(Y_n))P_n,$ $Y_n$ is a finite CW complex,
$P_n\in M_{r(n)}(C(Y_n))$ is a projection and $\psi_n: B_n\to B_{n+1}$ is a unital \hm.
Denote by $\psi_{n, \infty}: B_n\to C$ the unital \hm\, induced by the inductive limit system.
Then $\phi_{n, \infty}(B_n)\cong Q_nM_{r(n)}(C(X_n))Q_n,$ where $X_n\subset Y_n$ is a compact subset,
$Q_n=\pi_n(P_n)$ and where $\pi_n: M_{r(n)}(C(Y_n))\to M_{r(n)}(C(X_n))$ is the quotient map defined
by $\pi_n(f)=f|_{X_n}.$ Put $C_n=\phi_{n, \infty}(B_n).$ We will identify $C_n$ with $Q_nM_{r(n)}(C(X_n))Q_n$ and write
$C=\overline{\cup_{n=1}^{\infty}  C_n},$ where $C_n=Q_nM_{r(n)}(C(X_n)Q_n,$

Denote by $\imath_n: C_n\to C_{n+1}$ and $\imath_{n, \infty}: C_n\to C$ be the embedding, respectively.
Let $\kappa_n=\kappa\circ [\imath_{n,\infty}]$ be
in $Hom_{\Lambda}(\underline{K}(C_n), \underline{K}(A))$ and
let $\af_n: U(M_{\infty}(C_n))/CU(M_{\infty}(C))\to
U(A)/CU(A)$ be defined by $\af_n=\af\circ \imath_{n,\infty}^{\ddag}.$
Let $(\imath_{n, \infty})_{\sharp}: T_{\rm f}(C)\to T_{\rm f}(C_n)$ be induced by $\imath_{n,\infty}$ and define
$\gamma_n: T(A)\to T_f(C_n)$ by $(\imath_{n, \infty})_{\sharp}\circ \gamma.$
Put ${\cal Q}_n=[\imath_{n,\infty}]({\cal P}_n),$ $n=1,2,....$ We may assume that
$\cup_{n=1}^{\infty}{\cal Q}_n=\underline{K}(C).$
Choose, for each $n,$  a finite subset ${\cal F}_n\subset C_n$ such that
${\cal F}_n\subset {\cal F}_{n+1}$ and $\cup_{n=1}^{\infty}{\cal F}_n$ is dense in
$C.$
It follows from \ref{Ran4} there is, for each $n,$ a unital monomorphism
$\phi_n: C_n\to A$ such that
\beq\label{ran5-2}
[\phi_n]=\kappa_n,\,\,\,
(\phi_n)_{\sharp} =\gamma_n
\andeqn \phi_n^{\ddag}=\af_n
\eneq
$n=1,2,....$  By applying \ref{MT5}, for each $n,$ there exists a unitary $u_n\in A$ (with $u_0=1$) such that
\beq\label{ran5-3}
\|{\rm Ad}\, u_n\circ  \phi_{n+1}\circ \imath_n(f)-{\rm Ad}\, u_{n-1}\circ \phi_n(f)\|<1/2^n\tforal f\in {\cal F}_n,
\eneq
$n=1,2,....$
We obtain a unital monomorphism $h: C\to A$ such that
\beq\label{ran5-4}
h(f)=\lim_{n\to\infty}{\rm Ad}\, u_n\circ  \phi_{n+1}\circ \imath_n(f)\tforal f\in C.
\eneq
One checks that $h$ meets all requirements of the theorem.

\end{proof}

\begin{cor}
Let $C$ be a unital AH-algebra and let $A$ be a unital infinite dimensional separable simple \CA\, with $TR(A)\le 1.$ For any $\kappa\in KL_e(C, A)^{++},$ any affine continuous map
$\gamma: T(A)\to T_{\rm f}(C)$ and any continuous \hm\, $\af: K_1(C)\to \Aff(T(A))/\overline{K_0(A)}$ such that
$\kappa,$ $\gamma$  are compatible, then there is a unital
\hm\, $h: C\to A$ such that
\beq\label{ran6-1}
[h]=\kappa,\,\,\, h_{\sharp}=\gamma\andeqn h^{\dag}=\af.
\eneq
\end{cor}

\end{document}